\documentclass[reqno]{amsart}
\usepackage{amssymb}
\usepackage{amsmath}
\usepackage{amsthm}
\usepackage[usenames]{color}
\usepackage{graphicx}
\usepackage{cite}
\usepackage{bbm}

\usepackage[colorlinks,linkcolor=blue,anchorcolor=red,citecolor=blue]{hyperref}
\usepackage[margin=1in]{geometry} 
\usepackage{marginnote}

\usepackage{color}

\newtheorem{lemma}{Lemma}[section]
\newtheorem{theorem}{Theorem}[section]

\newtheorem{proposition}{Proposition}[section]
\newtheorem{remark}{Remark}[section]
\newtheorem{corollary}{Corollary}[section]

\numberwithin{equation}{section}

\newcommand{\dis}{\displaystyle}

\newcommand{\R}{\mathbb{R}}

\renewcommand{\S}{\mathbb{S}}

\newcommand{\CD}{\mathcal{D}}

\newcommand{\CL}{\mathcal{L}}

\newcommand{\ep}{\epsilon}

\newcommand{\na}{\nabla}

\newcommand{\al}{\alpha}
\newcommand{\be}{\beta}
\newcommand{\ga}{\gamma}

\newcommand{\la}{\lambda}
\newcommand{\de}{\delta}
\newcommand{\si}{\sigma}
\newcommand{\pa}{\partial}
\newcommand{\ka}{\kappa}

\newcommand{\Ga}{\Gamma}

\newcommand{\vertiii}[1]{{\left\vert\kern-0.25ex\left\vert\kern-0.25ex\left\vert #1
		\right\vert\kern-0.25ex\right\vert\kern-0.25ex\right\vert}}

\makeatletter
\@namedef{subjclassname@2020}{%
	\textup{2020} Mathematics Subject Classification}
\makeatother

\begin{document}

\title[Polynomial tail Boltzmann solutions near local Maxwellians]{Polynomial tail solutions of the non-cutoff Boltzmann equation near local Maxwellians}

\author[R.-J. Duan]{Renjun Duan}
\address[RJD]{Department of Mathematics, The Chinese University of Hong Kong, Shatin, Hong Kong, P.R.~China}
\email{rjduan@math.cuhk.edu.hk}
	
\author[Z.-G. Li]{Zongguang Li}
\address[ZGL]{Department of Mathematics, The Chinese University of Hong Kong, Shatin, Hong Kong, P.R.~China}
\email{zgli@math.cuhk.edu.hk}

\begin{abstract}
This paper aims to incorporate the Caflisch's decomposition into the macro-micro decomposition in Boltzmann theory for allowing the microscopic component to exhibit only the polynomial tail in large velocities. In particular, we treat the Cauchy problem on the non-cutoff Boltzmann equation under the compressible Euler scaling in case of three-dimensional whole space. Up to a finite time we construct the Boltzmann solution around a local Maxwellian corresponding to small-amplitude classical solutions of the full compressible Euler system around constant states. We design a new energy functional which can capture the convergence rate in the small Knudsen number $\varepsilon$ and allow the microscopic part of solutions to decay polynomially in large velocities. Moreover, the energy norm of perturbations can be of the order $\varepsilon^{1/2}$ which the usual method of Hilbert expansion fails to obtain. As a byproduct of the proof, our estimates immediately yield a global-in-time existence result when  the Euler solutions are taken to be constant states.
\end{abstract}

	\date{\today}
	
	

	\maketitle
	\thispagestyle{empty}
	
	\tableofcontents
	\section{Introduction}
The Cauchy problem on the spatially inhomogeneous non-cutoff Boltzmann equation in $\R^3$ reads
	\begin{eqnarray}\label{BE}
		&\dis \pa_tF+v\cdot \na_x F=\frac{1}{\varepsilon}Q(F,F),   \quad &\dis F(0,x,v)=F_0(x,v),
	\end{eqnarray}
	where the unknown $F(t,x,v)\geq0$ stands for the density distribution function of gas particles with velocity $v=(v_1,v_2,v_3)\in \R^3$ at position $x=(x_1,x_2,x_3)\in \R^3$ and time $t> 0$, and initial data $F_0(x,v)\geq 0$ is given. The parameter $\varepsilon>0$ denotes the Knudsen number which is proportional to the mean free path.  The bilinear Boltzmann collision operator $Q(\cdot,\cdot)$ acting only on velocity variable is defined by
	\begin{align}\label{defQ}
		Q(G,F)(v)=&\int_{\R^3}\int_{\S^2}B(v-u,\si)\left[ G(u')F(v')-G(u)F(v)\right]d\sigma du, 
	\end{align}
	where the velocity pair $(v',u')$ is given in terms of the $\si$-representation:
$$
		v'=\frac{1}{2}(v+u)+\frac{1}{2}|v-u|\sigma,\quad
		u'=\frac{1}{2}(v+u)-\frac{1}{2}|v-u|\sigma,\quad \sigma\in \mathbb{S}^2.
$$
Throughout the paper, we consider the Boltzmann collision kernel $B(v-u,\si)$ in the non-cutoff case:	
\begin{equation}\label{b.kernel}
		B(v-u,\sigma)=\vert v-u\vert^\gamma b(\cos\theta),
	\end{equation}
	with
		$\cos \theta=\frac{v-u}{\vert v-u\vert}\cdot \sigma$ and $0<\theta \leq\pi/2$, satisfying
	\begin{equation}\label{b.kernel.ass}
		-3<\gamma\leq1;\quad \sin \theta b(\cos\theta)\sim \theta^{-1-2s}\quad \text{as }\ \theta\to 0,
		\  0<s<1.
	\end{equation}
	
At a formal level, as $\varepsilon\to 0$ the solution of Boltzmann equation \eqref{BE} converges to a local Maxwellian $M_{[\rho,u,\theta](t,x)}(v)$ with fluid quantities $[\rho,u,\theta](t,x)$ satisfying the full compressible Euler system, cf.~\cite{SR}. For the rigorous justification, we refer to \cite{Nish, UA} for analytic solutions, \cite{Caf} for Sobolev-in-$x$ solutions, and \cite{Guo-Jang-Jiang-2010} for $L^2\cap L^\infty$ solutions. In case of spatially one-dimensional domain, the Euler system admits particular solutions with wave patterns such as shock wave, rarefaction wave or contact discontinuity, and there have existed a lot of studies of constructing the Boltzmann solutions around those wave pattern profiles or their superposition in the compressible fluid dynamic limit; in this direction, we refer to \cite{HWWY}  for the limit of the Boltzmann equation to the Euler system with general Riemann data and references therein for related topics. In those works, only the cutoff case for the Boltzmann equation was treated since the cutoff assumption plays an important role in the analysis, in particular, in order to make use of the Grad's splitting technique for the linearized Boltzmann collision operator. 

A natural question is to develop some analogous theories in the non-cutoff case for the Boltzmann equation in the compressible Euler limit because the non-cutoff collision is physically more realistic. Similar to the cutoff case, there also could be two kinds of approaches to treat the non-cutoff case for compressible fluid limit of the Boltzmann equation. The first one is based on the low-regularity solutions, for instance, $L^2\cap L^\infty$ approach as in \cite{AMSY-bd} and Wiener-space-related approach as in \cite{DLSS}. Note that those works coped with the existence of solutions for finite Knudsen numbers, but so far it has been unclear for ones to apply them to the compressible fluid limit problem. The second approach is based on the pure energy method in high-order Sobolev spaces. In this direction, we refer to \cite{Duan-Yang-Yu-2022} for the Landau equation via the macro-micro decomposition; see two other applications \cite{DYY-EM} and \cite{DYY-KdV} of the approach. Note that \cite{Duan-Yang-Yu-2022} was inspired by a previous work \cite{Duan-Yu1} which seems to be the first result studying the Landau solutions around local Maxwellians.

The aim of the current paper is twofold. The first one is to extend the classical Caflisch's cutoff result \cite{Caf} to the non-cutoff Boltzmann equation for the compressible Euler limit. Our early work \cite{Duan-Yang-Yu-2022} on the Landau equation inspires us how to treat the non-cutoff Boltzmann collision operator in the framework of macro-micro decomposition around local Maxwellians.  
Another objective, which is much more important and thus becomes our core concern of this paper, is to allow the microscopic component in the macro-micro decomposition to decay in large velocities with only a polynomial rate. To the best of our knowledge, there is no such result in the literature so far. Moreover, we are very motivated by a recent work \cite{DL-hard} where it is essentially necessary to treat the polynomial tail solutions due to the presence of shearing forces. Hence we expect to develop a new approach on the basis of the macro-micro decomposition combined with techniques of constructing polynomial tail solutions in order to study the asymptotic stability of those spatially homogeneous solutions under the high-dimensional spatially inhomogeneous setting; this will be left for our future work.   

In what follow we mention some results on polynomial tail solutions. Whenever any part of solutions decay in large velocity polynomially, we say that they are solutions with polynomial tail. The study of this kind of solutions in kinetic theory could trace back again to Caflisch \cite{Caf} that first proposed the decomposition of solutions in the form $g=\sqrt{M_1}g_1+\sqrt{M_2}g_2$ for two Maxwellians with distinct temperatures, one is local and the other global, and further deduced the corresponding velocity-weighted $L^\infty$ estimates on both $g_1$ and $g_2$ via the bootstrap argument from the basic $L^2$  estimates. In fact, the Gaussian weight with uniform temperature in one component of the Caflisch's decomposition can be relaxed to be only a  polynomial weight, while the other component with Gaussian weight can be initially set to be zero, cf.~\cite{DL-arma,DL-cmaa, DL-hard}, where we still called it by Caflisch's decomposition. Furthermore, the velocity-pointwise estimates on the Boltzmann collision term weighted by the polynomial velocity functions were first considered by  \cite[Proposition 3.1]{AEP-87} and the technique has led to many applications \cite{ELM-94, ELM-95} for the fluid dynamic limit problem on the Boltzmann equation.

In the case of spatially homogeneous Boltzmann equation with hard potentials,  \cite{Mo} established the  spectral gap-like estimates  to construct solutions in some enlarged function spaces corresponding to the slow velocity decay. Later,  \cite{GMM} gave a systematic study for estimates on the resolvents and semigroups of non-symmetric operators which can be applied to the linearized Boltzmann equation in the torus. With such general theory,  \cite{MiMo} also obtained the solution with polynomial tail for kinetic Fokker-Planck equation. There have been a lot of works on slow decay solutions following those aforementioned results. Interested readers may refer to  \cite{CM, CTW} for the nonlinear Landau equation with Coulomb potentials in the torus,  \cite{Br, BG} for the cutoff Boltzmann equation in general bounded domains,  \cite{Cao} for soft potentials, and the references therein. 

In the non-cutoff Boltzmann case, for hard potentials \cite{AMSY}  first got the solutions in Sobolev spaces and then \cite{AMSY-bd} used the Di Giorgi argument to further establish the well-posedness in $L^2\cap L^\infty$ space; see also recent works \cite{CHJ,Cao-22} for the non-cutoff Boltzmann equation with soft potentials. 

Among all the literature mentioned above about the polynomial tail solutions, the space variable was assumed to be in either the torus or bounded domain. In case of the whole space $\R^3$, due to the unbounded property of the spatial domain that leads that the Poincar\'e inequality fails, it seems not straightforward to adopt the same approach for studying the problem. Motivated by \cite{Caf} as well as \cite{DL-arma}, the authors of this paper together with Liu \cite{DLL} first proved the well-posedness of polynomial tail solutions for the Boltzmann equation under the cut-off assumption in the whole space, and further we together with Cao \cite{CDL}  extended the corresponding result to the non-cutoff Boltzmann case; see also an independent work \cite{CG}. 

In the current paper, we basically employ the approaches combined from \cite{CDL} and \cite{Duan-Yang-Yu-2022} by incorporating the Caflisch's decomposition into the Liu-Yang-Yu's macro-micro decomposition (cf.~\cite{Liu-Yang-Yu}) around local Maxwellians, namely,
\begin{equation}\label{int.cd.mmd}
F(t,x,v)=M_{[\rho,u,\theta](t,x)}+\overline{G}(t,x,v)+g(t,x,v),\quad g(t,x,v)=g_1(t,x,v)+\sqrt{\mu}g_2(t,x,v),
\end{equation}
where $g_2$ carries zero initial data.
To state the results roughly, regarding the Cauchy problem \eqref{BE} in the compressible limit $\varepsilon\to 0$, we are able to construct in Theorem \ref{thm1.1} a unique strong solution $F(t,x,v)$ up to an arbitrarily given finite time $T>0$ such that it holds   
\begin{align}\label{int.est}
		\sup_{t\in[0,T]}\sum_{|\alpha|\leq N}\|w_{k-4|\al|}(\chi_{|\al|<N}+\varepsilon^2\chi_{|\al|=N})\partial^\alpha\{F(t,x,v)-M_{[\bar{\rho},\bar{u},\bar{\theta}](t,x)}(v)\}
		\|^2_{L^2_xL_{v}^{2}}\leq C_T\varepsilon^r,
	\end{align}
provided that the initial perturbation is suitably small and the smooth Euler solution $[\bar{\rho},\bar{u},\bar{\theta}](t,x)$ is sufficiently close to a constant state, where $1\leq r\leq 2$, $N\geq 3$, $k$ is a suitably large integer, $C_T>0$ is independent of $\varepsilon$ and $w_{k-4|\al|}=\langle v\rangle^{k-4|\al|}$ denotes the velocity polynomial weight. Note that the amplitude of the perturbation for the Euler solution $[\bar{\rho},\bar{u},\bar{\theta}](t,x)$ around constant states can be independent of $\varepsilon$. Moreover, as an immediate consequence of the proof for \eqref{int.est}, we can also show in Theorem \ref{thm1.2} that the Boltzmann solution to \eqref{BE} can be global in time when we take $r=2$ and take the Euler solution $[\bar{\rho},\bar{u},\bar{\theta}](t,x)$ to be a trivial constant state. To prove Theorems \ref{thm1.1} and \ref{thm1.2}, our goal is to use the combined decomposition \eqref{int.cd.mmd} for obtaining uniform estimates on the energy functional $\mathcal{E}_{N,k}(t)$  and energy dissipation functional $\mathcal{D}_{N,k}(t)$  given in \eqref{ENk} and \eqref{DNk}, respectively. It's delicate to estimate the highest order derivatives of $g_1$ and $g_2$; see Section \ref{sec.8} as well as our explanations in Subsection \ref{sub.sop}.  

Note from \eqref{int.est} that the constructed solution to the Cauchy problem \eqref{BE} behaves as
\begin{equation}
\label{intr.pert}
F(t,x,v)=M_{[\bar{\rho},\bar{u},\bar{\theta}](t,x)}(v)+O(1)\varepsilon^{r/2}\langle v\rangle^{-k},
\end{equation} 
for $0\leq t\leq T, x\in \R^3, v\in \R^3$. Thus, the solution is indeed close to local Maxwellians with the perturbation allowing for a polynomial tail $\langle v\rangle^{-k}$ in large velocities, which is our core concern in this paper. In addition, in case of $1\leq r<2$ the perturbation in \eqref{intr.pert} can be of the order $\varepsilon^{r/2}$ which the usual method of Hilbert expansion of the form
\begin{equation*}
F(t,x,v)=M_{[\bar{\rho},\bar{u},\bar{\theta}](t,x)}(v) +\varepsilon F_1(t,x,v) +\varepsilon^2 F_2(t,x,v)+\cdots,
\end{equation*}
may fail to obtain. This is one of main advantages that we would make use of the formulation through the macro-micro decomposition instead of the Hilbert expansion as above.    

We should also emphasize that the enlarged functional spaces are useful when we consider the kinetic shear flow problem such that one can control the polynomial growth caused by the shear force, see \cite{DL-arma,DL-cmaa, DL-hard} and many references therein. We expect that the techniques developed in this work can be applicable to further study the asymptotic stability of spatially homogeneous shear flow solutions under the spatially inhomogeneous perturbation. 

In the end we also mention an extensive literature for the incompressible fluid limit of the Boltzmann equation as in the classical works \cite{BGL1,BGL2, GL-2002, GoSR, LiMa1, LiMa2}, the method based on the spectral analysis \cite{BU,EP,Ca-Ca}, and recent works on polynomial tail hydrodynamic solutions \cite{BMAM, GaTr, Gerv, GeLo}.

The rest of this paper is arranged as follows. In Section \ref{sec.mr}, we first re-visit the macro-micro decomposition and further formulate the problem in the setting of Caflisch's decomposition for considering the polynomial tail solutions. The main results are stated in Theorem \ref{thm1.1}  and  Theorem \ref{thm1.2}. In Section \ref{sec.pre}, we provide preliminary preparations to be used in the later proof. The main estimates are established in Sections \ref{sec.5} to \ref{sec.8} and the most delicate part occurs to the highest-oder derivatives estimates in Section \ref{sec.8}. We conclude the proof of main theorems in Section \ref{sec.9}.      


\section{Main results}\label{sec.mr}
\subsection{Macro-micro decomposition}
The macro-micro decomposition of the Boltzmann equation was initiated by Liu-Yu \cite{Liu-Yu} and developed by Liu-Yang-Yu \cite{Liu-Yang-Yu}. In what follows we introduce a modified formulation such that one can study the solutions in some enlarged functional spaces. Recall that the collision operator $Q$ in $\eqref{defQ}$ has five collision invariants denoted by
$$
\psi_0(v)=1, \quad \psi_i(v)=v_i\ (i=1,2,3),\quad \psi_4(v)=\frac{1}{2}|v|^2,
$$
such that
\begin{equation*}
	\int_{\mathbb{R}^3}\psi_i(v)Q(F,F)dv=0,\quad i=0,1,2,3,4.
\end{equation*}
Assuming $F=F(t,x,v)$ is a solution of \eqref{BE}, we define the mass density $\rho(t,x)$, momentum $\rho(t,x)u(t,x)$, and
	energy density $e(t,x)+\frac{1}{2}|u(t,x)|^2$ to be
	\begin{align}\label{def.macom}
		\left\{
		\begin{array}{rl}
			\rho(t,x)&:=\int_{\mathbb{R}3}\psi_0(v)F(t,x,v)dv,
			\\
			\rho(t,x) u_{i}(t,x)&:=\int_{\mathbb{R}^3}\psi_{i}(v)F(t,x,v)dv, \quad i=1,2,3,
			\\
			\rho(t,x)\big(e(t,x)+\frac{1}{2}|u(t,x)|^{2}\big)&:=\int_{\mathbb{R}^3}\psi_4(v)F(t,x,v)dv,
		\end{array} \right.
	\end{align}
	where the internal energy 
	$e=\frac{3}{2}R\theta=\theta$ with the Boltzmann constant $R=2/3$ for simplicity, and $u(t,x)=(u_1(t,x),u_2(t,x),u_3(t,x))$ is the bulk velocity. Correspondingly, the local Maxwellian $M=M(t,x,v)$ is given by
	\begin{equation}\label{defM}
		M(t,x,v) \equiv M_{[\rho,u,\theta](t,x)}(v):=\frac{\rho(t,x)}{\sqrt{(2\pi R\theta(t,x))3}}\exp\big\{-\frac{|v-u(t,x)|^2}{2R\theta(t,x)}\big\}.
	\end{equation}
We then introduce the linearized Boltzmann collision operator around the local Maxwellian $M$ as
	\begin{equation*}
		L_{M}f:=Q(f,M)+Q(M,f),
	\end{equation*}
and denote the orthonormal basis of the
null space $\mathcal{N}$ of $L_{M}$ to be
\begin{align}\label{basis}
\left\{
\begin{array}{rl}
	\chi_0(v)&=\frac{1}{\sqrt{\rho}}M,
	\\
	\chi_i(v)&=\frac{v_{i}-u_{i}}{\sqrt{R\rho\theta}}M, \quad \mbox{for $i=1,2,3$,}
	\\
	\chi_4(v)&=\frac{1}{\sqrt{6\rho}}(\frac{|v-u|^{2}}{R\theta}-3)M,\\
	( \chi_i,\frac{\chi_j}{M})_{L^2_v}&=\delta_{ij},\quad i,j=0,1,2,3,4.
\end{array} \right.
\end{align}
	In terms of the orthonormal basis above, we define the macroscopic projection $P_0$ and 
	microscopic projection $P_1$ respectively to be
	\begin{equation}\label{defP}
		P_{0}h\equiv\sum_{i=0}^{4}( h,\frac{\chi_{i}}{M})_{L^2_v}\chi_{i},\quad P_{1}h\equiv h-P_{0}h.
	\end{equation}
	It is direct to verify that 
	$$
	P_0P_0=P_0,\quad
	P_1P_1=P_1,\quad
	P_1P_0=P_0P_1=0,
	$$
	and 
	$$
	P_0Q(F,F)=0.
	$$
	We then decompose the solution of \eqref{BE} $F=F(t,x,v)$ into the
	macroscopic part $M$ and the microscopic part $G$ as
	\begin{equation}\label{FGrelation}
		F=M+G,
	\end{equation}
with $ P_{0}F=M$ and $P_{1}F=G.$

Using the property $Q(M,M)=0$, the Boltzmann equation of \eqref{BE} can be written as
	\begin{equation}\label{1M+G}
		\partial_t (M+G)+v\cdot\nabla_x (M+G)
		=\frac{1}{\varepsilon}L_{M}G+\frac{1}{\varepsilon}Q(G,G).
	\end{equation}
We multiply both sides of \eqref{1M+G} by $\psi_i(v)$ with $i=0,1,2,3,4$
and integrate it with respect to $v$ over $\mathbb{R}^3$ to get the macroscopic system
\begin{align}\label{macro}
	\left\{
	\begin{array}{rl}
		&\partial_t\rho+\nabla_x\cdot(\rho u)=0,
		\\
		&\partial_t(\rho u)+\nabla_x\cdot(\rho u\otimes u)+\nabla_x p=-\int_{\mathbb{R}^{3}} v\otimes v\cdot\nabla_x G dv,
		\\
		&\partial_t [\rho(\theta+\frac{1}{2}|u|^{2})]+\nabla_x \cdot[\rho u(\theta+\frac{1}{2}|u|^{2})+pu]
		=-\int_{\mathbb{R}^3} \frac{1}{2}|v|^{2} v\cdot\nabla_x G dv,
	\end{array} \right.
\end{align}
where the pressure $p=R\rho\theta=\frac{2}{3}\rho\theta$ and $v\otimes v$ denotes the tensor product. 	
In the macroscopic system \eqref{macro}, if we expect the right hand side terms to vanish as $\varepsilon\rightarrow0$, then the limiting system turns to be the full compressible Euler system
\begin{align}
	\label{euler}
	\left\{
	\begin{array}{rl}
		&\partial_t\bar{\rho}+\nabla_x\cdot(\bar{\rho}\bar{u})=0,
		\\
		&\partial_t(\bar{\rho} \bar{u})+\nabla_x \cdot(\bar{\rho}\bar{u}\otimes \bar{u})+\nabla_x \bar{p}=0,
		\\
		&\partial_t[\bar{\rho}(\bar{\theta}+\frac{1}{2}|\bar{u}|^{2})]+\nabla_x \cdot[\bar{\rho}\bar{u}(\bar{\theta}+\frac{1}{2}|\bar{u}|^{2})+\bar{p}\bar{u}]
		=0,
	\end{array} \right.
\end{align}
with the state equation $\bar{p}=R\bar{\rho}\bar{\theta}=\frac{2}{3}\bar{\rho}\bar{\theta}$.
Under the perturbation framework, letting $(\rho,u,\theta)(t,x)$ and $(\bar{\rho},\bar{u},\bar{\theta})(t,x)$ solve \eqref{remacro} and the Euler system \eqref{euler} respectively, we define
\begin{equation}\label{defineper}
	\left\{
	\begin{array}{rl}
		(\widetilde{\rho},\widetilde{u},\widetilde{\theta})(t,x)=&(\rho-\bar{\rho},u-\bar{u},\theta-\overline{\theta})(t,x),
		\\
		g(t,x,v)=&G(t,x,v)-\overline{G}(t,x,v),
	\end{array} \right.
\end{equation}
where
\begin{align}\label{defbarG}
	\overline{G}(t,x,v)\equiv\varepsilon L_M^{-1}P_1\big\{v\cdot(\frac{|v-u|^{2}
		\nabla_x \bar{\theta}}{2R\theta^{2}}+\frac{(v-u)\cdot\nabla_x\bar{u}}{R\theta})M\big\}.
\end{align}
Note by \eqref{FGrelation} and \eqref{defineper} that 
\begin{equation}
\label{dec.MGg}
F(t,x,v)=M(t,x,v)+\overline{G}(t,x,v)+g(t,x,v).
\end{equation}	
For the microscopic system, we apply $P_1$ to both sides of \eqref{1M+G} to deduce
\begin{align}\label{eqG}
	\partial_t G+P_1(v\cdot\nabla_{x}G)+P_1(v\cdot\nabla_{x}M)
	=\frac{1}{\varepsilon}L_M G+\frac{1}{\varepsilon}Q(G,G).
\end{align}
Using the identity
\begin{align*}
	P_1(v\cdot\nabla_{x}M)=P_1\big\{v\cdot (\frac{|v-u|^2
		\nabla_x \widetilde{\theta}}{2R\theta^2}+\frac{(v-u)\cdot\nabla_x\widetilde{u}}{R\theta})M\big\}
	+\frac{1}{\varepsilon}L_M \overline{G},
\end{align*}
which can be obtained by direct calculation, we combine \eqref{defineper} and \eqref{eqG} to get
\begin{align*}
	&\partial_{t}g+v\cdot\nabla_{x}g+P_{1}\big\{v\cdot(\frac{|v-u|^{2}
		\nabla_x\widetilde{\theta}}{2R\theta^{2}}+\frac{(v-u)\cdot\nabla_{x}\widetilde{u}}{R\theta})M\big\}+P_{1}(v\cdot\nabla_{x}\overline{G})+\partial_{t}\overline{G}\notag\\
	=&P_{0}(v\cdot\nabla_x g)+\frac{1}{\varepsilon}L_M g+\frac{1}{\varepsilon}Q(G,G),
\end{align*}
with the initial data 
\begin{equation}
\label{id.g01}
g(0,x,v)=g_0(x,v):=G(0,x,v)-\overline{G}(0,x,v)=F_0(x,v)-M(0,x,v)-\overline{G}(0,x,v),
\end{equation}
where $M(0,x,v)$ and $\overline{G}(0,x,v)$ are defined in terms of $F_0(x,v)$ using \eqref{def.macom}, \eqref{defM} and \eqref{defbarG}. 
Although we study the solution near local Maxwellian, many important results on the perturbation near global Maxwellian also play crucial roles. Hence, we need to define the global Maxwellian to be
\begin{equation*}
	\mu\equiv M_{[1,0,\frac{3}{2}]}(v):=(2\pi)^{-\frac{3}{2}}e^{-\frac{|v|^2}{2}},
\end{equation*}
and denote
\begin{align*}
	L_{\mu}f:=Q(f,\mu)+Q(\mu,f).
\end{align*}
It is direct to see that
\begin{align}\label{inig}
	&\partial_{t}g+v\cdot\nabla_x g+P_{1}\big\{v\cdot(\frac{|v-u|^{2}
		\nabla_x \widetilde{\theta}}{2R\theta^2}+\frac{(v-u)\cdot\nabla_x \widetilde{u}}{R\theta})M\big\}+P_{1}(v\cdot\nabla_{x}\overline{G})+\partial_{t}\overline{G}\notag\\
	=&P_{0}(v\cdot\nabla_x g)+\frac{1}{\varepsilon}L_{\mu}g+\frac{1}{\varepsilon}Q(M-\mu,g)+\frac{1}{\varepsilon}Q(g,M-\mu)+\frac{1}{\varepsilon}Q(G,G).
\end{align}

\subsection{Caflisch's decomposition for the micro part}	
If we only require the  perturbation $g$ in \eqref{dec.MGg}  to be exponential in velocity, such as $g=\sqrt{\mu}f$, we can multiply $\mu^{-1/2}$ to both sides of \eqref{inig} to get the microscopic equation for our energy estimates. However, if we want the perturbation to admit only the polynomial tail, we should further make use of the Caflisch's decomposition \cite{Caf}. In what follows we introduce the formulation.
For constants $A$ and $M$ which will be determined later in Lemma \ref{leLD}, the linear operators $\CL_B$ and $\CL_D$ are respectively defined by
\begin{align}\label{defLB}
	\CL_Bg_1(t,x,v):=\mu^{-1/2}(v)A\chi_M(v) g_1(t,x,v),
\end{align}
and
\begin{align}\label{defLD}
	\CL_D g_1(t,x,v):=(L_\mu-A\chi_M(v)) g_1(t,x,v),
\end{align}
where $\chi_M$ is a smooth indicator function such that for any $M>0$, $\chi_M(v)=1$ if $|v|\leq M$ and $\chi_M(v)=0$ if $|v|\ge2M$. It is straightforward to see that $L_{\mu}=\CL_D+\mu^{1/2}(v)\CL_B$. Moreover, we set
\begin{equation}\label{def.selfL}
	\Gamma(h,g):=\frac{1}{\sqrt{\mu}}Q(\sqrt{\mu}h,\sqrt{\mu}g),
	\quad Lh:=\Gamma(h,\sqrt{\mu})+\Gamma(\sqrt{\mu},h).
\end{equation}
Note that $L$ is the usual self-adjoint operator on $L^2_v$. In terms of the above notations, we introduce two functions $g_1=g_1(t,x,v)$ and $g_2=g_2(t,x,v)$ satisfying
\begin{align}
	\pa_t g_1+v\cdot\nabla_x g_1 =&\frac{1}{\varepsilon}\CL_D g_1+\frac{1}{\varepsilon}Q(g_1,g_1)+\frac{1}{\varepsilon}Q(\sqrt{\mu}g_2,g_1)+\frac{1}{\varepsilon}Q(g_1,\sqrt{\mu}g_2)\notag\\
	&+\frac{1}{\varepsilon}Q(M-\mu,g_1)+\frac{1}{\varepsilon}Q(g_1,M-\mu)+\frac{1}{\varepsilon}Q(\overline{G},g_1)+\frac{1}{\varepsilon}Q(g_1,\overline{G}), \label{g1} 
\end{align}
and
\begin{align}\label{g2}
	\partial_{t}g_2+v\cdot\nabla_{x}g_2&=\frac{1}{\varepsilon}L g_2+\frac{1}{\varepsilon}\CL_B g_1
	+\frac{1}{\varepsilon}\Gamma(\frac{M-\mu}{\sqrt{\mu}},g_2)
	+\frac{1}{\varepsilon}\Gamma(g_2,\frac{M-\mu}{\sqrt{\mu}})
	\nonumber\\
	&\hspace{0.5cm}
	+\frac{1}{\varepsilon}\Gamma(\frac{\overline{G}+\sqrt{\mu}g_2}{\sqrt{\mu}},\frac{\overline{G}+\sqrt{\mu}g_2}{\sqrt{\mu}})+\frac{P_{0}[v\cdot\nabla_{x}(g_1+\sqrt{\mu}g_2)]}{\sqrt{\mu}}-\frac{P_{1}(v\cdot\nabla_{x}\overline{G})}{\sqrt{\mu}}-\frac{\partial_{t}\overline{G}}{\sqrt{\mu}}\nonumber\\
	&\hspace{0.5cm}
	-\frac{1}{\sqrt{\mu}}P_{1}\big\{v\cdot(\frac{|v-u|^{2}\nabla_{x}\widetilde{\theta}}{2R\theta^{2}}+\frac{(v-u)\cdot\nabla_{x}\widetilde{u}}{R\theta})M\big\},
\end{align}
with the initial data 
\begin{equation}
\label{id.g12}
g_1(0,x,v)=g_0(x,v),\quad g_2(0,x,v)=0,
\end{equation}
where $g_0(x,v)$ is given as in \eqref{id.g01}. Therefore, it holds that 
\begin{equation}
\label{cdec.g12}
g(t,x,v)=g_1(t,x,v)+\sqrt{\mu}g_2(t,x,v)
\end{equation}
is the solution to \eqref{inig}. Thus, to solve \eqref{inig} with \eqref{id.g01}, it suffices to solve \eqref{g1} and \eqref{g2} with \eqref{id.g12}.

Notice that in \eqref{g1} we can also write
\begin{align*}
	\pa_t g_1+v\cdot\nabla_x g_1 =&\frac{1}{\varepsilon}L_M g_1-\frac{1}{\varepsilon}\sqrt{\mu}\CL_Bg_1+\frac{1}{\varepsilon}Q(g_1,g_1)+\frac{1}{\varepsilon}Q(\sqrt{\mu}g_2,g_1)+\frac{1}{\varepsilon}Q(g_1,\sqrt{\mu}g_2)\notag\\
	&+\frac{1}{\varepsilon}Q(\overline{G},g_1)+\frac{1}{\varepsilon}Q(g_1,\overline{G}). 
\end{align*}
We apply $P_{1}$ to both sides above to get
\begin{align}\label{Pg1}
	P_{1}(\pa_t g_1)+P_{1}(v\cdot\nabla_x g_1) =&\frac{1}{\varepsilon}L_M g_1-\frac{1}{\varepsilon}P_{1}(\sqrt{\mu}\CL_Bg_1)+\frac{1}{\varepsilon}Q(g_1,g_1)+\frac{1}{\varepsilon}Q(\sqrt{\mu}g_2,g_1)+\frac{1}{\varepsilon}Q(g_1,\sqrt{\mu}g_2)\notag\\
	&+\frac{1}{\varepsilon}Q(\overline{G},g_1)+\frac{1}{\varepsilon}Q(g_1,\overline{G}). 
\end{align}
We turn to \eqref{eqG}, which can be rewritten to
\begin{align}\label{reeqG}
	&\partial_{t}(\overline{G}+\sqrt{\mu}g_2)+\pa_tg_1+P_{1}(v\cdot\nabla_{x}(\overline{G}+\sqrt{\mu}g_2))+P_{1}(v\cdot\nabla_x g_1)+P_{1}(v\cdot\nabla_{x}M)\notag\\
	=&\frac{1}{\varepsilon}L_{M}(\overline{G}+\sqrt{\mu}g_2)+\frac{1}{\varepsilon}L_{M}g_1+\frac{1}{\varepsilon}Q(\overline{G}+\sqrt{\mu}g_2,\overline{G}+\sqrt{\mu}g_2)+\frac{1}{\varepsilon}Q(g_1,g_1)+\frac{1}{\varepsilon}Q(\sqrt{\mu}g_2,g_1)\notag\\
	&+\frac{1}{\varepsilon}Q(g_1,\sqrt{\mu}g_2)+\frac{1}{\varepsilon}Q(\overline{G},g_1)+\frac{1}{\varepsilon}Q(g_1,\overline{G}). 
\end{align}
Then it follows from \eqref{Pg1} and \eqref{reeqG} that
\begin{align*}
	&\partial_{t}(\overline{G}+\sqrt{\mu}g_2)+P_{0}(\pa_t g_1)+P_{1}(v\cdot\nabla_{x}(\overline{G}+\sqrt{\mu}g_2))+P_{1}(v\cdot\nabla_{x}M)\notag\\
	=&\frac{1}{\varepsilon}L_{M}(\overline{G}+\sqrt{\mu}g_2)+\frac{1}{\varepsilon}P_{1}(\sqrt{\mu}\CL_Bg_1)+\frac{1}{\varepsilon}Q(\overline{G}+\sqrt{\mu}g_2,\overline{G}+\sqrt{\mu}g_2). 
\end{align*}
We use $P_{0}(\pa_t g_1)+P_{0}(\pa_t \sqrt{\mu}g_2)=P_{0}(\pa_t (G-\overline{G})=0$ to get
\begin{align*}
	&\partial_{t}\overline{G}+P_{1}(\partial_{t}\sqrt{\mu}g_2)+P_{1}(v\cdot\nabla_{x}(\overline{G}+\sqrt{\mu}g_2))+P_{1}(v\cdot\nabla_{x}M)\notag\\
	=&\frac{1}{\varepsilon}L_{M}(\overline{G}+\sqrt{\mu}g_2)+\frac{1}{\varepsilon}P_{1}(\sqrt{\mu}\CL_Bg_1)+\frac{1}{\varepsilon}Q(\overline{G}+\sqrt{\mu}g_2,\overline{G}+\sqrt{\mu}g_2), 
\end{align*}
which further implies 
	\begin{equation}\label{reG}
		\overline{G}+\sqrt{\mu}g_2=\varepsilon L^{-1}_{M}[P_{1}(v\cdot\nabla_{x}M)]+L^{-1}_{M}\Theta,
	\end{equation}
	with
	\begin{align}\label{defTheta}
		\Theta:=&\varepsilon \partial_{t}\overline{G}+\varepsilon P_{1}(\pa_t\sqrt{\mu}g_2)+\varepsilon P_{1}(v\cdot\nabla_{x}(\overline{G}+\sqrt{\mu}g_2))-P_{1}(\sqrt{\mu}\CL_Bg_1)-Q(\overline{G}+\sqrt{\mu}g_2,\overline{G}+\sqrt{\mu}g_2),
	\end{align}
	since $L_{M}$ is invertible on $\mathcal{N}^{\perp}$. Note that $\Theta$ is microscopic, i.e., $P_{1}\Theta=\Theta$.

\subsection{Macro fluid-type equations}
With the detailed decomposition of the microscopic part above, we now turn back to the macroscopic system \eqref{macro}. Denote smooth functions $\mu(\theta)>0$ and $\kappa(\theta)>0$, which depend only on $\theta$, to be
	\begin{align*}
		\mu(\theta)&=- R\theta\int_{\R^3}\hat{B}_{ij}(\frac{v-u}{\sqrt{R\theta}})
		B_{ij}(\frac{v-u}{\sqrt{R\theta}})dv>0,\quad i\neq j,
		\nonumber\\
		\kappa(\theta)&=-R^2\theta\int_{\R^3}\hat{A}_j(\frac{v-u}{\sqrt{R\theta}})
		A_j(\frac{v-u}{\sqrt{R\theta}})dv>0,
	\end{align*}
	with $i,j=1,2,3,$ where $\hat{A}_j(v)$ and $\hat{B}_{ij}(v)$ are Burnett functions defined by
	\begin{equation}\label{defbur1}
		\hat{A}_j(v)=\frac{|v|^2-5}{2}v_j\quad \mbox{and} \quad \hat{B}_{ij}(v)=v_iv_j-\frac{1}{3}\delta_{ij}|v|^2, 
	\end{equation}
and $A_j(v)$ and $B_{ij}(v)$ are given by
	\begin{equation}\label{defbur2}
		A_{j}(v)=L^{-1}_M[\hat{A}_j(v)M]\quad
		\mbox{and} \quad B_{ij}(v)=L^{-1}_M[\hat{B}_{ij}(v)M].
	\end{equation}
	In terms of the definition of the viscosity coefficient $\mu(\theta)$ and heat conductivity coefficient $\ka(\theta)$, using the identities
	\begin{align}
		-\int_{\R^3} v_i v\cdot\nabla_x
		L^{-1}_M[P_1(v\cdot\nabla_x M)] dv
		&\equiv\sum^{3}_{j=1}[\mu(\theta)D_{ij}]_{x_j},\quad i=1,2,3,\label{id1}
		\\
		-\int_{\R^3} \frac{1}{2}|v|^{2} v\cdot\nabla_{x}
		L^{-1}_M[P_{1}(v\cdot\nabla_{x}M)] dv
		&\equiv\sum^3_{j=1}(\kappa(\theta)\theta_{x_j})_{x_j}+
		\sum^3_{i,j=1}[\mu(\theta) u_iD_{ij}]_{x_j}\label{id2},
	\end{align} where
the viscous stress tensor $D=[D_{ij}]_{1\leq i,j\leq 3}$ is given by
\begin{equation}\label{def.vst}
	D_{ij}=\partial_{x_j}u_i+\partial_{x_i}u_j-\frac{2}{3}\delta_{ij}\nabla_x\cdot u,
\end{equation}
we combine \eqref{macro}, \eqref{reG}, \eqref{id1}, \eqref{id2} and the fact that $G=\overline{G}+g_1+\sqrt{\mu}g_2$ to get the following fluid-type system in the 
compressible Navier-Stokes form:
\begin{align}\label{remacro}
\left\{
\begin{array}{rl}
	\pa_t\rho+\na_x\cdot(\rho u)&=0,
	\\
	\pa_t(\rho u_{i})+\na_x\cdot(\rho u_i u)+\pa_{x_i}p
	&\dis=\varepsilon\sum^3_{j=1}[\mu(\theta)D_{ij}]_{x_j}-\int_{\R^3} v_i v\cdot\na_x g_1 dv
	\\
	&\quad-\int_{\mathbb{R}^{3}} v_{i}(v\cdot\na_x L^{-1}_M\Theta) dv, \quad i=1,2,3,
	\\
	\pa_t[\rho(\theta+\frac{1}{2}|u|^2)]+\na_x\cdot[\rho u(\theta+\frac{1}{2}|u|^2)+ pu]
	&\dis=\varepsilon\sum^3_{j=1}(\kappa(\theta)\theta_{x_j})_{x_j}+\varepsilon
	\sum^3_{i,j=1}[\mu(\theta) u_i D_{ij}]_{x_{j}}\\
	&\quad-\int_{\R^3} \frac{1}{2}|v|^2 v\cdot\na_{x}g_1 dv
	\\
	&\quad-\int_{\R^3} \frac{1}{2}|v|^2 v\cdot\na_x L^{-1}_M\Theta dv.
\end{array} \right.
\end{align}
Then, rewriting the macroscopic system in terms of the perturbation functions defined in \eqref{defineper} and combining \eqref{macro} and \eqref{defineper}, one gets
	\begin{align}\label{macro1}
		\left\{
		\begin{array}{rl}
			&\pa_t\widetilde{\rho}+u\cdot\na_x \widetilde{\rho}+\bar{\rho}\na_x\cdot\widetilde{u}
			+\widetilde{u}\cdot\na_x\bar{\rho}+\widetilde{\rho}\na_x\cdot u=0,
			\\
			&\pa_t\widetilde{u}+u\cdot\na_x\widetilde{u}+\frac{2\bar{\theta}}{3\bar{\rho}}\na_x\widetilde{\rho}
			+\frac{2}{3}\na_x\widetilde{\theta}+\widetilde{u}\cdot\na_x\bar{u}+\frac{2}{3}(\frac{\theta}{\rho}-\frac{\bar{\theta}}{\bar{\rho}})
			\na_x\rho
			\\
			&\qquad\qquad\qquad\quad=-\frac{1}{\rho}\int_{\R^3} v\otimes v\cdot\na_x G dv,
			\\
			&\pa_t\widetilde{\theta}+u\cdot\na_x\widetilde{\theta}+\frac{2}{3}\bar{\theta}\na_x\cdot \widetilde{u}
			+\widetilde{u}\cdot\na_x\bar{\theta}+\frac{2}{3}\widetilde{\theta}\na_x\cdot u
			\\
			&\qquad\qquad\qquad\quad=-\frac{1}{\rho}\int_{\R^3} \frac{1}{2}|v|^{2} v\cdot\na_x G dv+\frac{1}{\rho}u\cdot\int_{\R^3} v\otimes v\cdot\na_x G dv,
		\end{array} 
		\right.
	\end{align}
	which further yields by \eqref{reG}, \eqref{id1} and \eqref{id2} that
	\begin{align}\label{macro2}
		\left\{
		\begin{array}{rl}
			&\pa_t\widetilde{\rho}
			+u\cdot\na_x\widetilde{\rho}+\bar{\rho}\na_x\cdot\widetilde{u}+\widetilde{u}\cdot\na_x\bar{\rho}+\widetilde{\rho}\na_x\cdot u=0,
			\\
			&\pa_t\widetilde{u}_i+u\cdot\na_x\widetilde{u}_i+\frac{2\bar{\theta}}{3\bar{\rho}}\widetilde{\rho}_{x_i}
			+\frac{2}{3}\widetilde{\theta}_{x_i}+\widetilde{u}\cdot\na_x\bar{u}_i+\frac{2}{3}(\frac{\theta}{\rho}-\frac{\bar{\theta}}{\bar{\rho}})\rho_{x_i}
			\\
			&\qquad=\varepsilon\frac{1}{\rho}\sum^3_{j=1}[\mu(\theta)D_{ij}]_{x_j}-\frac{1}{\rho}\int_{\R^3} v_i v\cdot\na_xg_1 dv
			-\frac{1}{\rho}\int_{\R^3} v_i(v\cdot\na_xL^{-1}_M\Theta) dv, \quad i=1,2,3,
			\\
			&\pa_t\widetilde{\theta}+u\cdot\na_x\widetilde{\theta}+\frac{2}{3}\bar{\theta}\na_x\cdot \widetilde{u}
			+\widetilde{u}\cdot\na_x\bar{\theta}+\frac{2}{3}\widetilde{\theta}\na_x\cdot u
			\\
			&\qquad=\varepsilon\frac{1}{\rho}\sum^3_{j=1}(\kappa(\theta)\theta_{x_j})_{x_j}+\varepsilon\frac{1}{\rho}\sum^3_{i,j=1} \mu(\theta) u_{ix_j}D_{ij}-\frac{1}{\rho}\int_{\R^3} \frac{1}{2}|v|^2 v\cdot\na_xg_1 dv
			\\
			&\qquad\quad-\frac{1}{\rho}\int_{\R^3} \frac{1}{2}|v|^2 v\cdot\na_x L^{-1}_M\Theta dv+\frac{1}{\rho}u\cdot\int_{\R^3} v\otimes v\cdot\na_x g_1 dv
			+\frac{1}{\rho}u\cdot\int_{\R^3} v\otimes v\cdot\na_x L^{-1}_M\Theta dv.
		\end{array} \right.
	\end{align}
	We mainly use the perturbed macroscopic systems \eqref{macro1} and \eqref{macro2} to deduce the derivative estimates of the fluid quantities $\widetilde{\rho}$, $\widetilde{u}$ and $\widetilde{\theta}$ except for the highest order.

\subsection{Notations}
Now we introduce some norms. Given a function $f=f(v)$, we define the weighted Sobolev norm to be
$$
\Vert f\Vert_{H^s_k}^2:=\int_{\mathbb{R}^3} \vert\langle v\rangle^k \langle \nabla\rangle^s f(v)\vert^2 dv,
$$
where the weight function is denoted by 
$$
\langle v\rangle^k=w_k(v):=(\sqrt{1+|v|^2})^k.
$$
Particularly, if $s=0$, we write $\Vert f\Vert_{L^2_k}:=\Vert f\Vert_{H^0_k}$. In the case of symmetric perturbation, we introduce the dissipation norm as in \cite{AMUXY-2011-CMP} that
\begin{align}\label{defLvD}
	\Vert f \Vert_{L^2_{v,D}}^2&=\iiint_{\mathbb{R}^3\times \mathbb{R}^3\times \mathbb{S}^2} B(v-u,\sigma) \mu(u) (f(v')-f(v))^2 dvdu d\sigma \notag\\
	&\quad+\iiint_{\mathbb{R}^3\times \mathbb{R}^3\times \mathbb{S}^2} B(v-u,\sigma) f(u)^2 (\mu(v')^{1/2}-\mu(v)^{1/2})^2 dvdu d\sigma.
\end{align}
Given a function $f=f(x,v)$, we also define
\begin{align*}
	\|f\|^2_{L^2_xH^s_k}:=\int_{\R^3}\|f(x,\cdot)\|^2_{H^s_k}dx,\quad  \|f\|^2_{L^2_xL^2_{v,D}}:=\int_{\R^3}\|f(x,\cdot)\|^2_{L^2_{v,D}}dx.
\end{align*}
For simplicity, we use $\|f\|$ to denote $\|f\|_{L^2_v}$ or $\|f\|_{L^2_{x,v}}$. Also, $( \cdot , \cdot )$ means $L^{2}$ inner products in either
$\mathbb{R}^{3}_{x}$ or $\mathbb{R}^{3}_{x}\times \mathbb{R}_{v}^{3}$  with its corresponding $L^{2}$ norm $\|\cdot\|$.

One can see from the definition \eqref{defineper} that $g=g_1+\sqrt{\mu}g_2$ is purely microscopic since $G$ and $\overline{G}$ are both microscopic, which however, does not imply that $g_1$ or $g_2$ is separately microscopic. Thus, we need some notations to capture the macro and micro parts of $g_2$. In the symmetric case for the self-adjoint operator $L$ in \eqref{def.selfL}, the macroscopic projection $\mathbf{P}_0$ from $L^2_v$ to the null space of $L$ is given by
\begin{align}\label{defP0}
	\mathbf{P}_0f(t,x,v)=\{a(t,x)+b(t,x)\cdot v + c(t,x)(\vert v\vert^2-3)\}\mu^{1/2}(v),
\end{align}
with
\begin{align*}
\left\{\begin{array}{rl}
     a(t,x)&=\int_{\R^3}f(t,x,v)\mu^{1/2}(v)dv,\\
	b(t,x)&=\int_{\R^3}f(t,x,v)v\mu^{1/2}(v)dv,\\
	c(t,x)&=\int_{\R^3}f(t,x,v)\frac{1}{6}(|v|^2-3)\mu^{1/2}(v)dv.
\end{array}\right.
\end{align*}
Correspondingly, we write the microscopic projection as 
\begin{align}\label{defP1}
	\mathbf{P}_1f=f-\mathbf{P}_0f. 
\end{align} 


\subsection{Main theorems}
Our main goal is to construct the solution $F(t,x,v)$ of the Boltzmann equation \eqref{BE} which converges to a local Maxwellian 
\begin{equation}\label{localM}
	\overline{M}\equiv M_{[\bar{\rho},\bar{u},\bar{\theta}](t,x)}(v):
	=\frac{\bar{\rho}(t,x)}{\sqrt{(2\pi R\overline{\theta}(t,x))^3}}\exp\big\{-\frac{|v-\overline{u}(t,x)|^2}{2R\overline{\theta}(t,x)}\big\},
\end{equation}
as the Knudsen number $\varepsilon$ tends to zero, 
where $(\bar{\rho},\bar{u},\bar{\theta})(t,x)$ solves the compressible Euler system \eqref{euler}. 
With \eqref{dec.MGg} and \eqref{cdec.g12}, we have 
\begin{equation}
\label{sol.F}
F(t,x,v)=M_{[\rho,u,\theta](t,x)}+\overline{G}(t,x,v)+g_1(t,x,v)+\sqrt{\mu}g_2(t,x,v),
\end{equation}
where the macro perturbation $(\widetilde{\rho},\widetilde{u},\widetilde{\theta})$ in \eqref{defineper} satisfies \eqref{macro1} and \eqref{macro2}, $\overline{G}(t,x,v)$ in \eqref{defbarG} is the first-order correction, and $g_1$, $g_2$ satisfy the coupled system \eqref{g1} and \eqref{g2}. Note that for initial data, we have set $g_2(0,x,v)\equiv 0$ and hence $g_1(0,x,v):=F(0,x,v)-M_{[\rho,u,\theta](0,x)}-\overline{G}(0,x,v)$.

With the above notations we define the instant {\bf energy functional} and {\bf dissipation rate functional} respectively as
\begin{align}\label{ENk}
	\mathcal{E}_{N,k}(t):=&\sum_{|\alpha|\leq N-1}\{\|\partial^{\alpha}(\widetilde{\rho},\widetilde{u},\widetilde{\theta})(t)\|^{2}
	+\|w_{k-4|\al|}\partial^{\alpha}g_1(t)\|^{2}+\|\partial^{\alpha}g_2(t)\|^{2}\}
	\nonumber\\
	&\hspace{0.5cm}+\varepsilon^{2}\sum_{|\alpha|=N}\{\|\partial^{\alpha}(\widetilde{\rho},\widetilde{u},\widetilde{\theta})(t)\|^{2}
	+\|w_{k-4N}\partial^{\alpha}g_1(t)\|^{2}
	+\|\partial^{\alpha}g_2(t)\|^{2}\},
\end{align}
and
\begin{align}\label{DNk}
	\mathcal{D}_{N,k}(t):=&\varepsilon\sum_{1\leq|\alpha|\leq N}
	\|\partial^{\alpha}(\widetilde{\rho},\widetilde{u},\widetilde{\theta})(t)\|^{2}+\frac{1}{\varepsilon}\sum_{|\alpha|\leq N-1}\{\|\partial^{\alpha}g_1(t)\|_{L^2_xH^s_{k-4|\al|+\ga/2}}^{2}+\|\pa^\al\mathbf{P}_1g_2(t)\|^2_{L^2_xL^2_{v,D}}\}	\nonumber\\
	&\hspace{0.5cm}
	+\varepsilon\sum_{|\alpha|=N}\{\|\partial^{\alpha}g_1(t)\|_{L^2_xH^s_{k-4N+\ga/2}}^{2}+\|\pa^\al\mathbf{P}_1g_2(t)\|^2_{L^2_xL^2_{v,D}}\},
\end{align}
where $N\geq 3$, $k$ is large enough to be determined later, and we have denoted the spatial derivative
$\partial^{\alpha}=\partial_{x_{1}}^{\alpha_{1}}\partial_{x_{2}}^{\alpha_{2}}
	\partial_{x_{3}}^{\alpha_{3}}$
for the multi index $\alpha=(\alpha_{1},\alpha_{2},\alpha_{3})$. Note that no time or velocity derivatives are involved. 

\begin{remark}
An important observation here is that since $P_0G=0$ from \eqref{FGrelation}, it holds that $P_0(g_1+\sqrt\mu g_2)=0$, which yields
\begin{align}\label{g2macro}
	\mathbf{P}_0g_2=-\{\int_{\R^3}g_1 dv+v\cdot\int_{\R^3}vg_1 dv+(|v|^2-3)\int_{\R^3}\frac{1}{6}(|v|^2-3)g_1 dv\}\sqrt\mu.
\end{align}
Hence, one can recover the macroscopic dissipation of $g_2$ from the dissipation of $g_1$ in the sense that 
\begin{align}\label{DNk1}
	\mathcal{D}_{N,k}(t)\sim&\ \varepsilon\sum_{1\leq|\alpha|\leq N}
	\|\partial^{\alpha}(\widetilde{\rho},\widetilde{u},\widetilde{\theta})(t)\|^{2}+\frac{1}{\varepsilon}\sum_{|\alpha|\leq N-1}\{\|\partial^{\alpha}g_1(t)\|_{L^2_xH^s_{k-4|\al|+\ga/2}}^{2}+\|\pa^\al g_2(t)\|^2_{L^2_xL^2_{v,D}}\}	\nonumber\\
	&\hspace{0.5cm}
	+\varepsilon\sum_{|\alpha|=N}\{\|\partial^{\alpha}g_1(t)\|_{L^2_xH^s_{k-4N+\ga/2}}^{2}+\|\pa^\al g_2(t)\|^2_{L^2_xL^2_{v,D}}\}.
\end{align}
\end{remark}

To construct the Boltzmann solution, we first need the following proposition which gives the existence of \eqref{euler} so as to define \eqref{localM}. We remark that the statement of this proposition is given in \cite{Duan-Yang-Yu-2022} and its proof is a modification of \cite[Lemma 3.1 and Lemma 3.2]{Guo-Jang-Jiang-2010}.

\begin{proposition}
	Let $T>0$ be a fixed finite time and $\bar\eta>0$ be a constant, then there is a sufficiently small constant $\eta_T>0$ and a constant $C_T>1$ such that if the initial data 
	$(\bar{\rho},\bar{u},\bar{\theta})(0,x)=(\bar{\rho}_{0},\bar{u}_{0},\bar{\theta}_{0})(x)$ around the constant state $(1,0,3/2)$ satisfies
	\begin{equation}\label{defeta}
		\eta:=\|(\bar{\rho}_{0}(x)-1,\bar{u}_{0}(x),\bar{\theta}_{0}(x)-\frac{3}{2}-\bar\eta)\|_{H^{l}}
		\leq \eta_{T},\quad \inf_{x\in \R^{3}}\bar{\theta}_0(x)>\frac{3}{2},
	\end{equation}
	for $l\geq N+2$ with integer $N\geq 3$, then the Cauchy problem on the compressible Euler system \eqref{euler} admits a unique smooth solution $(\bar{\rho},\bar{u},\bar{\theta})(t,x)$
	over $[0,T]\times \R^3$ satisfying
	$$
	\inf_{t\in[0,T],\,x\in \R^{3}}\bar{\rho}(t,x)>0,\quad
	\inf_{t\in[0,T],\,x\in \R^{3}}\bar{\theta}(t,x)>\frac{3}{2},
	$$
	and
	\begin{equation}\label{background}
		\sup_{t\in[0,T]}\|(\bar{\rho}(t,x)-1,\bar{u}(t,x),\bar{\theta}(t,x)-\frac{3}{2}-\bar\eta)\|_{H^{l}}\leq C_{T}\eta.
	\end{equation}
\end{proposition}

With the solution to compressible Euler system obtained above, letting $1\leq r\leq 2$ to be a given parameter, we make the a priori assumption
\begin{align}
	\label{apriori}
	\sup_{0\leq t\leq T}\mathcal{E}_N(t)\leq  \eta_0\varepsilon^{r},
\end{align}
for a constant $0<\eta_0<1$ to be determined later.
Under the a priori assumption \eqref{apriori}, noting $N\geq 3$ in \eqref{ENk}, we have from the embedding inequality that
$$
\sup_{t\in[0,T]}\|\widetilde{\rho}(t,\cdot)\|_{L_x^{\infty}}
\leq C\sup_{t\in[0,T]}\|\widetilde{\rho}(t,\cdot)\|_{H^2_x}
\leq C\varepsilon^{r/2},
$$
which together with \eqref{background} yields
$$
|\rho(t,x)-1|\leq |\rho(t,x)-\bar{\rho}(t,x)|+|\bar{\rho}(t,x)-1|
\leq C_T(\varepsilon^{r/2}+\eta).
$$
Similar estimates also hold for $u(t,x)$ and $\theta(t,x)$. Therefore,
for sufficiently small $\varepsilon$ and $\eta$, it holds that
\begin{equation}\label{rhoutheta}
	|\rho(t,x)-1|+|u(t,x)|+|\theta(t,x)-\frac{3}{2}-\bar\eta|<C_T(\varepsilon^{r/2}+\eta),\quad
	\frac{3}{2}<\theta(t,x)<2,
\end{equation}
uniformly in all $(t,x)\in[0,T]\times\mathbb{R}^{3}$.

We now state our main result of this paper as follows.

\begin{theorem}\label{thm1.1}
Assume \eqref{b.kernel} and \eqref{b.kernel.ass} with $-3<\ga\leq1$, $0<s<1$ and $\gamma+2s >-1$. Let $N\geq 3$, $k>25+4N$, $1\leq r\leq 2$, and $T>0$ be given. There are constants $\bar\eta,\varepsilon_{0},\eta_0,C_1>0$ such that  for any $\varepsilon\in(0,\varepsilon_{0})$ and $\eta\in (0,\eta_0)$, 
if it holds that $F_0(x,v)\geq 0$ and
	\begin{align}\label{initial}
		\mathcal{E}_{N,k}(0)\leq \frac{1}{C_1}\eta_0\varepsilon^r,
	\end{align}
then the Cauchy problem on the non-cutoff Boltzmann equation \eqref{BE} admits a unique  solution $F(t,x,v)\geq 0$ of the form \eqref{sol.F} over $[0,T]\times \R^3\times \R^3$ satisfying the estimate
		\begin{align}\label{energyestimate}
		\mathcal{E}_{N,k}(t)+\frac{1}{2}\int^{t}_{0}\mathcal{D}_{N,k}(s) ds\leq \frac{1}{2}\eta_0\varepsilon^{r},\quad 0\leq t\leq T.
	\end{align}
	Moreover, there is  $C_T>0$ independent of $\varepsilon$ such that
\begin{align}\label{estsolution}
	\sup_{t\in[0,T]}\sum_{|\alpha|\leq N}\|w_{k-4|\al|}(\chi_{|\al|<N}+\varepsilon^2\chi_{|\al|=N})\pa^\al\{F(t,x,v)-M_{[\bar{\rho},\bar{u},\bar{\theta}](t,x)}(v)\}
	\|^2_{L^2_xL_{v}^{2}}\leq C_T\varepsilon^r.
\end{align}
\end{theorem} 

\begin{remark}
We emphasize three aspects of the main result above. First of all, regarding the compressible Euler limit of the Boltzmann equation, we extend the classical result in \cite{Caf} for the cutoff case to the non-cutoff case. Second, we also allow initial data to admit a polynomial tail at large velocities. In fact, in terms of \eqref{initial} and \eqref{ENk} together with \eqref{sol.F}, noting $g_2(0,x,v)=0$, initial data $g_1(0,x,v)$ and hence $F_0(x,v)$ may decay in large velocities with a polynomial rate. Therefore, the class of function spaces for initial data has been enlarged. Third, in comparison with the torus domain, we treat the whole space where the Poincar\'e inequality is no longer applicable. We notice that the polynomial tail solutions in the case of the whole space were studied very recently independently in \cite{CDL} and \cite{CG} for the Boltzmann equation around global Maxwellians. Note that time-decay properties are essential for the existence proof in \cite{CDL} and \cite{CG} but no longer needed in the proof of the current result. 
\end{remark}

\begin{remark}
From \eqref{estsolution}, we note that as $\varepsilon\to 0$, the rate of convergence of the Boltzmann solution $F(t,x,v)$ toward the Euler solution $M_{[\bar{\rho},\bar{u},\bar{\theta}](t,x)}(v)$ over $[0,T]$  can be of order $\varepsilon^{r/2}$ which is strictly lower than $\varepsilon$ in case of $1\leq r<2$.  Thus, the size of initial data becomes much larger when $\varepsilon$ is small. Such result is impossible to be obtained through the usual method of Hilbert expansion which gives only the $\varepsilon$-order rate as in \cite{Caf}. Moreover, when $r=1$, the convergence rate in $L^2$ norm is of order $\varepsilon^{1/2}$ which seems to be sharp according to our proof; see the detailed reason in the next subsection. Of course it could be interesting to find a different approach to construct solutions for larger size perturbation in case $0<r<1$.
\end{remark}

\begin{remark}
We point out that if one replaces the Euler system \eqref{euler} by the acoustic system as the linearization around constant states, a similar result on the acoustic limit as in \cite{Guo-Jang-Jiang-2010} and \cite{Duan-Yang-Yu-2022}  could also be obtained. We omit the proof and discussions of such result to restrict the length of this paper. 
\end{remark}

Moreover, the constructed Boltzmann solution in Theorem \ref{thm1.1} can be global in time provided that the Euler solution is identical to constant states. In case of a fixed finite Knudsen number, such result was obtained in a previous work \cite{CDL}, where even the large time behavior of solutions has been figured out in detail. Here, we focus on the compressible fluid limit and adopt a different approach based on the macro-micro decomposition instead of the perturbation approach around global Maxwellians in \cite{CDL}. In fact, the proof of the following result is an immediate consequence of that of Theorem \ref{thm1.1} relying on nonlinear estimates in two kinds of different forms to be explained in the next subsection, and thus it may be of its own interest to give a separate statement.

\begin{theorem}\label{thm1.2}
Let all the assumptions of 	Theorem \ref{thm1.1} be satisfied. Take $r=2$ and let $(\bar{\rho},\bar{u},\bar{\theta})(t,x)\equiv(1,0,\frac{3}{2}+\bar\eta)$ be a constant solution of the compressible Euler system \eqref{euler}. There are constants $\bar\eta,\varepsilon_{0},\eta_0,C_1>0$ such that  for any $\varepsilon\in(0,\varepsilon_{0})$, if it holds that $F_0(x,v)\geq 0$ and
	\begin{align*}
		\mathcal{E}_{N,k}(0)\leq \frac{1}{C_1}\eta_0\varepsilon^2,
	\end{align*}
	then the Cauchy problem on the non-cutoff Boltzmann equation \eqref{BE} admits a unique global-in-time solution $F(t,x,v)\geq 0$ of the form \eqref{sol.F}  over $[0,\infty]\times \R^3\times \R^3$ satisfying the estimate
	\begin{align}\label{globalenergy}
		\mathcal{E}_{N,k}(t)+\frac{1}{4}\int^{t}_{0}\mathcal{D}_{N,k}(s) ds\leq \frac{1}{2}\eta_0\varepsilon^2, \quad t\geq 0.
	\end{align}
	Moreover, there is $C>0$ independent of $\varepsilon$ such that 
	\begin{align}\label{estglobal}
		\sup_{t\in[0,\infty]}\sum_{|\alpha|\leq N}\|w_{k-4|\al|}(\chi_{|\al|<N}+\varepsilon^2\chi_{|\al|=N})\pa^\al\{F^{\varepsilon}(t,x,v)-M_{[1,0,3/2+\bar\eta]}(v)\}
		\|^2_{L^2_xL_{v}^{2}}\leq C\eta_0\varepsilon^2.
	\end{align}
\end{theorem} 

\begin{remark}
As mentioned before, the large time behavior of solutions can be further studied in terms of the approach in \cite{CDL}, in particular, $g_1(t,x,v)$ decays exponentially while $(\widetilde{\rho},\widetilde{u},\widetilde{\theta})$ and $g_1(t,x,v)$ decay only polynomially as the heat semigroup under certain extra smallness assumptions on initial data. We omit the proof of such result in the current framework for brevity.
\end{remark}

\subsection{Strategies of the proof}\label{sub.sop}
We now explain some strategies of our proof. One of the key points is how to decompose the macro and micro equations in order to allow the velocity perturbation of the microscopic part to be only polynomial. In terms of the macro-micro decomposition $F=M+G=M+\overline{G}+g$ with the correction term $\overline{G}$ that decays in $v$ exponentially, we make a similar Caflisch-type decomposition for the microscopic part $g=g_1+\sqrt{\mu}g_2$ as in \cite{CDL,DLL} such that $g_1$ carries the initial data with polynomial tail and $g_2$ behaves like the classic symmetric case but with zero initial data. Then we also decompose the linear operator $L_{\mu}=\CL_D+\mu^{1/2}(v)\CL_B$, which describes the linear part of $g$, such that $\CL_D$ as a relaxation operator dominates the linear part of $g_1$ with only polynomial perturbation in velocity and the degenerate symmetric operator $\CL_B$ can absorb any increase in $v$. But unlike the previous case \cite{CDL,DLL}  where we just construct solutions near global Maxwellians, there are fluid quantities in the equation of $g$ now. Also, $g$ appears in the macroscopic equations. So we must design both fluid and non-fluid systems in a completely different way in order to control the interaction among Knudsen number, macroscopic and microscopic quantities, and in the same time, avoid the exponential increase of $g_1$. 

To interpret the difficulty mentioned above, we start from the standard process of getting the fluid type system \eqref{macro}. The left hand side of the macroscopic system is the same as Euler system, but on the right hand side, there are terms including microscopic quantities such as $\int_{\mathbb{R}^{3}} v\otimes v\cdot\nabla_x G dv$. To estimate this term and obtain the viscosity, one usually uses the microscopic equation
$$
\partial_t G+P_1(v\cdot\nabla_{x}G)+P_1(v\cdot\nabla_{x}M)
=\frac{1}{\varepsilon}L_M G+\frac{1}{\varepsilon}Q(G,G),
$$
so as to get
$G=\varepsilon L^{-1}_{M}[P_{1}(v\cdot\nabla_{x}M)]+L^{-1}_{M}(\varepsilon \partial_{t}G+\varepsilon P_{1}(v\cdot\nabla_{x}G)-Q(G,G))$, cf.~\cite{Liu-Yang-Yu,Duan-Yu1,Duan-Yang-Yu-2022}. Then, we decompose $G=\overline{G}+\sqrt{\mu}f$ and treat $f$ in the microscopic estimates. However, for our decomposition $G=\overline{G}+g_1+\sqrt{\mu}g_2$, due to the non-cutoff assumption, the best estimate for $L^{-1}_{M}$ one can obtain so far is $$
\|w_kM^{-\frac{1}{2}}L_M^{-1}f\|\leq C_{\rho,u,\theta}\|w_k M^{-\frac{1}{2}}f\|_{H^{-s}_{-\ga/2}},
$$
where the presence of $M^{-\frac{1}{2}}$ implies the exponential increase of $g_1$ which is harmful for our estimates. Hence, we write 
$$
\int_{\mathbb{R}^{3}} v\otimes v\cdot\nabla_x G dv=\int_{\mathbb{R}^{3}} v\otimes v\cdot\nabla_x g_1 dv+\int_{\mathbb{R}^{3}} v\otimes v\cdot\nabla_x (G-g_1) dv,
$$
and consider the equation of $G-g_1$ instead, which leads to
\begin{align*}
\partial_t (G-g_1)+P_1(v\cdot\nabla_{x}(G-g_1))+P_1(v\cdot\nabla_{x}M)
=&\frac{1}{\varepsilon}L_M (G-g_1)+\frac{1}{\varepsilon}(Q(G,G)-Q(g_1,g_1))\\
&-(\pa_t g_1+P_1(v\cdot\nabla_{x}g_1)-Q(g_1,g_1)).
\end{align*} 
Then our construction of $g_1$ should meet two requirements, the first one is allowing polynomial increase, and the other is that the remainder $\pa_t g_1+P_1(v\cdot\nabla_{x}g_1)-Q(g_1,g_1)$ can be controlled even with the exponential increase of velocity. One can see later from our proof that the equation of $g_1$ we construct in \eqref{g1} satisfies both requirements. 

The same strategy occurs in our highest order estimates. In fact, in order to avoid the $\na_{x}$ in the fluid-type equation which increases the order of spacial derivatives, one may directly start from the original Boltzmann equation \eqref{BE} and turn to $\|F/\sqrt{\mu}\|$ as in \cite{Liu-Yang-Yu,Duan-Yu1,Duan-Yang-Yu-2022}, but the factor $\sqrt{\mu}^{-1}$ causes an exponential increase for $g_1$ in terms of our representation $G=\overline{G}+g_1+\sqrt{\mu}g_2$. Thus, we estimate $g_1$ separately and then derive the equation for $(F-g_1)/\sqrt{\mu}$ to control $\|(F-g_1)/\sqrt{\mu}\|$. It can be proved that the estimate of $\|(F-g_1)/\sqrt{\mu}\|$ shows the bound for all the highest order energy except for $g_1$. We also mention that it is possible to build other equations than \eqref{g1} for $g_1$, like moving some nonlinear terms from \eqref{g2} to \eqref{g1}. But a good property of our construction is that we are still able to keep the exponential in time decay of our $g_1$, which means that even for the time-weighted term $e^{\la t}g_1$ with a small constant $\la>0$, one can still close the estimates. Though, we would not study the time decay property in the current paper; see the time-velocity decay structure established in \cite{CDL}.

The other main concern of our paper is to capture the order of Knudsen number $\varepsilon$ for the perturbation. If one applies the normal Hilbert expansion 
$$
F(t,x,v)=\sum_{n=0}^{3}\varepsilon^n F_n(t,x,v)+
F_R(t,x,v),
$$
and turns to estimate the remainder $F_R$, the size of some Sobolev norm (either $L^2\cap L^\infty$ or $H^N$) of the perturbation should be $\varepsilon^{3/2}$, since the behavior of each $F_n$ is determined by $F_i$ with $0\leq i\leq n-1$. In our approach, we directly study the perturbation near the local Maxwellian which contains all the macroscopic information of the solution by defining the fluid quantities of $M$ by  the macroscopic part of $F$. Therefore, the rest part $G=F-M$ should be purely microscopic. Such way avoids the expansion in the order of $\varepsilon$ and allows us to study its lower-order behavior. Of course, this also causes a lot of difficulties due to the local Maxwellian. 

Another core observation that helps us to study the order of $\varepsilon$ is on the nonlinear estimates. Let us take one type of fluid term to explain the issue. For instance, when we estimate terms like 
$$
\int^t_0(\nabla_{x}u\,\partial^{\alpha}\widetilde{\rho},\partial^{\alpha}\widetilde{\rho})_{L^2_x}(s)ds,
$$
it is noticed that there are two ways to bound this type of integral by 
\begin{equation}
\label{intr.cc}
\text{either $\dis \varepsilon^{-1}\sup\limits_{0\leq s\leq t}\sqrt{\mathcal{E}_{N,k}(s)}\int^t_0\mathcal{D}_{N,k}(s)ds$ or $\varepsilon^{-1/2}t\sup\limits_{0\leq s\leq t}\mathcal{E}_{N,k}(s)^{3/2}$.}
\end{equation}
The advantage of the former one is that it is more convenient to study global solutions in the hydrodynamic limit problem since there is no growth in time, while it forces ones to impose the necessary a priori assumption that $\mathcal{E}_{N,k}(t)\leq \eta_0\varepsilon^2$, where the order of $\varepsilon$ cannot be lower in this framework. Using the latter bound in \eqref{intr.cc} one can allow the estimate $\mathcal{E}_{N,k}(t)\leq \eta_0\varepsilon$ up to finite time. Hence, in our proof, we keep both bounds and assume $\mathcal{E}_{N,k}(t)\leq \eta_0\varepsilon^r$.  To prove the hydrodynamic limit to compressible Euler system over any finite time, we take $1\leq r\leq 2$ and use the latter one in \eqref{intr.cc}. To obtain the global existence in case when  the background Euler solutions are constant states, we take $r=2$ and use the former bound in \eqref{intr.cc}.

\section{Preliminaries}\label{sec.pre}	
Before going to the details of energy estimates, in this section, we give some useful properties of the linear operators $L_\mu$, $L_M^{-1}$ and nonlinear operators $Q$, $\Ga$ for later use. Moreover, when it comes to the higher order derivative estimates, there will be many similar terms involving $\overline{G}$ and its derivatives and hence we also take care of such terms in this section to avoid repeating computations.

The next three lemmas, which help us control the collision operator with polynomial velocity weight, can be found in \cite[Theorem 1.1]{H} and \cite[Lemma 2.3, Theorem 2.1]{CHJ}, respectively.

\begin{lemma}\label{upperQ} 
Suppose $-3<\ga\leq1$, $0<s<1$, $\gamma+2s >-1$. 
	Let $w_1, w_2 \in \R$, $0\leq a, b \leq 2s$ with $w_1+w_2 =\gamma+2s$  and $a+b =2s$. Then there exists a constant $C$ such that for any functions $f, g, h$, we have \\
	(1) if $\gamma + 2s >0$, then 
	\[
	|(Q(g, h), f)| \le C (\Vert g \Vert_{L^1_{\gamma+2s +(-w_1)^++(-w_2)^+}}  +\Vert g \Vert_{L^2} ) \Vert h \Vert_{H^a_{w_1}}   \Vert f \Vert_{H^b_{w_2}} ,
	\]
	(2) if $\gamma + 2s = 0$, then 
	\[
	|(Q(g, h), f)|\le C (\Vert g \Vert_{L^1_{w_3}} + \Vert g \Vert_{L^2}) \Vert h \Vert_{H^a_{w_1}}\Vert f \Vert_{H^b_{w_2}} ,
	\]
	where $w_3 = \max\{\delta,(-w_1)^+ +(-w_2)^+ \}$, with $\delta>0$ sufficiently small,\\
	(3) if $-1< \gamma + 2s < 0$, then 
	\[
	|(Q(g, h), f)|\le C  (\Vert g \Vert_{L^1_{w_4}} + \Vert g \Vert_{L^2_{-(\gamma+2s)}}) \Vert h \Vert_{H^a_{w_1}}   \Vert f \Vert_{H^b_{w_2}},  
	\]
	where $w_4= \max\{-(\gamma+2s), \gamma+2s +(-w_1)^+ +(-w_2)^+ \}$.
\end{lemma}

\begin{lemma} 
Suppose $-3<\ga\leq1$, $0<s<1$, $\gamma+2s >-1$.  For any $k \ge 14$, and functions $g, h$, we have
	\begin{align*}
		\quad\,|(Q (h, \mu), g \langle v \rangle^{2k}) |
		&\le   \Vert b(\cos\theta) \sin^{k-\frac {3+\gamma} 2} \frac \theta 2 \Vert_{L^1_\theta}    \Vert h \Vert_{L^2_{k+\gamma/2}}\Vert g \Vert_{L^2_{k+\gamma/2}} + C_k \Vert h \Vert_{L^2_{k+\gamma/2-1/2}}\Vert g \Vert_{L^2_{k+\gamma/2-1/2}}
		\\
		&\le   \Vert b(\cos\theta) \sin^{k- 2} \frac \theta 2 \Vert_{L^1_\theta}    \Vert h \Vert_{L^2_{k+\gamma/2}}\Vert g \Vert_{L^2_{k+\gamma/2}}  + C_k \Vert h \Vert_{L^2_{k+\gamma/2-1/2}}\Vert g \Vert_{L^2_{k+\gamma/2-1/2}},
	\end{align*}
	for some constant $C_k>0$.  
\end{lemma}

\begin{lemma}
	Suppose $-3<\gamma\le 1$, $0<s<1$, $\gamma+2s>-1$, $k\ge 14$ and $F= \mu +g \ge 0$. If there exist $C_1,C_2>0$ with
	\begin{align*}
		F \ge 0,\quad  \Vert F \Vert_{L^1} \ge C_1, \quad \Vert F \Vert_{L^1_2} +\Vert F \Vert_{L \log L} \le C_2,
	\end{align*}
where
$\Vert F \Vert_{L \log L}=\int_{\R^3}|F(v)|\log(1+|F(v)|)dv$, 
	then there are constants $c, C_k>0$ such that
	\begin{align*}
		(Q(F, f), w_{2k}f )
		&\le   - \frac {1} {8} \Vert  b(\cos \theta) \sin^2 \frac \theta 2\Vert_{L^1_\theta}\Vert f \Vert_{L^2_{k+\gamma/2}}^2  - c \Vert f \Vert_{H^s_{k+\gamma/2}}^2 + C_k  \Vert f \Vert_{L^2}^2
		\notag\\
		&\quad+C_k\Vert f \Vert_{L^2_{14}} \Vert g \Vert_{H^s_{ k+\gamma/2 }}\Vert f \Vert_{H^s_{ k + \gamma/2}} +C_k\Vert g \Vert_{L^2_{14} } \Vert f \Vert_{H^s_{ k + \gamma/2}}^2.
	\end{align*}
\end{lemma}
Collecting the above two lemmas, one gets the following corollary.
\begin{corollary}
	Suppose $-3<\gamma\le 1,\ 0<s<1,\ \gamma+2s>-1$, $k\ge 14$ and $F= \mu +g \ge 0$. If there exist $C_1,C_2>0$ such that 
	\begin{align*}
		F \ge 0,\quad  \Vert F \Vert_{L^1} \ge C_1, \quad \Vert F \Vert_{L^1_2} +\Vert F \Vert_{L \log L} \le C_2,
	\end{align*}
	then there are constants $c, C_k>0$ such that
	\begin{align*}
		(L_\mu f, w_{2k}f )   +   (Q(g, f), w_{2k}f)  = &   (Q(\mu+g, f), w_{2k}f ) + (Q(f, \mu), w_{2k}f  ) 
		\notag\\
		\le &  - c\Vert f \Vert_{H^s_{k+\gamma/2}}^2 + C_k  \Vert f \Vert_{L^2}^2
		+C_k\Vert f \Vert_{L^2_{14}} \Vert g \Vert_{H^s_{ k+\gamma/2 }}\Vert f \Vert_{H^s_{ k + \gamma/2}}\notag\\
		&\qquad +C_k\Vert g \Vert_{L^2_{14} } \Vert f \Vert_{H^s_{ k + \gamma/2}}^2.
	\end{align*}
\end{corollary}

Noticing there is an $L^2$ norm on the right hand side of the above inequality which may cause trouble for our estimates, we use $\CL_\CD$ defined in \eqref{defLD} to avoid the extra $L^2$ norm.

\begin{lemma}\label{leLD}
	Suppose that $-3<\gamma\le 1$, $0<s<1$, $\gamma+2s>-1$ and $F= \mu +g \ge 0$. Then there is a constant $\de>0$ such that if there exist $A_1$, $A_2>0$ satisfying
	\begin{align*}
		F \ge 0,\quad  \Vert F \Vert_{L^1} \ge A_1, \quad \Vert F \Vert_{L^1_2} +\Vert F \Vert_{L \log L} \le A_2,
	\end{align*}
	then for $k\geq14$, there are $A$ and $M$ for the operator $\CL_\CD$, and a constant $C_k$, such that
	\begin{align}\label{leld}
		(\CL_\CD f, w_{2k}f )   +   (Q(g, f), w_{2k}f)  \le&   - \de \Vert f \Vert_{H^s_{k+\gamma/2}}^2
		+C_k\Vert f \Vert_{L^2_{14}} \Vert g \Vert_{H^s_{ k+\gamma/2 }}\Vert f \Vert_{H^s_{ k + \gamma/2}}\notag\\
		& +C_k\Vert g \Vert_{L^2_{14} } \Vert f \Vert_{H^s_{ k + \gamma/2}}^2.
	\end{align}
\end{lemma}
For a commutator on the collision operator $Q$ with weight $w_k(v)$, we have the following lemma from \cite[Lemma 2.4]{CHJ}.

\begin{lemma}\label{commu}
	Suppose $-3<\gamma\le 1,\ 0<s<1,\ \gamma +2s>-1$, $k\ge  14$ and $g, f, h$ are smooth functions. It holds that
	\begin{align}\label{0order}
		|(w_kQ(g, f)-Q(g,w_kf),w_kh)|&\le C_k\Vert f\Vert_{L^2_{14}}\Vert g\Vert_{H^s_{k+\gamma/2}}\Vert h\Vert_{H^s_{k+\gamma/2}}\notag\\
		&\qquad+C_k\Vert g\Vert_{L^2_{14}}  \Vert f\Vert_{H^s_{k+\gamma/2}}\Vert h\Vert_{H^s_{k+\gamma/2}}.
	\end{align}
\end{lemma}
With the above lemma, we are able to control $|(Q(f, g),w_{2k}h)|$.
\begin{lemma}
	Suppose $-3<\gamma\le 1,\ 0<s<1,\ \gamma+2s >-1$. For any functions $f, g, h$ and $k \ge 14$, it holds that
	\begin{align}\label{fgh}
		(Q(f, g),w_{2k}h)\leq &C_k \Vert f \Vert_{L^2_{14}} \min\{\|g\|_{H^s_{k+\gamma/2}}\|h\|_{H^s_{k+\gamma/2+2s}},\,\Vert g \Vert_{H^s_{k+\gamma/2+2s }}  \Vert h \Vert_{H^s_{k+\gamma/2}}\}\notag\\
		&+  C_k\Vert g \Vert_{L^2_{14}}  \Vert f \Vert_{H^s_{k+\gamma/2}}  \Vert h \Vert_{H^s_{k+\gamma/2}} .
	\end{align}	
\end{lemma}
\begin{proof}
	A direct calculation shows that
	$$
	(Q(f, g), w_{2k}h)\leq |(Q(g,w_kf),w_kh)|+|(w_kQ(g, f)-Q(g,w_k f),w_kh)|.
	$$
	Then one has from Lemma \ref{upperQ} that
	$$
	|(Q(g,w_kf),w_kh)|\leq C_k \Vert f \Vert_{L^2_{14}} \min\{\|g\|_{H^s_{k+\gamma/2}}\|h\|_{H^s_{k+\gamma/2+2s}},\,\Vert g \Vert_{H^s_{k+\gamma/2+2s }}  \Vert h \Vert_{H^s_{k+\gamma/2}}\},
	$$
	which, together with \eqref{0order}, yields \eqref{fgh}.
\end{proof}
Then we turn to the symmetric case where we focus on the operators $L$ and $\Ga$.
We recall the results in \cite[Proposition 2.1]{AMUXY-2012-JFA} and \cite[Theorem 1.2]{AMUXY-KRM-2013}.
\begin{lemma}
	If $0<s<1$  and $\gamma>-3$, there exists a constant $c$ such that
\begin{align}\label{coercive}
	-(L g,g)\geq  c \Vert \mathbf{P}_1g \Vert_{L^2_{v,D}}^2.
\end{align}
\end{lemma}
\begin{lemma}
If $0<s<1$  and $\gamma>\{-3,-3/2-2s\}$, there exists a constant $C$ such that
\begin{align}\label{Trilinear}
	\vert (\Gamma(f,g),h)\vert \le C \Vert f\Vert_{L^2_v} \Vert g \Vert_{L^2_{v,D}} \Vert h \Vert_{L^2_{v,D}},
\end{align}
and 
\begin{align}\label{Trilinears}
	\vert (\Gamma(f,g),h)\vert \le C& \big(\Vert f\Vert_{L^2_{s+\ga/2}} \Vert g \Vert_{L^2_{v,D}}+\Vert f\Vert_{L^2_{v,D}} \Vert g \Vert_{L^2_{s+\ga/2}}\notag\\
	&\quad+\min\{\Vert f\Vert_{L^2_{s+\ga/2}} \Vert g \Vert_{L^2_v}+\Vert f\Vert_{L^2_v} \Vert g \Vert_{L^2_{s+\ga/2}}\}\big) \Vert h \Vert_{L^2_{v,D}}.
\end{align}
\end{lemma}
The $\|\cdot\|_{L^2_{v,D}}$ norm defined in \eqref{defLvD} has the following property from Proposition 2.2 in \cite{AMUXY-KRM-2013}.
\begin{lemma}
	Assume $0<s<1$ and $\ga>-3$, then there exist two constants $c,C>0$ such that
	\begin{align}\label{boundnorm}
		c\{\Vert g \Vert^2_{H^s_{\ga/2}}+\Vert g \Vert^2_{L^2_{s+\ga/2}}\}\leq\Vert g \Vert^2_{L^2_{v,D}}\leq C \Vert g \Vert^2_{H^s_{s+\ga/2}}.
	\end{align}
\end{lemma}
One also sees from \eqref{reG} that it is necessary to control the operator $L_M^{-1}$.

\begin{lemma}\label{leL-1}
Let $0<s<1$, $\gamma>\{-3,-3/2-2s\}$, $\be$ be any multi-index and $k\geq 0$, then there are constants $l$ and $C_{\rho,u,\theta}$ such that for $f\in (\ker L_M)^\perp$, it holds that
\begin{align}\label{LM-1}
	\|w_k\mu^{-\frac{1}{2}}L_M^{-1}f\|_{L^2_{v,D}}+\|w_kM^{-\frac{1}{2}}L_M^{-1}f\|_{L^2_{v,D}}\leq C_{\rho,u,\theta}\|w_k M^{-\frac{1}{2}}f\|_{H^{-s}_{-\ga/2}},
\end{align}
and
\begin{align}\label{highL-1}
	\|w_k\pa^\be_v(M^{-\frac{1}{2}}L_M^{-1}f)\|_{L^2_{v,D}}\leq C_{\rho,u,\theta}\sum_{|\be_1|\leq|\be|}\|w_{k+l|\be-\be_1|}\pa^{\be_1}_v(M^{-\frac{1}{2}}f)\|_{H^{-s}_{-\ga/2}}.
\end{align}
In particular,
\begin{align}\label{boundbur}
	\|w_k(v)\mu(v)^{-\frac{1}{2}}\pa^\be_vA_{i}(\frac{v-u}{\sqrt{R\theta}})\|_2+\|w_k(v)\mu(v)^{-\frac{1}{2}}\pa^\be_vB_{ij}(\frac{v-u}{\sqrt{R\theta}})\|_2\leq C.
\end{align}
\end{lemma}
\begin{proof}
	Let $g=L_M^{-1}f$. Since $\frac{1}{\sqrt{M}}L_M g=\frac{1}{\sqrt{M}}(Q(M,\sqrt{M}\frac{g}{\sqrt{M}})+Q(\sqrt{M}\frac{g}{\sqrt{M}},M))$, we have $\frac{g}{\sqrt{M}}=\mathbf{P}_1\frac{g}{\sqrt{M}}$. By a simple modification of Theorem 1.1 in \cite{AMUXY-2012-JFA} and Lemma 2.6 in \cite{GS}, it holds that
	\begin{align*}
		-(\frac{1}{\sqrt{M}}L_M g,\frac{g}{\sqrt{M}})\geq c_{\rho,u,\theta}\|\frac{g}{\sqrt{M}}\|_{L^2_{v,D}}^2,
	\end{align*}
and
	\begin{align*}
	-(w_{2k}\frac{1}{\sqrt{M}}L_M g,\frac{g}{\sqrt{M}})\geq c_{\rho,u,\theta}\|w_k\frac{g}{\sqrt{M}}\|_{L^2_{v,D}}^2-C_{\rho,u,\theta}\|\frac{g}{\sqrt{M}}\|^2_{L^2_{\ga/2}}.
\end{align*}
Then it follows from the above two inequalities that
\begin{align*}
	\|w_kM^{-\frac{1}{2}}g\|_{L^2_{v,D}}&\leq C_{\rho,u,\theta}\big|\big(w_{2k}\frac{1}{\sqrt{M}}L_M g,\frac{g}{\sqrt{M}})\big|+\big|\big(\frac{1}{\sqrt{M}}L_M g,\frac{g}{\sqrt{M}})\big|\big)\notag\\
	&\leq C_{\rho,u,\theta}\|w_k \frac{1}{\sqrt{M}}L_M g\|_{H^{-s}_{-\ga/2}}\|w_k \frac{1}{\sqrt{M}}g\|_{L^2_{v,D}},
\end{align*}
which yields
\begin{align*}
	\|w_kM^{-\frac{1}{2}}L_M^{-1}f\|_{L^2_{v,D}}\leq C_{\rho,u,\theta}\|w_k M^{-\frac{1}{2}}f\|_{H^{-s}_{-\ga/2}},
\end{align*}
 by $g=L_M^{-1}f$. Then \eqref{rhoutheta} and the fact that $\mu/M<C$ imply that
 \begin{align*}
 		\|w_k\mu^{-\frac{1}{2}}L_M^{-1}f\|_{L^2_{v,D}}&\leq 	\|w_kM^{-\frac{1}{2}}L_M^{-1}f\|_{L^2_{v,D}}+	\|w_k\frac{\mu^{1/2}-M^{1/2}}{\mu^{1/2}M^{1/2}}L_M^{-1}f\|_{L^2_{v,D}}\notag\\
 		&\leq \|w_kM^{-\frac{1}{2}}L_M^{-1}f\|_{L^2_{v,D}}+C(\ep+\eta+\bar{\eta})	\|w_k\mu^{-\frac{1}{2}}L_M^{-1}f\|_{L^2_{v,D}}.
 \end{align*}
Hence, \eqref{LM-1} follows from the smallness of $\ep,\eta$ and $\bar{\eta}$.
  Similarly, there exists some constant $l$ such that for general $\be$, we have
	\begin{align*}
	-(w_{2k}\pa^\be_v(\frac{1}{\sqrt{M}}L_M g),\pa^\be_v\frac{g}{\sqrt{M}})\geq&  c_{\rho,u,\theta}\|w_k\pa^\be_v\frac{g}{\sqrt{M}}\|_{L^2_{v,D}}^2\notag\\
	&-\eta\sum_{|\be_1|<|\be|}\|w_{k+l|\be-\be_1|}\pa^{\be_1}_v\frac{g}{\sqrt{M}}\|_{L^2_{v,D}}^2-C_{\rho,u,\theta}\|\frac{g}{\sqrt{M}}\|^2_{L^2_{\ga/2}}.
\end{align*}
Then one gets \eqref{highL-1} by induction argument.
\end{proof}
For $\overline{G}$ defined in \eqref{defbarG}, we have the following bound.
\begin{lemma}\label{lebarG}
	Let $k\geq 0$, it holds that
	\begin{equation}\label{boundbarG0}
		\|w_k(\frac{\overline{G}}{\sqrt{\mu}})\|_{2}
		\leq C\varepsilon(|\nabla_{x}\bar{u}|+|\nabla_{x}\bar{\theta}|).
	\end{equation}
Furthermore, for any multi-index $|\alpha|\geq1$, it holds that
	\begin{align}\label{boundbarG}
		&\|w_k\partial^{\alpha}(\frac{\overline{G}}{\sqrt{\mu}})\|_{2}+\|w_k\nabla_{v}\partial^{\alpha}(\frac{\overline{G}}{\sqrt{\mu}})\|_{2}\notag\\\leq& C\varepsilon\{(|\partial^{\alpha}\nabla_{x}\bar{u}|	+|\partial^{\alpha}\nabla_{x}\bar{\theta}|)+\cdot\cdot\cdot+(|\nabla_{x}\bar{u}|+|\nabla_{x}\bar{\theta}|)
		(|\partial^{\alpha}u|+|\partial^{\alpha}\theta|)\}.
	\end{align}
\end{lemma}
\begin{proof}
	We use \eqref{defbur1} and \eqref{defbur2} to write
	\begin{equation}\label{barG}
		\overline{G}=\varepsilon\frac{\sqrt{R}}{\sqrt{\theta}}\sum^{3}_{j=1}
		\frac{\partial\bar{\theta}}{\partial x_{j}}A_{j}(\frac{v-u}{\sqrt{R\theta}})
		+\varepsilon\sum^{3}_{j=1}\sum^{3}_{i=1}\frac{\partial\bar{u}_{j}}{\partial x_{i}}B_{ij}(\frac{v-u}{\sqrt{R\theta}}).
	\end{equation}
Hence, \eqref{boundbarG0} naturally follows from the above equality and \eqref{boundbur}. 
Then one takes derivative to get
\begin{equation*}
	\frac{\partial\overline{G}}{\partial v_{k}}=\varepsilon\frac{\sqrt{R}}{\sqrt{\theta}}\sum^{3}_{j=1}
	\frac{\partial\bar{\theta}}{\partial x_{j}}\partial_{v_{k}}A_{j}(\frac{v-u}{\sqrt{R\theta}})\frac{1}{\sqrt{R\theta}} +\varepsilon\sum^{3}_{i,j=1}\frac{\partial\bar{u}_{j}}{\partial x_{i}}
	\partial_{v_{k}}B_{ij}(\frac{v-u}{\sqrt{R\theta}})\frac{1}{\sqrt{R\theta}},
\end{equation*}
and
\begin{align*}
	\frac{\partial\overline{G}}{\partial x_{k}}&=\varepsilon\Big\{
	\frac{\sqrt{R}}{\sqrt{\theta}}\sum^{3}_{j=1}\frac{\partial^{2}\bar{\theta}}{\partial x_{j}\partial x_{k}}A_{j}(\frac{v-u}{\sqrt{R\theta}})
	-\frac{\sqrt{R}}{2\sqrt{\theta^{3}}}\sum^{3}_{j=1}\frac{\partial\bar{\theta}}{\partial x_{j}}\frac{\partial\theta}{\partial x_{k}}A_{j}(\frac{v-u}{\sqrt{R\theta}})
	\nonumber\\
	&\quad-\frac{\sqrt{R}}{\sqrt{\theta}}\sum^{3}_{j=1}\frac{\partial\bar{\theta}}{\partial x_{j}}\frac{\partial u}{\partial x_{k}}\cdot
	\nabla_{v}A_{j}(\frac{v-u}{\sqrt{R\theta}})\frac{1}{\sqrt{R\theta}}
	-\frac{\sqrt{R}}{\sqrt{\theta}}\sum^{3}_{j=1}
	\frac{\partial\bar{\theta}}{\partial x_{j}}\frac{\partial\theta}{\partial x_{k}}\nabla_{v}A_{j}(\frac{v-u}{\sqrt{R\theta}})
	\cdot\frac{v-u}{2\sqrt{R\theta^{3}}}
	\nonumber\\
	&\quad+\sum^{3}_{i,j=1}\frac{\partial^{2}\bar{u}_{j}}{\partial x_{i}\partial x_{k}}B_{ij}(\frac{v-u}{\sqrt{R\theta}})
	-\sum^{3}_{i,j=1}\frac{\partial\bar{u}_{j}}{\partial x_{i}}\frac{\partial u}{\partial x_{k}}\cdot\nabla_{v}B_{ij}(\frac{v-u}{\sqrt{R\theta}})\frac{1}{\sqrt{R\theta}}
	\nonumber\\
	&\quad-\sum^{3}_{i,j=1}\frac{\partial\bar{u}_{j}}{\partial x_{i}}\frac{\partial\theta}{\partial x_{k}}\nabla_{v}B_{ij}(\frac{v-u}{\sqrt{R\theta}})\cdot\frac{v-u}{2\sqrt{R\theta^{3}}}\Big\},
\end{align*}
which, combining with \eqref{boundbur}, yield \eqref{boundbarG} for $|\al|=1$. The case when $1<|\al|\leq N$ can be proven in a similar way.
\end{proof}

\section{Estimates on zero order fluid quantities}
Due to the lack of zero order dissipation on the fluid part, in our energy estimates, we need to avoid terms like $(u\cdot\nabla_x\widetilde{\rho},\frac{2\bar{\theta}}{3\bar{\rho}^{2}}\widetilde{\rho})$, which is difficult to control in zero order but can be bounded in higher order case, see \eqref{dutrr}. Motivated by \cite{Liu-Yang-Yu}, we use the entropy and entropy flux method.

\begin{lemma}
	It holds that
	\begin{align}\label{zerofluid}
		&\|(\widetilde{\rho},\widetilde{u},\widetilde{\theta})(t)\|^{2}+c\varepsilon\int^{t}_0
		\|\nabla_x(\widetilde{\rho},\widetilde{u},\widetilde{\theta})(s)\|^{2} ds
		\nonumber\\
		\leq&C\|(\widetilde{\rho},\widetilde{u},\widetilde{\theta})(0)\|^{2}+ C \frac{1}{\varepsilon}\int^t_0\|g_1(s)\|^2ds+C\varepsilon^{r/2}\int^t_0\mathcal{D}_{N,k}(s)ds
		+C\varepsilon^{1+r}+C t\eta\varepsilon^{r}.
	\end{align}
\end{lemma}
\begin{proof}
	As in \cite{Duan-Yang-Yu-2022,Liu-Yang-Yu}, define the macroscopic entropy $S$ to be
	\begin{equation*}
		-\frac{3}{2}\rho S:=\int_{\mathbb{R}^{3}}M\log M dv.
	\end{equation*}
	Recalling $R=\frac{2}{3}$, we substitute $M$ in \eqref{defM} into the above definition to get
	\begin{equation*}
		S=-\frac{2}{3}\log\rho+\log(2\pi R\theta)+1,\quad
		p=R\rho\theta=\frac{1}{2\pi e}\rho^{\frac{5}{3}}\exp (S),
	\end{equation*}
	which, by direct calculation, gives
	\begin{align*}
		-\frac{3}{2}\partial_t(\rho S)-\frac{3}{2}\nabla_{x}\cdot(\rho uS)
		+\nabla_{x}\cdot\int_{\mathbb{R}^{3}} vG\log M dv=\int_{\mathbb{R}^{3}} \frac{G v\cdot\nabla_{x} M}{M}dv.
	\end{align*}
	To simplify the above equation, letting $(\cdot,\cdot,\cdot)^{t}$ to be the transpose of a row vector, we denote
	\begin{align*}
		m:=&(m_{0},m_{1},m_{2},m_{3},m_{4})^{t}=(\rho,\rho u_{1},\rho u_{2},\rho u_{3},\rho(\theta+\frac{|u|^{2}}{2}))^{t},
		\\
		n:=&(n_{0},n_{1},n_{2},n_{3},n_{4})^{t}
		=(\rho u,\rho uu_{1}+p\mathbb{I}_{1},\rho uu_{2}+p\mathbb{I}_{2},\rho uu_{3}+p\mathbb{I}_{3},\rho u(\theta+\frac{|u|^{2}}{2})+pu)^{t},
	\end{align*}
	where $\mathbb{I}_{1}=(1,0,0)^{t}$, $\mathbb{I}_{2}=(0,1,0)^{t}$, $\mathbb{I}_{3}=(0,0,1)^{t}$.
	Then it holds that
	\begin{equation}\label{mn}
		\partial_tm+\nabla_{x}\cdot n=\begin{pmatrix}
			0
			\\
			\varepsilon\sum^{3}_{j=1}\partial_{x_{j}}[\mu(\theta)D_{1j}]
			-\int_{\mathbb{R}^{3}} v_{1}(v\cdot\nabla_{x}L^{-1}_{M}\Theta) dv
			\\
			\varepsilon\sum^{3}_{j=1}\partial_{x_{j}}[\mu(\theta)D_{2j}]
			-\int_{\mathbb{R}^{3}} v_{2}(v\cdot\nabla_{x}L^{-1}_{M}\Theta) dv
			\\
			\varepsilon\sum^{3}_{j=1}\partial_{x_{j}}[\mu(\theta)D_{3j}]
			-\int_{\mathbb{R}^{3}} v_{3}(v\cdot\nabla_{x}L^{-1}_{M}\Theta) dv
			\\
			\varepsilon\nabla_{x}\cdot(\kappa(\theta)\nabla_{x}\theta)
			+\varepsilon\nabla_{x}\cdot[\mu(\theta) u\cdot D]-\int_{\mathbb{R}^{3}} \frac{1}{2}|v|^{2} v\cdot\nabla_{x}L^{-1}_{M}\Theta dv
		\end{pmatrix},
	\end{equation}
	where $D=[D_{ij}]_{1\leq i,j\leq 3}$ is given in \eqref{def.vst}. We further define a relative entropy-entropy flux pair $(\eta,q)(t,x)$ around the local Maxwellian to be
	\begin{equation*}
		\left\{
		\begin{array}{rl}
			&\eta(t,x)=\bar{\theta}\{-\frac{3}{2}\rho S+\frac{3}{2}\bar{\rho}\bar{S}+\frac{3}{2}\nabla_{m}(\rho S)|_{m=\bar{m}}(m-\bar{m})\},
			\\
			&q_{j}(t,x)=\bar{\theta}\{-\frac{3}{2}\rho u_{j}S+\frac{3}{2}\bar{\rho}\bar{u}_{j}\bar{S}+\frac{3}{2}\nabla_{m}(\rho S)|_{m=\bar{m}}\cdot(n_{j}-\bar{n}_{j})\},\quad j=1,2,3,
		\end{array} \right.
	\end{equation*}
	where $\bar{S}=-\frac{2}{3}\log\bar{\rho}+\log(2\pi R\bar{\theta})+1$ and $\bar{m}=(\bar{\rho},\bar{\rho}\bar{ u}_{1},\bar{\rho}\bar{ u}_{2},\bar{\rho}\bar{ u}_{3},\bar{\rho}(\bar{\theta}
	+\frac{1}{2}|\bar{u}|^{2}))^{t}$. It is straightforward to verify that
	\begin{equation*}
		\partial_{m_{0}}(\rho S)=S+\frac{|u|^{2}}{2\theta}-\frac{5}{3}, \quad
		\partial_{m_{i}}(\rho S)=-\frac{u_{i}}{\theta}, ~~~i=1,2,3, \quad \partial_{m_{4}}(\rho S)=\frac{1}{\theta},
	\end{equation*}
	which yields
	\begin{equation}\notag
		\left\{
		\begin{array}{rl}
			\eta(t,x)&=\frac{3}{2}\{\rho\theta-\bar{\theta}\rho S+\rho[(\bar{S}-\frac{5}{3})\bar{\theta}
			+\frac{|u-\bar{u}|^{2}}{2}]+\frac{2}{3}\bar{\rho}\bar{\theta}\}
			\\
			&=\rho\bar{\theta}\Psi(\frac{\bar{\rho}}{\rho})+\frac{3}{2}\rho\bar{\theta}\Psi(\frac{\theta}{\bar{\theta}})
			+\frac{3}{4}\rho|u-\bar{u}|^{2},
			\\
			q_{j}(t,x)&=u_{j}\eta(t,x)+(u_{j}-\bar{u}_{j})(\rho\theta-\bar{\rho}\bar{\theta}), \quad j=1,2,3,
		\end{array} \right.
	\end{equation}
	where $\Psi(s)=s-\log s-1$ is a strictly convex function around $s=1$ and hence comparable with $|s-1|^2$. By our a priori assumption \eqref{apriori}, there exists a constant $C>1$ such that
	\begin{equation}\label{equieta}
		C^{-1}|(\widetilde{\rho},\widetilde{u},\widetilde{\theta})|^{2}\leq \eta(t,x)\leq C|(\widetilde{\rho},\widetilde{u},\widetilde{\theta})|^{2}.
	\end{equation}
	Then a direct calculation gives
	\begin{multline}\label{etaq1}
		\partial_{t}\eta(t,x)+\nabla_{x}\cdot q(t,x)
		=\nabla_{[\bar{\rho},\bar{u},\bar{S}]}\eta(t,x)\cdot \partial_t(\bar{\rho},\bar{u},\bar{S})
		+\sum^{3}_{j=1}\nabla_{[\bar{\rho},\bar{u},\bar{S}]}q_{j}(t,x)\cdot \partial_{x_j}(\bar{\rho},\bar{u},\bar{S})\\
		+\bar{\theta}\{-\frac{3}{2}\partial_t(\rho S)-\frac{3}{2}\nabla_{x}\cdot(\rho uS)\}
		+\frac{3}{2}\bar{\theta}\nabla_{m}(\rho S)|_{m=\bar{m}}(\partial_tm+\nabla_{x}\cdot n).
	\end{multline}
	By the definitions of $S$, $\eta(t,x)$, $q(t,x)$ and the equation of $(m,n)$ in \eqref{mn}, one can deduce that 
	\begin{align*}
		&\nabla_{[\bar{\rho},\bar{u},\bar{S}]}\eta(t,x)\cdot \partial_t(\bar{\rho},\bar{u},\bar{S})
		+\sum^{3}_{j=1}\nabla_{[\bar{\rho},\bar{u},\bar{S}]}q_{j}(t,x)\cdot \partial_{x_j}(\bar{\rho},\bar{u},\bar{S})
		\nonumber\\
		=&-\frac{3}{2}\rho\widetilde{u}\cdot(\widetilde{u}\cdot\nabla_{x}\bar{u})
		-\frac{2}{3}\rho\bar{\theta}(\nabla_{x}\cdot\bar{u})\Psi(\frac{\bar{\rho}}{\rho})
		-\rho\bar{\theta}(\nabla_{x}\cdot\bar{u})\Psi(\frac{\theta}{\bar{\theta}})-\frac{3}{2}\rho\nabla_{x}\bar{\theta}\cdot\widetilde{u}(\frac{2}{3}
		\log\frac{\bar{\rho}}{\rho}+\log\frac{\theta}{\bar{\theta}}),\notag\\
		&\bar{\theta}\{-\frac{3}{2}\partial_t(\rho S)-\frac{3}{2}\nabla_{x}\cdot(\rho uS)\}
		=\bar{\theta}\{-\frac{3}{2}\frac{\rho}{\theta}(\partial_t\theta+u\cdot\nabla_{x}\theta)-\rho\nabla_{x}\cdot u\}
		\\
		=&-\frac{3}{2}\frac{\bar{\theta}}{\theta}\varepsilon\Big\{\sum^{3}_{j=1}\partial_{x_j}(\kappa(\theta)\partial_{x_{j}}\theta)+
		\sum^{3}_{i,j=1}(\mu(\theta) \partial_{x_j}u_{i}D_{ij})\Big\}
		\\
		&+\frac{3}{2}\frac{\bar{\theta}}{\theta}\Big\{\int_{\mathbb{R}^{3}}\frac{1}{2}|v|^{2}v\cdot\nabla_{x}
		L^{-1}_{M}\Theta\, dv-\sum^{3}_{i=1}u_{i}\int_{\mathbb{R}^{3}} v_{i}v\cdot\nabla_{x}L^{-1}_{M}\Theta\, dv\Big\},
	\end{align*}
	and
	\begin{align*}
		&\frac{3}{2}\bar{\theta}\nabla_{m}(\rho S)|_{m=\bar{m}}(\partial_tm+\nabla_{x}\cdot n)
		\\
		=&\frac{3}{2}\sum^{3}_{i=1}\bar{u}_{i}\int_{\mathbb{R}^{3}} v_{i}v\cdot\nabla_{x}L^{-1}_{M}\Theta dv
		-\frac{3}{2}\int_{\mathbb{R}^{3}} \frac{1}{2}|v|^{2} v\cdot\nabla_{x}L^{-1}_{M}\Theta dv
		+\frac{3}{2}\varepsilon\nabla_{x}\cdot(\kappa(\theta)\nabla_{x}\theta)
		\\
		&-\frac{3}{2}\varepsilon\sum^{3}_{i,j=1}\bar{u}_{i}\partial_{x_j}[\mu(\theta)D_{ij}]
		+\frac{3}{2}\varepsilon\nabla_{x}\cdot[\mu(\theta) u\cdot D].
	\end{align*}
	Substituting the above three identities into \eqref{etaq1}, we get
	\begin{align}\label{etaq}
		&\partial_{t}\eta(t,x)+\nabla_{x}\cdot q(t,x)
		+\varepsilon\frac{3\bar{\theta}}{2\theta}\mu(\theta)\sum^{3}_{i,j=1} \partial_{x_j}\widetilde{u}_{i}(\partial_{x_j}\widetilde{u}_{i}+\partial_{x_i}\widetilde{u}_{j}-\frac{2}{3}\delta_{ij}\nabla_{x}\cdot \widetilde{u})
		+\varepsilon\frac{3\bar{\theta}}{2\theta^{2}}\kappa(\theta)|\nabla_{x}\widetilde{\theta}|^{2}
		\nonumber\\
		=&\frac{3}{2}\varepsilon\nabla_{x}\cdot[\mu(\theta)\widetilde{u}\cdot D] +\frac{3}{2}\varepsilon\nabla_{x}\cdot(\frac{\widetilde{\theta}}{\theta}\kappa(\theta)\nabla_{x}\theta)
		-\frac{3}{2}\rho\widetilde{u}\cdot(\widetilde{u}\cdot\nabla_{x}\bar{u})
		-\frac{2}{3}\rho\bar{\theta}(\nabla_{x}\cdot\bar{u})\Psi(\frac{\bar{\rho}}{\rho})\nonumber\\
		&
		-\rho\bar{\theta}(\nabla_{x}\cdot\bar{u})\Psi(\frac{\theta}{\bar{\theta}})
		-\frac{3}{2}\rho\nabla_{x}\bar{\theta}\cdot\widetilde{u}(\frac{2}{3} \log\frac{\bar{\rho}}{\rho}+\log\frac{\theta}{\bar{\theta}})
		-\varepsilon\frac{3\bar{\theta}}{2\theta^{2}}\kappa(\theta)
		\nabla_{x}\widetilde{\theta}\cdot\nabla_{x}\bar{\theta} +\varepsilon\frac{3\widetilde{\theta}}{2\theta^{2}}\kappa(\theta)\nabla_{x}\bar{\theta}\cdot\nabla_{x}\theta \nonumber\\
		&         +\varepsilon\frac{3\widetilde{\theta}}{2\theta}\sum^{3}_{i,j=1}[\partial_{x_j}\bar{u}_{i}\mu(\theta)D_{ij}]-\varepsilon\frac{3\bar{\theta}}{2\theta}\sum^{3}_{i,j=1}\mu(\theta) \partial_{x_j}\widetilde{u}_{i}(\partial_{x_j}\bar{u}_{i}+\partial_{x_i}\bar{u}_{j}-\frac{2}{3}\delta_{ij}\nabla_{x}\cdot \bar{u})
		+I_0(x),
	\end{align}
	where 
	\begin{align}\label{defI0}
		I_0(x)
		=&-\frac{3}{2}\nabla_{x}\cdot(\frac{\widetilde{\theta}}{\theta}\int_{\mathbb{R}^{3}}(\frac{1}{2}|v|^{2}-u\cdot v)vL^{-1}_{M}\Theta dv)
		-\frac{3}{2}\nabla_{x}\cdot(\sum^{3}_{i=1}\widetilde{u}_{i}\int_{\mathbb{R}^{3}} v_{i}vL^{-1}_{M}\Theta dv)
		\nonumber\\
		&+\frac{3}{2}\nabla_{x}(\frac{\widetilde{\theta}}{\theta})\cdot\int_{\mathbb{R}^{3}}(\frac{1}{2}|v|^{2}-v\cdot u)vL^{-1}_{M}\Theta dv
		+\frac{3}{2}\frac{\bar{\theta}}{\theta}\sum^{3}_{i=1}\nabla_{x}\widetilde{u}_{i}\cdot\int_{\mathbb{R}^{3}} v_{i}vL^{-1}_{M}\Theta dv
		\nonumber\\
		&-\frac{3}{2}\frac{\widetilde{\theta}}{\theta}\sum^{3}_{i=1}\nabla_{x}\bar{u}_{i}\cdot\int_{\mathbb{R}^{3}} v_{i}vL^{-1}_{M}\Theta dv.
	\end{align}
	Then we should first integrate with respect to $x$ on both side of \eqref{etaq} and estimate each term. The resulting equation after integration is omitted here and we directly start from the estimates. For the corresponding third term on the left hand side in \eqref{etaq}, we integrate by parts to get
	\begin{align*}
		&\varepsilon\int_{\mathbb{R}^{3}}\frac{3\bar{\theta}}{2\theta}\mu(\theta)\sum^{3}_{i,j=1} \partial_{x_j}\widetilde{u}_{i}(\partial_{x_j}\widetilde{u}_{i}+\partial_{x_i}\widetilde{u}_{j}-\frac{2}{3}\delta_{ij}\nabla_{x}\cdot \widetilde{u})dx
		\\
		=&\varepsilon\sum^{3}_{i,j=1}\int_{\mathbb{R}^{3}}\frac{3\bar{\theta}}{2\theta}\mu(\theta) (\partial_{x_j}\widetilde{u}_{i})^2dx
		+
		\varepsilon\int_{\mathbb{R}^{3}}\frac{3\bar{\theta}}{2\theta}\mu(\theta)\frac{1}{3}(\nabla_{x}\cdot \widetilde{u})^2dx
		\\
		&-\varepsilon\sum^{3}_{i,j=1}\int_{\mathbb{R}^{3}}
		\partial_{x_i}(\frac{3\bar{\theta}}{2\theta}\mu(\theta))\partial_{x_j}\widetilde{u}_{i}\widetilde{u}_{j}dx
		+\varepsilon\sum^{3}_{i,j=1}\int_{\mathbb{R}^{3}}
		\partial_{x_j}(\frac{3\bar{\theta}}{2\theta}\mu(\theta)) \partial_{x_i}\widetilde{u}_{i}\widetilde{u}_{j}dx
		\\
		\geq&c\varepsilon\|\nabla_x\widetilde{u}\|^2
		-C\varepsilon\sum^{3}_{i,j=1}\|\nabla_x(\frac{3\bar{\theta}}{2\theta}\mu(\theta))\|_{L^\infty_x}(\|\partial_{x_j}\widetilde{u}_{i}\|\|\widetilde{u}_{j}\|+\|\partial_{x_i}\widetilde{u}_{i}\|\|\widetilde{u}_{j}\|).
	\end{align*}
	Using the smoothness of $\mu(\theta)$, the inequality that $$\varepsilon\|\nabla_{x}\widetilde{\theta}\|_{L^\infty_x}\|\na_x\widetilde{u}\|\|\widetilde{u}\|\leq C\varepsilon\|\nabla^2_{x}\widetilde{\theta}\|^\frac{1}{2}\|\nabla^3_{x}\widetilde{\theta}\|^\frac{1}{2}\|\na_x\widetilde{u}\|\|\widetilde{u}\|\leq C\sqrt{\mathcal{E}_{N,k}(t)}\mathcal{D}_{N,k}(t),$$ \eqref{background} and \eqref{apriori}, we obtain
	\begin{align}\label{1dissu}
		&\varepsilon\int_{\mathbb{R}^{3}}\frac{3\bar{\theta}}{2\theta}\mu(\theta)\sum^{3}_{i,j=1} \partial_{x_j}\widetilde{u}_{i}(\partial_{x_j}\widetilde{u}_{i}+\partial_{x_i}\widetilde{u}_{j}-\frac{2}{3}\delta_{ij}\nabla_{x}\cdot \widetilde{u})dx\notag
		\\
		\geq&c\varepsilon\|\nabla_x\widetilde{u}\|^2
		-C\eta\varepsilon^{1+r}-C\sqrt{\mathcal{E}_{N,k}(t)}\mathcal{D}_{N,k}(t).
	\end{align}
	Moreover, it is straightforward to see
	\begin{align}\label{1disstheta}
		&\varepsilon\int_{\mathbb{R}^{3}}\frac{3\bar{\theta}}{2\theta^{2}}\kappa(\theta)|\nabla_{x}\widetilde{\theta}|^{2}dx
		\geq c\varepsilon\|\nabla_x\widetilde{\theta}\|^2.
	\end{align}
	Then we turn to the right hand side of \eqref{etaq}. Note that the first two terms vanish after integration. By the facts that $\Psi(s)$ is comparable to $|s-1|^2$ and $\log(s)$ is comparable to $|s-1|$ for $s$ which is closed enough to $1$, we use H\"older's inequality, Sobolev imbedding, \eqref{background} and \eqref{apriori} to get
	\begin{align}\label{bar2tilde}
		&\int_{\mathbb{R}^{3}}|\frac{3}{2}\rho\widetilde{u}\cdot(\widetilde{u}\cdot\nabla_{x}\bar{u})
		+\frac{2}{3}\rho\bar{\theta}(\nabla_{x}\cdot\bar{u})\Psi(\frac{\bar{\rho}}{\rho})
		+\rho\bar{\theta}(\nabla_{x}\cdot\bar{u})\Psi(\frac{\theta}{\bar{\theta}})
		+\frac{3}{2}\rho\nabla_{x}\bar{\theta}\cdot\widetilde{u}(\frac{2}{3}
		\ln\frac{\bar{\rho}}{\rho}+\ln\frac{\theta}{\bar{\theta}})| dx
		\nonumber\\
		&\leq C(\|\nabla_{x}\bar{u}\|_{L^{\infty}}+\|\nabla_{x}\bar{\theta}\|_{L^{\infty}})
		\|(\widetilde{\rho},\widetilde{u},\widetilde{\theta})\|^{2}
		\leq C\eta\varepsilon^{r}.
	\end{align} 
	Likewise, using the smoothness and boundedness of $\mu(\theta)$ and $\ka(\theta)$, it holds that
	\begin{align}\label{epbartilde}
		&\int_{\mathbb{R}^{3}}\big\{|\varepsilon\frac{3\bar{\theta}}{2\theta^{2}}\kappa(\theta)
		\nabla_{x}\widetilde{\theta}\cdot\nabla_{x}\bar{\theta}|+
		|\varepsilon\frac{3\widetilde{\theta}}{2\theta^{2}}\kappa(\theta)\nabla_{x}\bar{\theta}\cdot\nabla_{x}\theta|\notag\\
		&\qquad+|\varepsilon\frac{3\widetilde{\theta}}{2\theta}\sum^{3}_{i,j=1}[\partial_{x_j}\bar{u}_{i}\mu(\theta)D_{ij}]|+|\varepsilon\frac{3\bar{\theta}}{2\theta}\sum^{3}_{i,j=1}\mu(\theta)
		\partial_{x_j}\widetilde{u}_{i}(\partial_{x_j}\bar{u}_{i}+\partial_{x_i}\bar{u}_{j}-\frac{2}{3}\delta_{ij}\nabla_{x}\cdot \bar{u})|\big\}dx\notag
		\\
		\leq& C\varepsilon\|\nabla_{x}\widetilde{\theta}\|\|\nabla_{x}\bar{\theta}\|
		+C\varepsilon\|\nabla_{x}\bar{\theta}\|_{L^{\infty}}\|\widetilde{\theta}\|\|\nabla_{x}\theta\|+C\varepsilon\|\nabla_{x}\bar{u}\|_{L^{\infty}}\|\widetilde{\theta}\|\|\nabla_{x}u\|+C\varepsilon\|\nabla_{x}\widetilde{u}\|\|\nabla_{x}\bar{u}\|\notag\\
		\leq& C\eta\varepsilon^{1+r/2}.
	\end{align}
	Combining \eqref{etaq}, \eqref{1dissu}, \eqref{1disstheta}, \eqref{bar2tilde} and \eqref{epbartilde}, we have
	\begin{align}\label{10est}
		&\frac{d}{dt}\int_{\mathbb{R}^3}\eta(t,x)dx
		+c\varepsilon(\|\nabla_x\widetilde{u}\|^2+\|\nabla_x\widetilde{\theta}\|^{2})
		\leq \int_{\mathbb{R}^3}I_0(x)dx+C\eta\varepsilon^{r}+C\sqrt{\mathcal{E}_{N,k}(t)}\mathcal{D}_{N,k}(t).
	\end{align}
	Now we turn to $\int_{\mathbb{R}^3}I_0(x)dx$. Notice that the first two terms in \eqref{defI0} vanish after integration. Using the identity
	\begin{align*}
		&\int_{\mathbb{R}^{3}} (\frac{1}{2}v_{i}|v|^{2}-v_{i}u\cdot v)L^{-1}_{M}\Theta dv=
		\int_{\mathbb{R}^{3}} L^{-1}_{M}\{P_{1}(\frac{1}{2}v_{i}|v|^{2}-v_{i}u\cdot v)M\}\frac{\Theta}{M}dv
		\nonumber\\
		=&\int_{\mathbb{R}^{3}} L^{-1}_{M}\{(R\theta)^{\frac{3}{2}}\hat{A}_{i}(\frac{v-u}{\sqrt{R\theta}})M\}\frac{\Theta}{M}dv
		=(R\theta)^{\frac{3}{2}}\int_{\mathbb{R}^{3}}A_{i}(\frac{v-u}{\sqrt{R\theta}})\frac{\Theta}{M}dv,
	\end{align*}
	for $i=1,2,3$, we rewrite
	\begin{align}\label{reA}
		&\int_{\mathbb{R}^{3}}\frac{3}{2}\nabla_{x}(\frac{\widetilde{\theta}}{\theta})
		\cdot\int_{\mathbb{R}^{3}}(\frac{1}{2}|v|^{2}-v\cdot u)vL^{-1}_{M}\Theta dvdx
		\nonumber\\ 
		=&\sum_{i=1}^{3}\int_{\mathbb{R}^{3}}\Big\{\frac{3}{2}\partial_{x_{i}}(\frac{\widetilde{\theta}}{\theta}) (R\theta)^{\frac{3}{2}}\int_{\mathbb{R}^{3}}A_{i}(\frac{v-u}{\sqrt{R\theta}})\frac{\Theta}{M}dv\Big\}dx\notag\\
		=&\sum_{i=1}^{3}\int_{\mathbb{R}^{3}}\Big\{\frac{3}{2}\partial_{x_{i}}(\frac{\widetilde{\theta}}{\theta}) (R\theta)^{\frac{3}{2}}\int_{\mathbb{R}^{3}}A_{i}(\frac{v-u}{\sqrt{R\theta}})\frac{1}{M}\big[\varepsilon \partial_{t}\overline{G}+\varepsilon P_{1}(\pa_t\sqrt{\mu}g_2)\notag\\
		&\qquad\qquad+\varepsilon P_{1}(v\cdot\nabla_{x}(\overline{G}+\sqrt{\mu}g_2))-P_{1}(\sqrt{\mu}\CL_Bg_1)\notag\\
		&\qquad\qquad-Q(\overline{G}+\sqrt{\mu}g_2,\overline{G}+\sqrt{\mu}g_2)\big]dv\Big\}dx
		.
	\end{align}
	We need to bound the five terms above. For the first one, from \eqref{controltime} and the proof of \eqref{patN-2}, it holds that \begin{equation*}
		\|\partial_t(u,\theta)\|\leq C,\quad \|\partial_t(\bar{\rho},\bar{u},\bar{\theta})\|_{H^{1}}\leq C\eta,
	\end{equation*} which, together with \eqref{euler}, \eqref{macro1}, \eqref{background}, \eqref{apriori} and a similar argument as in \eqref{boundbarG}, gives
	\begin{align}\label{0Theta1}
		&\int_{\mathbb{R}^{3}}\Big\{\frac{3}{2}\partial_{x_{i}}(\frac{\widetilde{\theta}}{\theta})
		(R\theta)^{\frac{3}{2}}\int_{\mathbb{R}^{3}}A_{i}(\frac{v-u}{\sqrt{R\theta}})
		\frac{\varepsilon \partial_t\overline{G}}{M} dv\Big\} dx
		\nonumber\\
		\leq& C\varepsilon^{2}\int_{\mathbb{R}^{3}}|\partial_{x_{i}}(\frac{\widetilde{\theta}}{\theta})|
		\{|(\nabla_{x}\partial_t\bar{u},\nabla_{x}\partial_t\bar{\theta})|
		+|(\nabla_{x}\bar{u},\nabla_{x}\bar{\theta})|\cdot|(\partial_tu,\partial_t\theta)|\} dx
		\nonumber\\
		\leq& C\varepsilon^{2}(\|\partial_{x_{i}}\widetilde{\theta}\|
		+\|\widetilde{\theta}\partial_{x_{i}}\theta\|)
		\{\|(\nabla_{x}\partial_t\bar{u},\nabla_{x}\partial_t\bar{\theta})\|
		+\|(\nabla_{x}\bar{u},\nabla_{x}\bar{\theta})\|_{L^{\infty}}\|(\partial_tu,\partial_t\theta)\|\}
		\nonumber\\
		\leq& C\eta\varepsilon^{2}.
	\end{align}
	For the second term, we integrate by parts to get
	\begin{align}\label{0Theta21}
		&\int_{\mathbb{R}^{3}}\Big\{\frac{3}{2}\partial_{x_{i}}(\frac{\widetilde{\theta}}{\theta})
		(R\theta)^{\frac{3}{2}}\int_{\mathbb{R}^{3}}A_{i}(\frac{v-u}{\sqrt{R\theta}})
		\frac{\varepsilon P_{1}(\pa_t\sqrt{\mu}g_2)}{M} dv\Big\} dx
		\nonumber\\
		=&\frac{d}{dt}\int_{\mathbb{R}^{3}}\int_{\mathbb{R}^{3}}\frac{3}{2}\partial_{x_{i}}(\frac{\widetilde{\theta}}{\theta})
		(R\theta)^{\frac{3}{2}}A_{i}(\frac{v-u}{\sqrt{R\theta}})
		\frac{\varepsilon P_{1}(\sqrt{\mu}g_2)}{M} dvdx
		\nonumber\\
		&-\int_{\mathbb{R}^{3}}\int_{\mathbb{R}^{3}}\partial_t\Big\{\frac{3}{2}\partial_{x_{i}}(\frac{\widetilde{\theta}}{\theta})
		(R\theta)^{\frac{3}{2}}A_{i}(\frac{v-u}{\sqrt{R\theta}})
		\frac{\varepsilon }{M}\Big\}P_{1}(\sqrt{\mu}g_2) dvdx\notag\\
		&+\int_{\mathbb{R}^{3}}\int_{\mathbb{R}^{3}}
		\Big\{\frac{3}{2}\partial_{x_{i}}(\frac{\widetilde{\theta}}{\theta})
		(R\theta)^{\frac{3}{2}}\int_{\mathbb{R}^{3}}A_{i}(\frac{v-u}{\sqrt{R\theta}})\frac{\varepsilon}{M}\sum_{l=0}^{4}( \sqrt{\mu}g_2,\partial_t\frac{\chi_{l}}{M})_{L^2_v}\chi_{l}\Big\}dvdx\notag\\
		&+\int_{\mathbb{R}^{3}}\int_{\mathbb{R}^{3}}
		\Big\{\frac{3}{2}\partial_{x_{i}}(\frac{\widetilde{\theta}}{\theta})
		(R\theta)^{\frac{3}{2}}\int_{\mathbb{R}^{3}}A_{i}(\frac{v-u}{\sqrt{R\theta}})\frac{\varepsilon}{M}\sum_{l=0}^{4}( \sqrt{\mu}g_2,\frac{\chi_{l}}{M})_{L^2_v}\partial_t\chi_{l}\Big\}dvdx,
	\end{align}
	with $\chi_{l}$ defined in \eqref{basis}.
	Using \eqref{boundbur}, one sees
	\begin{align*}
		&-\int_{\mathbb{R}^{3}}\int_{\mathbb{R}^{3}}\partial_t\Big\{\frac{3}{2}\partial_{x_{i}}(\frac{\widetilde{\theta}}{\theta})
		(R\theta)^{\frac{3}{2}}A_{i}(\frac{v-u}{\sqrt{R\theta}})
		\frac{\varepsilon}{M}\Big\}P_{1}(\sqrt{\mu}g_2) dvdx
		\notag\\
		\leq& C\varepsilon\|\partial_{x_{i}}\partial_t(\frac{\widetilde{\theta}}{\theta})\|\|w_{-2}g_2\|
		+C\varepsilon\|\partial_{x_{i}}(\frac{\widetilde{\theta}}{\theta})\|
		\|w_{-2}g_2\|_{L_x^{\infty}L^2_v}\|\partial_t(\rho,u,\theta)\|\notag\\
		\leq&C\varepsilon\|\partial_{x_{i}}\partial_t(\frac{\widetilde{\theta}}{\theta})\|\|w_{-2}g_2\|
		+C\varepsilon\|\partial_{x_{i}}(\frac{\widetilde{\theta}}{\theta})\|
		\|w_{-2}\nabla_{x}g_2\|^\frac{1}{2}\|w_{-2}\nabla^2_{x}g_2\|^\frac{1}{2}\|\partial_t(\rho,u,\theta)\|.
	\end{align*}
	By an inequality which we will prove in \eqref{patN-1} that \begin{align*}
		\|\partial^{\alpha}\partial_t(\widetilde{\rho},\widetilde{u},\widetilde{\theta})\|^{2}\leq 
		C\frac{1}{\varepsilon}\mathcal{D}_{N,k}(t)+C\eta^2\varepsilon^{r}+C\varepsilon^{-1}\mathcal{E}_{N,k}(t)\mathcal{D}_{N,k}(t),
	\end{align*}
	for all $|\al|\leq N-1$, combining with 
	$$
	\|w_{-2}g_2\|+\|w_{-2}\nabla_{x}g_2\|+\|w_{-2}\nabla^2_{x}g_2\|\leq C\min\{\varepsilon^{r/2},\sqrt{\varepsilon}\sqrt{\mathcal{D}_{k,N}(t)}\}
	$$ by \eqref{apriori} and \eqref{DNk1}, one further has
	\begin{align}\label{0Theta22}
		&-\int_{\mathbb{R}^{3}}\int_{\mathbb{R}^{3}}\partial_t\Big\{\frac{3}{2}\partial_{x_{i}}(\frac{\widetilde{\theta}}{\theta})
		(R\theta)^{\frac{3}{2}}A_{i}(\frac{v-u}{\sqrt{R\theta}})
		\frac{\varepsilon}{M}\Big\}P_{1}(\sqrt{\mu}g_2) dvdx
		\notag\\
		\leq& C\eta\varepsilon^{1+r}+C\varepsilon\mathcal{D}_{N,k}(t)+C\sqrt{\mathcal{E}_{N,k}(t)}\mathcal{D}_{N,k}(t).
	\end{align}
	The other terms in \eqref{0Theta21} can be estimated in the similar way. It then follows from \eqref{0Theta21} and \eqref{0Theta22} that
	\begin{align}\label{0Theta2}
		&\int_{\mathbb{R}^{3}}\Big\{\frac{3}{2}\partial_{x_{i}}(\frac{\widetilde{\theta}}{\theta})
		(R\theta)^{\frac{3}{2}}\int_{\mathbb{R}^{3}}A_{i}(\frac{v-u}{\sqrt{R\theta}})
		\frac{\varepsilon P_{1}(\pa_t\sqrt{\mu}g_2)}{M} dv\Big\} dx
		\nonumber\\
		\leq&\frac{d}{dt}\int_{\mathbb{R}^{3}}\int_{\mathbb{R}^{3}}\frac{3}{2}\partial_{x_{i}}(\frac{\widetilde{\theta}}{\theta})
		(R\theta)^{\frac{3}{2}}A_{i}(\frac{v-u}{\sqrt{R\theta}})
		\frac{\varepsilon P_{1}(\sqrt{\mu}g_2)}{M} dvdx\notag\\
		&+C\eta\varepsilon^{1+r}+C\varepsilon\mathcal{D}_{N,k}(t)+C\sqrt{\mathcal{E}_{N,k}(t)}\mathcal{D}_{N,k}(t).
	\end{align}
	Similarly, it holds that
	\begin{align}\label{0Theta3}
		&\int_{\mathbb{R}^{3}}\Big\{\frac{3}{2}\partial_{x_{i}}(\frac{\widetilde{\theta}}{\theta})
		(R\theta)^{\frac{3}{2}}\int_{\mathbb{R}^{3}}A_{i}(\frac{v-u}{\sqrt{R\theta}})
		\frac{\varepsilon P_{1}(v\cdot\nabla_{x}(\overline{G}+\sqrt{\mu}g_2))}{M} dv\Big\} dx
		\leq C\eta\varepsilon^{1+r}+C\varepsilon\mathcal{D}_{N,k}(t).
	\end{align}
	For the fourth term on the right hand side of \eqref{reA}, we use the definition of the bounded operator $\CL_B$ in \eqref{defLB} to get
	\begin{align}\label{0Theta4}
		&\int_{\mathbb{R}^{3}}\Big\{\frac{3}{2}\partial_{x_{i}}(\frac{\widetilde{\theta}}{\theta})
		(R\theta)^{\frac{3}{2}}\int_{\mathbb{R}^{3}}A_{i}(\frac{v-u}{\sqrt{R\theta}})
		\frac{P_{1}(\sqrt{\mu}\CL_Bg_1)}{M} dv\Big\} dx
		\nonumber\\
		\leq&C\|\partial_{x_{i}}(\frac{\widetilde{\theta}}{\theta})\|\|g_1\|\leq C(\|\partial_{x_{i}}\widetilde{\theta}\|+\|\partial_{x_{i}}\bar{\theta}\|_{L^\infty_x}\|\widetilde{\theta}\|)\|g_1\|\nonumber\\
		\leq&\ka \varepsilon\|\partial_{x_{i}}\widetilde{\theta}\|^2+C_\ka \frac{1}{\varepsilon}\|g_1\|^2 +C\eta\varepsilon^r.
	\end{align}
	Likewise, using \eqref{Trilinear}, \eqref{boundbur} and \eqref{boundbarG0}, it holds that
	\begin{align}\label{0Theta5}
		&\big|\int_{\mathbb{R}^{3}}\Big\{\frac{3}{2}\partial_{x_{i}}(\frac{\widetilde{\theta}}{\theta})
		(R\theta)^{\frac{3}{2}}\int_{\mathbb{R}^{3}}A_{i}(\frac{v-u}{\sqrt{R\theta}})
		\frac{Q(\overline{G}+\sqrt{\mu}g_2,\overline{G}+\sqrt{\mu}g_2)}{M}dv\Big\}dx\big|\notag
		\\
		&\leq C\int_{\mathbb{R}^{3}}\|\frac{\overline{G}+\sqrt{\mu}g_2}{\sqrt{\mu}}\|_{L^2_v}
		\|\frac{\overline{G}+\sqrt{\mu}g_2}{\sqrt{\mu}}\|_{L^2_{v,D}}\big|\nabla_{x}(\frac{\widetilde{\theta}}{\theta})\big|dx\notag
		\\
		&\leq C(\|\frac{\overline{G}}{\sqrt{\mu}}\|_{L^\infty_xL^2_v}+\|g_2\|_{L^\infty_xL^2_v})(\|\frac{\overline{G}}{\sqrt{\mu}}\|_{L^2_xL^2_{v,D}}+\|g_2\|_{L^2_xL^2_{v,D}})
		\|\nabla_{x}(\frac{\widetilde{\theta}}{\theta})\|\notag
		\\
		&\leq C\eta\varepsilon^{1+r/2}+C\sqrt{\mathcal{E}_{N,k}(t)}\mathcal{D}_{N,k}(t).
	\end{align}
	The combination of \eqref{reA}, \eqref{0Theta1}, \eqref{0Theta2}, \eqref{0Theta3}, \eqref{0Theta4} and \eqref{0Theta5} gives 
	\begin{align}\label{estA}
		&\int_{\mathbb{R}^{3}}\frac{3}{2}\nabla_{x}(\frac{\widetilde{\theta}}{\theta})
		\cdot\int_{\mathbb{R}^{3}}(\frac{1}{2}|v|^{2}-v\cdot u)vL^{-1}_{M}\Theta dvdx
		\nonumber\\ 
		\leq&\frac{d}{dt}\int_{\mathbb{R}^{3}}\int_{\mathbb{R}^{3}}\frac{3}{2}\nabla_x(\frac{\widetilde{\theta}}{\theta})
		(R\theta)^{\frac{3}{2}}A_{i}(\frac{v-u}{\sqrt{R\theta}})
		\frac{\varepsilon P_{1}(\sqrt{\mu}g_2)}{M} dvdx+ \ka \varepsilon\|\nabla_{x}\widetilde{\theta}\|^2\notag\\
		&+C_\ka \frac{1}{\varepsilon}\|g_1\|^2+ C\eta\varepsilon^{1+r/2}+C\varepsilon\mathcal{D}_{N,k}(t)+C\sqrt{\mathcal{E}_{N,k}(t)}\mathcal{D}_{N,k}(t).
	\end{align}
	For the rest two terms in \eqref{defI0}, using the identity
	\begin{align*}
		\int_{\mathbb{R}^{3}}v_{i}v_{j}L^{-1}_{M}\Theta dv=&
		\int_{\mathbb{R}^{3}} L^{-1}_{M}\{P_{1}(v_{i}v_{j}M)\}\frac{\Theta}{M}dv
		\nonumber\\
		=&\int_{\mathbb{R}^{3}} L^{-1}_{M}\{R\theta\hat{B}_{ij}(\frac{v-u}{\sqrt{R\theta}})M\}\frac{\Theta}{M}dv
		=R\theta\int_{\mathbb{R}^{3}}B_{ij}(\frac{v-u}{\sqrt{R\theta}})\frac{\Theta}{M}dv,
	\end{align*}
	and performing similar calculations to how we get \eqref{estA}, it holds that
	\begin{align}\label{estrestA}
		&\int_{\mathbb{R}^{3}}\big(\frac{3}{2}\frac{\bar{\theta}}{\theta}\sum^{3}_{i=1}\nabla_{x}\widetilde{u}_{i}\cdot\int_{\mathbb{R}^{3}} v_{i}vL^{-1}_{M}\Theta dv-\frac{3}{2}\frac{\widetilde{\theta}}{\theta}\sum^{3}_{i=1}\nabla_{x}\bar{u}_{i}\cdot\int_{\mathbb{R}^{3}} v_{i}vL^{-1}_{M}\Theta dv\big)dx
		\nonumber\\ 
		\leq& \frac{d}{dt}\frac{3}{2}\sum^{3}_{i,j=1}\int_{\mathbb{R}^{3}}\int_{\mathbb{R}^{3}}\partial_{x_j}\widetilde{u}_{i}R\bar{\theta} B_{ij}(\frac{v-u}{\sqrt{R\theta}})
		\frac{\varepsilon P_{1}(\sqrt{\mu}g_2)}{M}dvdx\notag\\
		&-\frac{d}{dt}\frac{3}{2}\sum^{3}_{i,j=1}\int_{\mathbb{R}^{3}}\int_{\mathbb{R}^{3}}\partial_{x_j}\bar{u}_{i}R\widetilde{\theta} B_{ij}(\frac{v-u}{\sqrt{R\theta}})
		\frac{\varepsilon P_{1}(\sqrt{\mu}g_2)}{M}dvdx+\ka \varepsilon\|\nabla_{x}\widetilde{u}\|^2\notag\\
		&+C_\ka \frac{1}{\varepsilon}\|g_1\|^2+ C\eta\varepsilon^{1+r/2}+C\varepsilon\mathcal{D}_{N,k}(t)+C\sqrt{\mathcal{E}_{N,k}(t)}\mathcal{D}_{N,k}(t).
	\end{align}
	Hence, it follows from \eqref{defI0}, \eqref{estA} and \eqref{estrestA} that
	\begin{align}
		\int_{\mathbb{R}^3}I_0(x)dx \leq& \frac{d}{dt}\Big\{\int_{\mathbb{R}^{3}}\int_{\mathbb{R}^{3}}\frac{3}{2}\partial_{x_{i}}(\frac{\widetilde{\theta}}{\theta})
		(R\theta)^{\frac{3}{2}}A_{i}(\frac{v-u}{\sqrt{R\theta}})
		\frac{\varepsilon P_{1}(\sqrt{\mu}g_2)}{M}dvdx\notag\\
		&\qquad+\frac{3}{2}\sum^{3}_{i,j=1}\int_{\mathbb{R}^{3}}\int_{\mathbb{R}^{3}}\partial_{x_j}\widetilde{u}_{i}R\bar{\theta} B_{ij}(\frac{v-u}{\sqrt{R\theta}})
		\frac{\varepsilon P_{1}(\sqrt{\mu}g_2)}{M}dvdx\notag\\
		&\qquad-\frac{3}{2}\sum^{3}_{i,j=1}\int_{\mathbb{R}^{3}}\int_{\mathbb{R}^{3}}\partial_{x_j}\bar{u}_{i}R\widetilde{\theta} B_{ij}(\frac{v-u}{\sqrt{R\theta}})
		\frac{\varepsilon P_{1}(\sqrt{\mu}g_2)}{M}dvdx\Big\}+\ka \varepsilon\|\nabla_{x}(\widetilde{u},\widetilde{\theta})\|^2\notag\\
		&+C_\ka \frac{1}{\varepsilon}\|g_1\|^2+ C\eta\varepsilon^{1+r/2}+C\varepsilon\mathcal{D}_{N,k}(t)+C\sqrt{\mathcal{E}_{N,k}(t)}\mathcal{D}_{N,k}(t),\notag
	\end{align}
	which, combined with \eqref{10est}, further yields
	\begin{align*}
		&\frac{d}{dt}\int_{\mathbb{R}^3}\eta(t,x)dx
		+c\varepsilon(\|\nabla_x\widetilde{u}\|^2+\|\nabla_x\widetilde{\theta}\|^{2})\notag\\
		\leq& \frac{d}{dt}\Big\{\int_{\mathbb{R}^{3}}\int_{\mathbb{R}^{3}}\frac{3}{2}\nabla_x(\frac{\widetilde{\theta}}{\theta})
		(R\theta)^{\frac{3}{2}}A_{i}(\frac{v-u}{\sqrt{R\theta}})
		\frac{\varepsilon P_{1}(\sqrt{\mu}g_2)}{M}dvdx\notag\\
		&\qquad+\frac{3}{2}\sum^{3}_{i,j=1}\int_{\mathbb{R}^{3}}\int_{\mathbb{R}^{3}}\partial_{x_j}\widetilde{u}_{i}R\bar{\theta} B_{ij}(\frac{v-u}{\sqrt{R\theta}})
		\frac{\varepsilon P_{1}(\sqrt{\mu}g_2)}{M}dvdx\notag\\
		&\qquad-\frac{3}{2}\sum^{3}_{i,j=1}\int_{\mathbb{R}^{3}}\int_{\mathbb{R}^{3}}\partial_{x_j}\bar{u}_{i}R\widetilde{\theta} B_{ij}(\frac{v-u}{\sqrt{R\theta}})
		\frac{\varepsilon P_{1}(\sqrt{\mu}g_2)}{M}dvdx\Big\}\notag\\
		&+\ka \varepsilon\|\nabla_{x}(\widetilde{u},\widetilde{\theta})\|^2+C_\ka \frac{1}{\varepsilon}\|g_1\|^2+ C\eta\varepsilon^{1+r/2}+C\varepsilon\mathcal{D}_{N,k}(t)+C\sqrt{\mathcal{E}_{N,k}(t)}\mathcal{D}_{N,k}(t).
	\end{align*}
	Similar arguments as in \eqref{0Theta22} show that 
	\begin{align*}
		&\Big|\int_{\mathbb{R}^{3}}\int_{\mathbb{R}^{3}}\frac{3}{2}\nabla_x(\frac{\widetilde{\theta}}{\theta})
		(R\theta)^{\frac{3}{2}}A_{i}(\frac{v-u}{\sqrt{R\theta}})
		\frac{\varepsilon P_{1}(\sqrt{\mu}g_2)}{M}dvdx\notag\\
		&+\frac{3}{2}\sum^{3}_{i,j=1}\int_{\mathbb{R}^{3}}\int_{\mathbb{R}^{3}}\partial_{x_j}\widetilde{u}_{i}R\bar{\theta} B_{ij}(\frac{v-u}{\sqrt{R\theta}})
		\frac{\varepsilon P_{1}(\sqrt{\mu}g_2)}{M}dvdx\notag\\
		&-\frac{3}{2}\sum^{3}_{i,j=1}\int_{\mathbb{R}^{3}}\int_{\mathbb{R}^{3}}\partial_{x_j}\bar{u}_{i}R\widetilde{\theta} B_{ij}(\frac{v-u}{\sqrt{R\theta}})
		\frac{\varepsilon P_{1}(\sqrt{\mu}g_2)}{M}dvdx\Big|\leq C\varepsilon^{1+r}.
	\end{align*}
	By the two estimates above and choosing $\ka$ to be small enough, it holds that
	\begin{align}\label{0utheta}
		&\int_{\mathbb{R}^3}\eta(t,x)dx
		+c\varepsilon\int^t_0\|\nabla_x(\widetilde{u},\widetilde{\theta})(s)\|^{2}ds\notag\\
		\leq&\int_{\mathbb{R}^3}\eta(0,x)dx+C \frac{1}{\varepsilon}\int^t_0\|g_1(s)\|^2ds+ Ct\eta\varepsilon^{1+r/2}+C\varepsilon^{1+r}\notag\\
		&+C\varepsilon\int^t_0\mathcal{D}_{N,k}(s)ds+C\int^t_0\sqrt{\mathcal{E}_{N,k}(s)}\mathcal{D}_{N,k}(s)ds.
	\end{align}
	Compared to our expected estimate \eqref{zerofluid}, it is still needed to obtain the first order dissipation of $\widetilde{\rho}$. Take the inner product of the second equation of \eqref{macro1} with $\nabla_{x}\widetilde{\rho}$ to get
	\begin{align}\label{diss0rho1}
		(\frac{2\bar{\theta}}{3\bar{\rho}}\nabla_{x}\widetilde{\rho},\nabla_{x}\widetilde{\rho})
		=&-(\partial_{t}\widetilde{u},\nabla_{x}\widetilde{\rho})-(u\cdot\nabla_{x}\widetilde{u}
		+\frac{2}{3}\nabla_{x}\widetilde{\theta}
		,\nabla_{x}\widetilde{\rho})\notag
		\\
		&-(\widetilde{u}\cdot\nabla_{x}\bar{u}+\frac{2}{3}(\frac{\theta}{\rho}-\frac{\bar{\theta}}{\bar{\rho}})\nabla_{x}\rho,\nabla_{x}\widetilde{\rho})
		-(\frac{1}{\rho}\int_{\mathbb{R}^{3}} v\otimes v\cdot\nabla_{x} G\,dv,\nabla_{x}\widetilde{\rho}).
	\end{align}
	For the first term on the right hand side above, it follows from integration by parts and the first equation of \eqref{macro1} that
	\begin{align}\label{diss0rho2}
		-(\partial_{t}\widetilde{u},\nabla_{x}\widetilde{\rho})=&-\frac{d}{dt}(\widetilde{u},\nabla_{x}\widetilde{\rho})-(\nabla_{x}\widetilde{u},
		\partial_t\widetilde{\rho})\notag
		\\
		=&-\frac{d}{dt}(\widetilde{u},\nabla_{x}\widetilde{\rho})+(\nabla_{x}\widetilde{u},
		u\cdot\nabla_{x}\widetilde{\rho}+\bar{\rho}\nabla_{x}\cdot\widetilde{u}
		+\widetilde{u}\cdot\nabla_{x}\bar{\rho}
		+\widetilde{\rho}\nabla_{x}\cdot u),
	\end{align}
	which, combined with Cauchy-Schwarz inequality, Sobolev imbedding, \eqref{background} and \eqref{apriori}, yields
	\begin{align}\label{diss0rho3}
		-(\partial_{t}\widetilde{u},\nabla_{x}\widetilde{\rho})\leq&-\frac{d}{dt}(\widetilde{u},\nabla_{x}\widetilde{\rho})+C\|\nabla_{x}\widetilde{u}\|^2+C\eta\varepsilon^r+C\frac{1}{\varepsilon}\sqrt{\mathcal{E}_{N,k}(t)}\mathcal{D}_{N,k}(t).
	\end{align}
	Likewise, it holds that
	\begin{align}\label{diss0rho4}
		&|(u\cdot\nabla_{x}\widetilde{u}
		+\frac{2}{3}\nabla_{x}\widetilde{\theta}
		,\nabla_{x}\widetilde{\rho})|+|(\widetilde{u}\cdot\nabla_{x}\bar{u}+\frac{2}{3}(\frac{\theta}{\rho}-\frac{\bar{\theta}}{\bar{\rho}})\nabla_{x}\rho,\nabla_{x}\widetilde{\rho})|\notag
		\\
		\leq& \ka\|\nabla_{x}\widetilde{\rho}\|^2+C_\ka\|\nabla_{x}\widetilde{\theta}\|^2
		+C\eta\varepsilon^r+C\frac{1}{\varepsilon}\sqrt{\mathcal{E}_{N,k}(t)}\mathcal{D}_{N,k}(t),
	\end{align}
	for any small $\ka>0$. Furthermore, from \eqref{boundbarG0} and $G=\overline{G}+g=\overline{G}+g_1+\sqrt{\mu}g_2$, we have
	\begin{align}\label{diss0rho5}
		&|(\frac{1}{\rho}\int_{\mathbb{R}^{3}} v\otimes v\cdot\nabla_{x} Gdv,\nabla_{x}\widetilde{\rho})|\notag\\
		\leq &C\ka\|\nabla_{x}\widetilde{\rho}\|^2
		+C_\ka\|w_2\nabla_{x}g_1\|^{2}+C_\ka\|\nabla_{x}g_2\|_{L^2_{v,D}}^{2}
		+C_\ka \eta\varepsilon^{1+r/2}.
	\end{align}
	Combining \eqref{diss0rho1}, \eqref{diss0rho2}, \eqref{diss0rho3}, \eqref{diss0rho4} and \eqref{diss0rho5}, then choosing $\ka$ to be sufficiently small and noticing $1\leq r\leq 2$, one gets
	\begin{align}\label{control0error}
		\frac{d}{dt}\varepsilon(\widetilde{u},\nabla_{x}\widetilde{\rho})+c\varepsilon\|\nabla_{x}\widetilde{\rho}\|^2
		\leq &C\varepsilon(\|\nabla_{x}\widetilde{u}\|^2
		+\|\nabla_{x}\widetilde{\theta}\|^2+\|w_2\nabla_{x}g_1\|^2+\|\nabla_{x}g_2\|_{L^2_xL^2_{v,D}}^{2})\notag\\
		&+C\eta\varepsilon^r+C\sqrt{\mathcal{E}_{N,k}(t)}\mathcal{D}_{N,k}(t).
	\end{align}
	Using the fact that
	$$
	\varepsilon|(\widetilde{u},\nabla_{x}\widetilde{\rho})|\leq C\varepsilon\|\widetilde{u}\|\|\nabla_{x}\widetilde{\rho}\|
	\leq C\varepsilon^{1+r},
	$$
	it follows from \eqref{control0error} that
	\begin{align}\label{diss0rho}
		\varepsilon\int^{t}_0\|\nabla_{x}\widetilde{\rho}(s)\|^2ds
		\leq& C\varepsilon\int^{t}_0
		\|\nabla_x(\widetilde{u},\widetilde{\theta})(s)\|^{2}ds+C\varepsilon\int^t_0\mathcal{D}_{N,k}(s)ds\notag\\&
		+Ct\eta\varepsilon^{r}+C\varepsilon^{1+r}+C\int^t_0\sqrt{\mathcal{E}_{N,k}(s)}\mathcal{D}_{N,k}(s)ds.
	\end{align}
	Then \eqref{zerofluid} holds from our a priori assumption \eqref{apriori} and a suitable linear combination of \eqref{equieta}, \eqref{0utheta} and \eqref{diss0rho}.
\end{proof}

\section{Estimates on zero order microscopic quantities}\label{sec.5}
In order to obtain the final zero order energy estimate, we need to bound the non-fluid quantities $g_1$ and $g_2$. First for $g_1$, we will prove the following lemma.

\begin{lemma}
	For $k\geq 25$, it holds that
	\begin{align}\label{0orderg1}
		&\frac{1}{2}\frac{d}{dt}\|w_{k}g_1\|^{2}+c \frac{1}{\varepsilon}\Vert g_1 \Vert_{L^2_xH^s_{k+\ga/2}}^2
		\leq C_k(\varepsilon^{r/2}+\eta+\bar\eta)\mathcal{D}_{N,k}(t).
	\end{align}
\end{lemma}

\begin{proof}
	It is natural to take the inner product of the equation \eqref{g1} with
	$w_{2k}g_1$ over $\mathbb{R}^{3}\times\mathbb{R}^{3}$, one has
	\begin{align}\label{ig1}
		\frac{1}{2}\frac{d}{dt}\|w_{k}g_1\|^{2}\leq& \frac{1}{\varepsilon}(\CL_D g_1,w_{2k}g_1)+\frac{1}{\varepsilon}(Q(g_1,g_1),w_{2k}g_1)+\frac{1}{\varepsilon}\big\{(Q(\sqrt{\mu}g_2,g_1),w_{2k}g_1)\notag\\
		&
		+(Q(g_1,\sqrt{\mu}g_2),w_{2k}g_1)\big\}
		+\frac{1}{\varepsilon}\big\{(Q(M-\mu,g_1),w_{2k}g_1)
		+(Q(g_1,M-\mu),w_{2k}g_1)\big\}
		\notag\\
		&+\frac{1}{\varepsilon}\big((Q(\overline{G},g_1),w_{2k}g_1)
		+(Q(g_1,\overline{G}),w_{2k}g_1)\big).
	\end{align}
	Substituting $F=\mu+g_1+\sqrt\mu g_2+(M-\mu)+\overline{G}$ into \eqref{leld}, one gets
	\begin{align*}
		&\frac{1}{\varepsilon}(L_D   g_1,  w_{2k} g_1 )   +   \frac{1}{\varepsilon}(Q(g_1+\sqrt\mu g_2+(M-\mu)+\overline{G},   g_1),  w_{2k} g_1 )  
		\notag\\
		\le & - \de\frac{1}{\varepsilon} \Vert g_1 \Vert_{L^2_x H^s_{k+\gamma/2 }}^2
		+C_k\frac{1}{\varepsilon}\Big(\Vert g_1 \Vert_{L^\infty_x L^2_{14}} \Vert g_1 \Vert_{L^2_x   H^s_{ k+\gamma/2  }}\Vert  g_1 \Vert_{L^2_x H^s_{ k+\gamma/2 }}\notag\\
		& +\Vert g_1 \Vert_{L^\infty_x L^2_{14} } \Vert   g_1 \Vert_{L^2_xH^s_{ k+\gamma/2 }}^2
		+\Vert   g_1 \Vert_{L^2_x L^2_{14}} \Vert \sqrt{\mu }g_2 \Vert_{L^\infty_x   H^s_{ k+\gamma/2 }}\Vert   g_1 \Vert_{L^2_x H^s_{ k+\gamma/2 }}\notag\\
		& +\Vert \sqrt{\mu }g_2 \Vert_{L^\infty_x L^2_{14} } \Vert   g_1 \Vert_{L^2_xH^s_{ k+\gamma/2 }}^2+\Vert   g_1 \Vert_{L^2_x L^2_{14}} \Vert M-\mu \Vert_{L^\infty_x   H^s_{ k+\gamma/2}}\Vert   g_1 \Vert_{L^2_x H^s_{ k+\gamma/2 }}\notag\\
		& +\Vert M-\mu \Vert_{L^\infty_x L^2_{14} } \Vert   g_1 \Vert_{L^2_xH^s_{ k+\gamma/2 }}^2+\Vert   g_1 \Vert_{L^2_x L^2_{14}} \Vert \overline{G} \Vert_{L^\infty_x   H^s_{ k+\gamma/2}}\Vert   g_1 \Vert_{L^2_x H^s_{ k+\gamma/2}}\notag\\
		& +\Vert \overline{G} \Vert_{L^\infty_x L^2_{14} } \Vert  g_1 \Vert_{L^2_xH^s_{ k+\gamma/2}}^2\Big).
	\end{align*}
	By the fact that $\|M-\mu\|_{H^s_{ k+\gamma/2}}$ can be bounded by $|(\rho-1,u,\theta-\frac{3}{2})|$ for any $k\in \R$, we further use the equivalence form of $\mathcal{D}_{N,k}$ in \eqref{DNk1}, \eqref{background}, \eqref{apriori}, \eqref{rhoutheta} and \eqref{boundbarG} to get
	\begin{align}\label{LD0order}
		&\frac{1}{\varepsilon}(L_D   g_1,  w_{2k} g_1 )   +   \frac{1}{\varepsilon}(Q(g_1+\sqrt\mu g_2+(M-\mu)+\overline{G},   g_1),  w_{2k} g_1 )  
		\notag\\
		\leq& - \de\frac{1}{\varepsilon} \Vert g_1 \Vert_{L^2_x H^s_{k+\gamma/2 }}^2+ C_k(\varepsilon^{r/2}+\eta+\bar\eta)\mathcal{D}_{N,k}(t)+C_k\sqrt{\mathcal{E}_{N,k}(t)}
		\mathcal{D}_{N,k}(t).
	\end{align}
	For the rest terms in \eqref{ig1}, it follows by \eqref{fgh} that
	\begin{align}\label{0microrest}
		&\frac{1}{\varepsilon}(Q( g_1,\{\sqrt{\mu}g_2+(M-\mu)+\overline{G}\}),  w_{2k} g_1 )\notag\\
		\leq& C_k\frac{1}{\varepsilon}\big( \Vert g_1 \Vert_{L^2_xL^2_{14}} \Vert \{\sqrt{\mu}g_2+(M-\mu)+\overline{G}\} \Vert_{L^\infty_xH^s_{k+\gamma/2+2s }}  \Vert g_1\Vert_{L^2_vH^s_{k+\gamma/2}}\notag\\
		&+  \Vert \{\sqrt{\mu}g_2+(M-\mu)+\overline{G}\} \Vert_{L^\infty_xL^2_{14}}  \Vert g_1 \Vert_{L^2_xH^s_{k+\gamma/2}}  \Vert g_1\Vert_{L^2_vH^s_{k+\gamma/2}}\big)\notag\\
		\leq& - \de\frac{1}{\varepsilon} \Vert g_1 \Vert_{L^2_x H^s_{k+\gamma/2 }}^2+ C_k(\varepsilon^{r/2}+\eta+\bar\eta)\mathcal{D}_{N,k}(t)+C_k\sqrt{\mathcal{E}_{N,k}(t)}
		\mathcal{D}_{N,k}(t).
	\end{align}
	Hence, \eqref{0orderg1} follows from \eqref{LD0order}, \eqref{0microrest} and \eqref{apriori}.
\end{proof}
It remains to obtain the zero order estimate for $g_2$.

\begin{lemma}
	For $k\geq 25$, it holds that
	\begin{align}\label{zerog2}
		&\frac{1}{2}\frac{d}{dt}\|g_2\|^2+c\frac{1}{\varepsilon}\|\mathbf{P}_1g_2\|_{L^2_xL^2_{v,D}}\notag\\ 
		\leq& C_k\frac{1}{\varepsilon}\|g_1\|_{L^2_xH^s_{k+\ga/2}} + C_k\varepsilon\{\|(\nabla_{x}\widetilde{u},\nabla_{x}\widetilde{\theta})\|^{2}+\|\nabla_{x}g_1(t)\|_{L^2_xH^s_{k+\ga/2}}^{2}\}\notag\\&+C_k(\varepsilon^{r/2}+\eta+\bar\eta)\mathcal{D}_{N,k}(t)+C\eta\varepsilon^r.
	\end{align}
\end{lemma}

\begin{proof} 
	We take the inner product of \eqref{g2} with $g_2$ over $\mathbb{R}^3\times\mathbb{R}^{3}$ and use \eqref{coercive} to get
	\begin{align}\label{ip0g2}
		&\frac{1}{2}\frac{d}{dt}\|g_2\|^{2}+
		c \frac{1}{\varepsilon}\|\mathbf{P}_1g_2\|_{L^2_xL^2_{v,D}}^{2}\notag\\
		\leq&\frac{1}{\varepsilon}(\CL_B g_1+\Gamma(\frac{M-\mu}{\sqrt{\mu}},g_2)
		+\Gamma(g_2,\frac{M-\mu}{\sqrt{\mu}}),g_2)+\frac{1}{\varepsilon}(\Gamma(\frac{\overline{G}+\sqrt{\mu}g_2}{\sqrt{\mu}},\frac{\overline{G}+\sqrt{\mu}g_2}{\sqrt{\mu}}),g_2)
		\nonumber\\
		&+(\frac{P_{0}(v\cdot\nabla_{x}(g_1+\sqrt{\mu}g_2))}{\sqrt{\mu}},g_2)
		-(\frac{P_{1}(v\cdot\nabla_{x}\overline{G})}{\sqrt{\mu}},g_2)
		-(\frac{\partial_{t}\overline{G}}{\sqrt{\mu}},g_2)
		\nonumber\\
		&-(\frac{1}{\sqrt{\mu}}P_{1}\{v\cdot(\frac{|v-u|^{2}\nabla_{x}\widetilde{\theta}}{2R\theta^{2}}
		+\frac{(v-u)\cdot\nabla_{x}\widetilde{u}}{R\theta})M\},g_2).
	\end{align}
	To bound each term above, we first use the definition of $\CL_B$ in \eqref{defLB}, \eqref{boundnorm} and \eqref{g2macro} to get
	\begin{align}\label{0Lb}
		\frac{1}{\varepsilon}(\CL_B g_1,g_2)&\leq C_k\frac{1}{\varepsilon}\int_{\R^3}\int_{\R^3}|g_1g_2|dxdv\notag\\
		&\leq C_k\frac{1}{\varepsilon}\|w_{|\ga|/2}g_1\|(\|\mathbf{P}_0g_2\|+\|\mathbf{P}_1g_2\|_{L^2_xL^2_{v,D}})\notag\\
		&\leq C_{k,\ka}\frac{1}{\varepsilon}\|w_{k+\ga/2}g_1\|^2+\ka\frac{1}{\varepsilon}\|\mathbf{P}_1g_2\|^2_{L^2_xL^2_{v,D}},
	\end{align}
	for any $\ka>0$. Then \eqref{boundbarG}, \eqref{Trilinear} and similar argument in \eqref{LD0order} show that
	\begin{align}\label{03g2}
		&\frac{1}{\varepsilon}\big((\Gamma(\frac{M-\mu+\overline{G}+\sqrt{\mu}g_2}{\sqrt{\mu}},g_2),g_2)+(\Gamma(\frac{\overline{G}}{\sqrt{\mu}},\frac{\overline{G}}{\sqrt{\mu}}),g_2)\big)\notag\\
		\leq& C(\varepsilon^{r/2}+\eta+\bar\eta)\mathcal{D}_{N,k}(t)+C_k\sqrt{\mathcal{E}_{N,k}(t)}
		\mathcal{D}_{N,k}(t).
	\end{align}
	Likewise, using \eqref{Trilinears}, we have
	\begin{align}\label{0g2g2}
		&\frac{1}{\varepsilon}(\Gamma(g_2,\frac{M-\mu+\overline{G}}{\sqrt{\mu}}),g_2)\leq C(\varepsilon^{r/2}+\eta+\bar\eta)\mathcal{D}_{N,k}(t)+C_k\sqrt{\mathcal{E}_{N,k}(t)}
		\mathcal{D}_{N,k}(t).
	\end{align}
	Using the definition of $P_0$ in \eqref{defP} and the equivalence form of $\mathcal{D}_{N,k}(t)$ in \eqref{DNk1}, it holds that
	\begin{align}\label{micro0P0}
		(\frac{P_{0}(v\cdot\nabla_{x}(g_1+\sqrt{\mu}g_2))}{\sqrt{\mu}},g_2)&=\int_{\R^3}\int_{\R^3}\mu^{-\frac{1}{2}}\big\{\sum_{i=0}^{4}\langle v\cdot\nabla_{x}(g_1+\sqrt{\mu}g_2),\frac{\chi_{i}}{M}\rangle\chi_{i}\big\}g_2dxdv\notag\\
		&\leq C\int_{\R^3}\int_{\R^3}\mu^{\frac{1}{4}}\|w_{3}\nabla_{x}(g_1+\sqrt{\mu}g_2)\|_{L^2_v}|g_2|dxdv\notag\\
		&\leq C\varepsilon\mathcal{D}_{N,k}(t).		
	\end{align}
	By \eqref{boundbarG0}, \eqref{background}, \eqref{apriori} and similar argument in \eqref{0Theta1}, one has 
	\begin{align}\label{Gg20}
		(\frac{P_{1}(v\cdot\nabla_{x}\overline{G})}{\sqrt{\mu}},g_2)
		+(\frac{\partial_{t}\overline{G}}{\sqrt{\mu}},g_2)\leq C\eta\varepsilon^{1+r/2}.
	\end{align}
	Likewise, 
	\begin{align}\label{thetag20}
		&(\frac{1}{\sqrt{\mu}}P_{1}\{v\cdot(\frac{|v-u|^{2}
			\nabla_{x}\widetilde{\theta}}{2R\theta^{2}}
		+\frac{(v-u)\cdot\nabla_{x}\widetilde{u}}{R\theta})M\},g_2)\notag\\
		\leq& \ka\frac{1}{\varepsilon}(\|w_{k+\ga/2}g_1(t)\|^{2}+\|\mathbf{P}_1g_2\|^2_{L^2_xL^2_{v,D}})+ C_{\ka}\varepsilon\|(\nabla_{x}\widetilde{u},\nabla_{x}\widetilde{\theta})\|^{2}.
	\end{align}
	Hence, noticing $\|w_{k+\ga/2}g_1(t)\|$ is bounded by $\|g_1\|_{L^2_xH^s_{k+\ga/2}}$, then by \eqref{ip0g2}, \eqref{0Lb}, \eqref{03g2}, \eqref{0g2g2}, \eqref{micro0P0}, \eqref{Gg20}, \eqref{thetag20}, \eqref{apriori} and choosing $\ka$ to be small, we deduce \eqref{zerog2}.
\end{proof}

\section{Estimates on lower order fluid quantities}
With all the preparation in the previous sections, we now study the fluid quantities with the effect of $\pa^\al$. Notice in this section we only estimate all the derivatives up to $N-1$ order. The analysis of the highest order, which requires some other approach, is left to another section. Our major task here is to deduce Lemma \ref{lem.macroN-1}. We start with the estimate on the density $\widetilde{\rho}$.
\begin{lemma}\label{lem.rho}
	For $1\leq|\alpha|\leq N-1$ and $0<\ka<1$, it holds that
	\begin{align}\label{energyrho}
		&\frac{1}{2}\frac{d}{dt}\int_{\mathbb{R}^{3}} \frac{2\bar{\theta}}{3\bar{\rho}^{2}}|\partial^{\alpha}\widetilde{\rho}|^{2} dx
		+(\nabla_x\cdot\partial^{\alpha}\widetilde{u},\frac{2\bar{\theta}}{3\bar{\rho}}\partial^{\alpha}\widetilde{\rho})\notag\\
		\leq& C\eta\varepsilon^r+C\ka\varepsilon\|\nabla_x\cdot\partial^{\alpha}\widetilde{u}\|^2+C_\ka\min\{\varepsilon^{-1}\sqrt{\mathcal{E}_{N,k}(t)}\mathcal{D}_{N,k}(t),\eta^{3/2}_0\varepsilon^r\}.
	\end{align}
\end{lemma}
\begin{proof}
	We apply $\partial^{\alpha}$ with $1\leq|\alpha|\leq N-1$ to the first equation  of \eqref{macro2} and take the inner product of the resulting equation with 
	$\frac{2\bar{\theta}}{3\bar{\rho}^{2}}\partial^{\alpha}\widetilde{\rho}$ to get
	\begin{align}
		\label{rho1}
		&\frac{1}{2}\frac{d}{dt}(\partial^{\alpha}\widetilde{\rho},\frac{2\bar{\theta}}{3\bar{\rho}^{2}}\partial^{\alpha}\widetilde{\rho})-\frac{1}{2}(\partial^{\alpha}\widetilde{\rho},(\frac{2\bar{\theta}}{3\bar{\rho}^{2}})_t\partial^{\alpha}\widetilde{\rho})
		+(\bar{\rho}\nabla_x\cdot\partial^{\alpha}\widetilde{u},\frac{2\bar{\theta}}{3\bar{\rho}^{2}}\partial^{\alpha}\widetilde{\rho})
		\notag\\
		&\qquad\qquad\qquad\qquad\qquad+\sum_{\alpha_{1}\leq \alpha, |\alpha_{1}|\geq1}C^{\alpha_{1}}_\alpha(\partial^{\alpha_1}\bar{\rho}\nabla_x\cdot\partial^{\alpha-\alpha_{1}}\widetilde{u},\frac{2\bar{\theta}}{3\bar{\rho}^{2}}\partial^{\alpha}\widetilde{\rho})
		\nonumber\\
		=&-(\partial^{\alpha}(u\cdot\nabla_x\widetilde{\rho}),\frac{2\bar{\theta}}{3\bar{\rho}^{2}}\partial^{\alpha}\widetilde{\rho})
		-(\partial^{\alpha}(\widetilde{u}\cdot\nabla_x\bar{\rho}),\frac{2\bar{\theta}}{3\bar{\rho}^{2}}\partial^{\alpha}\widetilde{\rho})
		-(\partial^{\alpha}(\widetilde{\rho}\nabla_x\cdot u),\frac{2\bar{\theta}}{3\bar{\rho}^{2}}\partial^{\alpha}\widetilde{\rho}).
	\end{align}	
	It follows by Sobolev inequality, \eqref{background} and \eqref{apriori} that
	\begin{align}\label{trbetr}
		|(\partial^{\alpha}\widetilde{\rho},\partial_t(\frac{2\bar{\theta}}{3\bar{\rho}^{2}})\partial^{\alpha}\widetilde{\rho})|\leq C(\|\partial_t\bar{\rho}\|_{L^\infty}+
		\|\partial_t\bar{\theta}\|_{L^\infty})\|\partial^{\alpha}\widetilde{\rho}\|^2\leq C\eta\varepsilon^r.
	\end{align}
	Similarly,
	\begin{align}\label{brtutr}
		\sum_{\alpha_{1}\leq \alpha, |\alpha_{1}|\geq1}C^{\alpha_{1}}_\alpha|(\partial^{\alpha_1}\bar{\rho}\nabla_x\cdot\partial^{\alpha-\alpha_{1}}\widetilde{u},\frac{2\bar{\theta}}{3\bar{\rho}^{2}}\partial^{\alpha}\widetilde{\rho})|\leq C\eta\varepsilon^r.
	\end{align}
	For the first term on the right hand side of \eqref{rho1}, a direct calculation shows that
	\begin{align}\label{dutrr}
		&-(\partial^{\alpha}(u\cdot\nabla_x\widetilde{\rho}),\frac{2\bar{\theta}}{3\bar{\rho}^{2}}\partial^{\alpha}\widetilde{\rho})\notag\\
		=&-
		(u\cdot\nabla_x\partial^{\alpha}\widetilde{\rho},\frac{2\bar{\theta}}{3\bar{\rho}^{2}}\partial^{\alpha}\widetilde{\rho})-\sum_{\alpha_{1}\leq \alpha, |\alpha_{1}|\geq1}C^{\alpha_{1}}_\alpha(\partial^{\alpha_1}u\cdot\nabla_x\partial^{\alpha-\alpha_1}\widetilde{\rho},\frac{2\bar{\theta}}{3\bar{\rho}^{2}}\partial^{\alpha}\widetilde{\rho}).
	\end{align}
	Integrating by parts, one has
	\begin{align*}
		|(u\cdot\nabla_x\partial^{\alpha}\widetilde{\rho},\frac{2\bar{\theta}}{3\bar{\rho}^{2}}\partial^{\alpha}\widetilde{\rho})|&=	|\frac{1}{2}(\partial^{\alpha}\widetilde{\rho},\nabla_x\cdot(u\frac{2\bar{\theta}}{3\bar{\rho}^{2}})\partial^{\alpha}\widetilde{\rho})|
		\\	
		&\leq C(\|\nabla_{x}(\bar{\rho},\bar{u},\bar{\theta})\|_{L^{\infty}}+\|\nabla_{x}\widetilde{u}\|_{L^{\infty}})
		\|\partial^{\alpha}\widetilde{\rho}\|^2.
	\end{align*}
	Using the definition of $\mathcal{E}_{N,k}(t),\mathcal{D}_{N,k}(t)$ in \eqref{ENk} and \eqref{DNk}, and our a priori assumption \eqref{apriori}, it holds that
	\begin{align*}
		\|\nabla_{x}\widetilde{u}\|_{L^{\infty}}\leq C\|\nabla^2_{x}\widetilde{u}\|^{\frac{1}{2}}\|\nabla^3_{x}\widetilde{u}\|^{\frac{1}{2}}
		\leq C\min\{\varepsilon^{-1/2}\sqrt{\mathcal{D}_{N,k}(t)},\eta_0^{1/2}\varepsilon^{(r-1)/2}\},
	\end{align*}
	and
	\begin{align*}
		\|\partial^{\alpha}\widetilde{\rho}\|\leq C\min\{\varepsilon^{-1/2}\sqrt{\mathcal{D}_{N,k}(t)},\eta_0^{1/2}\varepsilon^{r/2}\}.
	\end{align*}
	Then we obtain
	\begin{align}\label{utrr1}
		|(u\cdot\nabla_x\partial^{\alpha}\widetilde{\rho},\frac{2\bar{\theta}}{3\bar{\rho}^{2}}\partial^{\alpha}\widetilde{\rho})|\leq C\eta\varepsilon^r+C\min\{\varepsilon^{-1}\sqrt{\mathcal{E}_{N,k}(t)}
		\mathcal{D}_{N,k}(t),\eta^{3/2}_0\varepsilon^r\}.
	\end{align}
	For the summation part in \eqref{dutrr}, rewrite
	\begin{align}\label{ddutrr}
		&(\partial^{\alpha_1}u\cdot\nabla_x\partial^{\alpha-\alpha_1}\widetilde{\rho},\frac{2\bar{\theta}}{3\bar{\rho}^{2}}\partial^{\alpha}\widetilde{\rho})\notag\\
		=&(\partial^{\alpha_1}\widetilde{u}\cdot\nabla_x\partial^{\alpha-\alpha_1}\widetilde{\rho},\frac{2\bar{\theta}}{3\bar{\rho}^{2}}\partial^{\alpha}\widetilde{\rho})+(\partial^{\alpha_1}\bar{u}\cdot\nabla_x\partial^{\alpha-\alpha_1}\widetilde{\rho},\frac{2\bar{\theta}}{3\bar{\rho}^{2}}\partial^{\alpha}\widetilde{\rho}).
	\end{align}
	We first consider the case $1\leq|\alpha_{1}|\leq |\alpha|-1$, which leads to
	\begin{align}\label{urr1}
		|(\partial^{\alpha_1}\widetilde{u}\cdot\nabla_x\partial^{\alpha-\alpha_1}\widetilde{\rho},\frac{2\bar{\theta}}{3\bar{\rho}^{2}}\partial^{\alpha}\widetilde{\rho})|
		&\leq \|\partial^{\alpha_1}\widetilde{u}\|_{L^{\infty}}\|\nabla_x\partial^{\alpha-\alpha_1}\widetilde{\rho}\|\|\partial^{\alpha}\widetilde{\rho}\|\notag\\
		&
		\leq C\min\{\varepsilon^{-1}\sqrt{\mathcal{E}_{N,k}(t)}
		\mathcal{D}_{N,k}(t),\eta^{3/2}_0\varepsilon^r\}.
	\end{align}
	Then similarly, for $|\alpha_{1}|=|\alpha|$,
	\begin{align}\label{urr2}
		|(\partial^{\alpha_1}\widetilde{u}\cdot\nabla_x\partial^{\alpha-\alpha_1}\widetilde{\rho},\frac{2\bar{\theta}}{3\bar{\rho}^{2}}\partial^{\alpha}\widetilde{\rho})|
		&\leq C\|\partial^{\alpha_1}\widetilde{u}\|\|\nabla_x\partial^{\alpha-\alpha_1}\widetilde{\rho}\|_{L^{\infty}}\|\partial^{\alpha}\widetilde{\rho}\|\notag\\
		&
		\leq C\min\{\varepsilon^{-1}\sqrt{\mathcal{E}_{N,k}(t)}
		\mathcal{D}_{N,k}(t),\eta^{3/2}_0\varepsilon^r\}.
	\end{align}
	It follows from \eqref{urr1} and \eqref{urr2} that
	\begin{align}\label{turr}
		|(\partial^{\alpha_1}\widetilde{u}\cdot\nabla_x\partial^{\alpha-\alpha_1}\widetilde{\rho},\frac{2\bar{\theta}}{3\bar{\rho}^{2}}\partial^{\alpha}\widetilde{\rho})|
		\leq C\min\{\varepsilon^{-1}\sqrt{\mathcal{E}_{N,k}(t)}
		\mathcal{D}_{N,k}(t),\eta^{3/2}_0\varepsilon^r\}.
	\end{align}
	Similar calculations show that
	\begin{align}\label{burr}
		|(\partial^{\alpha_1}\bar{u}\cdot\nabla_x\partial^{\alpha-\alpha_1}\widetilde{\rho},\frac{2\bar{\theta}}{3\bar{\rho}^{2}}\partial^{\alpha}\widetilde{\rho})|\leq C\eta\varepsilon^r.
	\end{align}
	We have from \eqref{dutrr}, \eqref{utrr1}, \eqref{ddutrr}, \eqref{turr} and \eqref{burr} that
	\begin{align}
		\label{urr}
		|(\partial^{\alpha}(u\cdot\nabla_x\widetilde{\rho}),\frac{2\bar{\theta}}{3\bar{\rho}^{2}}\partial^{\alpha}\widetilde{\rho})|
		\leq C\eta\varepsilon^r+C\min\{\varepsilon^{-1}\sqrt{\mathcal{E}_{N,k}(t)}
		\mathcal{D}_{N,k}(t),\eta^{3/2}_0\varepsilon^r\}.
	\end{align} 
	Applying the Sobolev inequality and \eqref{background}, one gets
	\begin{align}
		\label{tubrtr}
		|(\partial^{\alpha}(\widetilde{u}\cdot\nabla_x\bar{\rho}),\frac{2\bar{\theta}}{3\bar{\rho}^{2}}\partial^{\alpha}\widetilde{\rho})|
		\leq C\sum_{\alpha_{1}\leq\alpha}
		\|\nabla_x\partial^{\alpha-\alpha_{1}}\bar{\rho}\|_{L^{\infty}}\|\partial^{\alpha_1}\widetilde{u}\|\|\partial^{\alpha}\widetilde{\rho}\|
		\leq C\eta\varepsilon^r.
	\end{align}
	For the last term on the right hand side of \eqref{rho1}, by a direct calculation, we have
	\begin{align}\label{dtrutr}
		-(\partial^{\alpha}(\widetilde{\rho}\nabla_x\cdot u),\frac{2\bar{\theta}}{3\bar{\rho}^{2}}\partial^{\alpha}\widetilde{\rho})=&
		-(\widetilde{\rho}\nabla_x\cdot \partial^{\alpha}\widetilde{u},\frac{2\bar{\theta}}{3\bar{\rho}^{2}}\partial^{\alpha}\widetilde{\rho})
		-(\partial^{\alpha}(\widetilde{\rho}\nabla_x\cdot \bar{u}),\frac{2\bar{\theta}}{3\bar{\rho}^{2}}\partial^{\alpha}\widetilde{\rho})
		\notag\\
		&-\sum_{\alpha_{1}\leq \alpha, |\alpha_{1}|\geq1}C^{\alpha_{1}}_\alpha
		(\partial^{\alpha_1}\widetilde{\rho}\nabla_x\cdot\partial^{\alpha-\alpha_{1}}\widetilde{u},\frac{2\bar{\theta}}{3\bar{\rho}^{2}}\partial^{\alpha}\widetilde{\rho}).
	\end{align}
	As how we obtain \eqref{tubrtr} and \eqref{turr}, the last two terms on the right hand side above can be bounded as follows:
	\begin{align}\label{trutr1}
		&\sum_{\alpha_{1}\leq \alpha, |\alpha_{1}|\geq1}C^{\alpha_{1}}_\alpha
		|(\partial^{\alpha_1}\widetilde{\rho}\nabla_x\cdot\partial^{\alpha-\alpha_{1}}\widetilde{u},\frac{2\bar{\theta}}{3\bar{\rho}^{2}}\partial^{\alpha}\widetilde{\rho})|+
		|(\partial^{\alpha}(\widetilde{\rho}\nabla_x\cdot \bar{u}),\frac{2\bar{\theta}}{3\bar{\rho}^{2}}\partial^{\alpha}\widetilde{\rho})|\notag\\
		\leq& C\eta\varepsilon^r+C\min\{\varepsilon^{-1}\sqrt{\mathcal{E}_{N,k}(t)}
		\mathcal{D}_{N,k}(t),\eta^{3/2}_0\varepsilon^r\}.
	\end{align}
	What left now is the first term on the right hand side of \eqref{dtrutr}. An application of Sobolev inequality shows that
	\begin{align}\label{trutr2}
		&|(\widetilde{\rho}\nabla_x\cdot \partial^{\alpha}\widetilde{u},\frac{2\bar{\theta}}{3\bar{\rho}^{2}}\partial^{\alpha}\widetilde{\rho})|\notag\\
		\leq& C\|\nabla_x\cdot\partial^{\alpha}\widetilde{u}\|\|\widetilde{\rho}\|_{L^{\infty}}\|\partial^{\alpha}\widetilde{\rho}\|\notag\\
		\leq&
		\min\{C\varepsilon^{-1}\sqrt{\mathcal{E}_{N,k}(t)}\mathcal{D}_{N,k}(t),\ka\varepsilon\|\nabla_x\cdot\partial^{\alpha}\widetilde{u}\|^2+C_\ka\varepsilon^{-1}\|\nabla_{x}\widetilde{\rho}\|\|\nabla^2_{x}\widetilde{\rho}\|\|\partial^{\alpha}\widetilde{\rho}\|^2\}\notag\\
		\leq& \min\{C\varepsilon^{-1}\sqrt{\mathcal{E}_{N,k}(t)}\mathcal{D}_{N,k}(t),\ka\varepsilon\|\nabla_x\cdot\partial^{\alpha}\widetilde{u}\|^2+C_\ka\eta^{3/2}_0\varepsilon^r\}\notag\\
		\leq& \ka\varepsilon\|\nabla_x\cdot\partial^{\alpha}\widetilde{u}\|^2+C_\ka\min\{\varepsilon^{-1}\sqrt{\mathcal{E}_{N,k}(t)}\mathcal{D}_{N,k}(t),\eta^{3/2}_0\varepsilon^r\}.
	\end{align}
	Then the combination of \eqref{dtrutr}, \eqref{trutr1} and \eqref{trutr2} yields
	\begin{align}\label{trutr}
		&|(\partial^{\alpha}(\widetilde{\rho}\nabla_x\cdot u),\frac{2\bar{\theta}}{3\bar{\rho}^{2}}\partial^{\alpha}\widetilde{\rho})|\notag\\
		\leq& C\eta\varepsilon^r+\ka\varepsilon\|\nabla_x\cdot\partial^{\alpha}\widetilde{u}\|^2+C_\ka\min\{\varepsilon^{-1}\sqrt{\mathcal{E}_{N,k}(t)}\mathcal{D}_{N,k}(t),\eta^{3/2}_0\varepsilon^r\}.
	\end{align}
	We finish the derivative estimate \eqref{energyrho} for $\widetilde{\rho}$ by collecting \eqref{rho1}, \eqref{trbetr}, \eqref{brtutr}, \eqref{urr}, \eqref{tubrtr} and \eqref{trutr}.
\end{proof}

For the second equation of \eqref{macro2}, we can get similar estimate to \eqref{energyrho}.
\begin{lemma}\label{lem.u}
	For $1\leq|\alpha|\leq N-1$ and $0<\ka<1$, it holds that
	\begin{align}\label{energyu}
		&\frac{1}{2}\frac{d}{dt}\|\partial^{\alpha}\widetilde{u}\|^{2}+c\varepsilon\|\partial^{\alpha}\nabla_{x}\widetilde{u}\|^{2}+
		(\frac{2\bar{\theta}}{3\bar{\rho}}\partial^{\alpha}\nabla_{x}\widetilde{\rho},\partial^{\alpha}\widetilde{u})+(\frac{2}{3}\partial^{\alpha}\nabla_{x}\widetilde{\theta},\partial^{\alpha}\widetilde{u})
		\nonumber\\
		&+\sum^{3}_{i,j=1}\frac{d}{dt}\int_{\mathbb{R}^{3}}\int_{\mathbb{R}^{3}}
		\Big\{\partial^{\alpha}[\frac{1}{\rho}\partial_{x_j}(R\theta B_{ij}(\frac{v-u}{\sqrt{R\theta}})\frac{\varepsilon P_{1}(\sqrt{\mu}g_2)}{M})]
		\partial^{\alpha}\widetilde{u}_i\Big\} dvdx
		\nonumber\\
		\leq&(\ka+C_\ka\varepsilon^{r/2})\mathcal{D}_{N,k}(t)+C_\ka \frac{1}{\varepsilon}\|w_2g_1\|_{H^{N-1}_xL^2_v}+C_\ka\eta \varepsilon^r\notag\\
		&+C_\ka\min\{\varepsilon^{-1}\sqrt{\mathcal{E}_{N,k}(t)}\mathcal{D}_{N,k}(t),\eta^{3/2}_0\varepsilon^r\}.
	\end{align}
\end{lemma}
\begin{proof}
	Applying $\partial^{\alpha}$ with $1\leq|\alpha|\leq N-1$  and taking the inner product of the resulting equation with 
	$\partial^{\alpha}\widetilde{u}_i$, one has
	\begin{align}\label{u1}
		&\frac{1}{2}\frac{d}{dt}\|\partial^{\alpha}\widetilde{u}_i\|^{2}+
		(\frac{2\bar{\theta}}{3\bar{\rho}}\partial^{\alpha}\widetilde{\rho}_{x_{i}},\partial^{\alpha}\widetilde{u}_i)+\sum_{\alpha_{1}\leq \alpha, |\alpha_{1}|\geq1}C_{\alpha}^{\alpha_{1}}
		(\partial^{\alpha_1}(\frac{2\bar{\theta}}{3\bar{\rho}})\partial^{\alpha-\alpha_1}\widetilde{\rho}_{x_{i}},\partial^{\alpha}\widetilde{u}_i)
		\nonumber\\
		&+
		(\frac{2}{3}\partial^{\alpha}\widetilde{\theta}_{x_{i}},\partial^{\alpha}\widetilde{u}_i)+(\partial^{\alpha}(u\cdot\nabla_{x}\widetilde{u}_{i})+\partial^{\alpha}(\widetilde{u}\cdot\nabla_{x}\bar{u}_{i}),\partial^{\alpha}\widetilde{u}_i)
		+(\partial^{\alpha}[\frac{2}{3}(\frac{\theta}{\rho}-\frac{\bar{\theta}}{\bar{\rho}})\rho_{x_{i}}],\partial^{\alpha}\widetilde{u}_i)
		\nonumber\\
		=&\varepsilon\sum^{3}_{j=1}(\partial^{\alpha}(\frac{1}{\rho}[\mu(\theta)D_{ij}]_{x_{j}}),\partial^{\alpha}\widetilde{u}_i)-(\partial^{\alpha}(\frac{1}{\rho}\int_{\mathbb{R}^{3}} v_{i}v\cdot\nabla_{x}g_1 dv),\partial^{\alpha}\widetilde{u}_i)\notag\\
		&
		-(\partial^{\alpha}(\frac{1}{\rho}\int_{\mathbb{R}^{3}} v_{i}v\cdot\nabla_{x}L^{-1}_{M}\Theta dv),\partial^{\alpha}\widetilde{u}_i).
	\end{align}
	One may expect that the left hand side of the above equation can be estimated in the very similar way to how we obtain \eqref{energyrho}. The major difference is the two terms on the right hand side. Using the arguments as in \eqref{brtutr}, \eqref{urr1}, \eqref{urr2} and \eqref{trutr2}, one arrives at
	\begin{align}\label{ufluid}
		&\sum_{\alpha_{1}\leq \alpha, |\alpha_{1}|\geq1}C_{\alpha}^{\alpha_{1}}|
		(\partial^{\alpha_1}(\frac{2\bar{\theta}}{3\bar{\rho}})\partial^{\alpha-\alpha_1}\widetilde{\rho}_{x_{i}},\partial^{\alpha}\widetilde{u}_i)|+|(\partial^{\alpha}(u\cdot\nabla_{x}\widetilde{u}_{i})+\partial^{\alpha}(\widetilde{u}\cdot\nabla_{x}\bar{u}_{i}),\partial^{\alpha}\widetilde{u}_i)|\notag\\
		&	\qquad+(\partial^{\alpha}[\frac{2}{3}(\frac{\theta}{\rho}-\frac{\bar{\theta}}{\bar{\rho}})\rho_{x_{i}}],\partial^{\alpha}\widetilde{u}_i)
		\notag\\&\leq 
		C\eta\varepsilon^r+\ka\varepsilon\|\nabla_x\cdot\partial^{\alpha}\widetilde{u}\|^2+C_\ka\min\{\varepsilon^{-1}\sqrt{\mathcal{E}_{N,k}(t)}\mathcal{D}_{N,k}(t),\eta^{3/2}_0\varepsilon^r\}.
	\end{align}
	The above inequality includes all we need to control the left hand side of \eqref{u1}. We now focus on the right hand side. It is direct to get
	\begin{align}\label{dtensor}
		&\varepsilon\sum^{3}_{j=1}(\partial^{\alpha}(\frac{1}{\rho}\partial_{x_{j}}[\mu(\theta)D_{ij}]),\partial^{\alpha}\widetilde{u}_i)\notag\\
		=&\varepsilon\sum^{3}_{j=1}(\partial^{\alpha}(\frac{1}{\rho}\partial_{x_{j}}[\mu(\theta)(\partial_{x_{j}}\widetilde{u}_{i}+\partial_{x_{i}}\widetilde{u}_{j}
		-\frac{2}{3}\delta_{ij}\nabla_{x}\cdot\widetilde{u})]),\partial^{\alpha}\widetilde{u}_i)\notag
		\\
		&+\varepsilon\sum^{3}_{j=1}(\partial^{\alpha}(\frac{1}{\rho}\partial_{x_{j}}[\mu(\theta)(\partial_{x_{j}}\bar{u}_{i}+\partial_{x_{i}}\bar{u}_{j}-\frac{2}{3}\delta_{ij}\nabla_{x}\cdot \bar{u})]),\partial^{\alpha}\widetilde{u}_i)\notag\\
		=&I_1+I_2+I_3,
	\end{align}
	where
	\begin{align*}
		I_1&=-\varepsilon\sum^{3}_{j=1}(\partial^{\alpha}[\frac{1}{\rho}\mu(\theta)(\partial_{x_{j}}\widetilde{u}_{i}+\partial_{x_{i}}\widetilde{u}_{j}
		-\frac{2}{3}\delta_{ij}\nabla_{x}\cdot \widetilde{u})],\partial^{\alpha}\partial_{x_{j}}\widetilde{u}_{i}),\notag\\
		I_2&=-\varepsilon\sum^{3}_{j=1}(\partial^{\alpha}[\partial_{x_{j}}(\frac{1}{\rho})\mu(\theta)(\partial_{x_{j}}\widetilde{u}_{i}+\partial_{x_{i}}\widetilde{u}_{j}
		-\frac{2}{3}\delta_{ij}\nabla_{x}\cdot\widetilde{u})],\partial^{\alpha}\widetilde{u}_{i}),\notag\\
		I_3&=\varepsilon\sum^{3}_{j=1}(\partial^{\alpha}(\frac{1}{\rho}\partial_{x_{j}}[\mu(\theta)(\partial_{x_{j}}\bar{u}_{i}+\partial_{x_{i}}\bar{u}_{j}-\frac{2}{3}\delta_{ij}\nabla_{x}\cdot \bar{u})]),\partial^{\alpha}\widetilde{u}_i).
	\end{align*}
	Furthermore, $I_1$ can be decomposed into
	\begin{align}\label{defI1}
		I_1=&-\varepsilon\sum^{3}_{j=1}(\frac{1}{\rho}\mu(\theta)(\partial^{\alpha}\partial_{x_{j}}\widetilde{u}_{i}+\partial^{\alpha}\partial_{x_{i}}\widetilde{u}_{j}
		-\frac{2}{3}\delta_{ij}\nabla_{x}\cdot \partial^{\alpha}\widetilde{u}),\partial^{\alpha}\partial_{x_{j}}\widetilde{u}_{i})\notag
		\\
		&-\varepsilon\sum^{3}_{j=1}\sum_{\alpha_{1}\leq\alpha,|\alpha_{1}|\geq1}C_{\alpha}^{\alpha_{1}}(\partial^{\alpha_1}[\frac{1}{\rho}\mu(\theta)]
		\partial^{\alpha-\alpha_1}(\partial_{x_{j}}\widetilde{u}_{i}+\partial_{x_{i}}\widetilde{u}_{j}-\frac{2}{3}\delta_{ij}\nabla_{x}\cdot\widetilde{u}),\partial^{\alpha}\partial_{x_{j}}\widetilde{u}_{i}).
	\end{align}
	The first term on the right hand side will be handled later. For the second term, applying the Sobolev inequality with $1\leq|\alpha_{1}|\leq |\alpha|-1$, one has
	\begin{align}\label{I11}
		&\varepsilon|(\partial^{\alpha_1}[\frac{1}{\rho}\mu(\theta)]\partial^{\alpha-\alpha_1}
		(\partial_{x_{j}}\widetilde{u}_{i}+\partial_{x_{i}}\widetilde{u}_{j}-\frac{2}{3}\delta_{ij}\nabla_{x}\cdot\widetilde{u}),\partial^{\alpha}\partial_{x_{j}}\widetilde{u}_{i})|\notag
		\\
		\leq& C\varepsilon \|\partial^{\alpha_1}[\frac{1}{\rho}\mu(\theta)]\|_{L^{\infty}}
		\|\nabla_x\partial^{\alpha-\alpha_1}\widetilde{u}\|\|\partial^{\alpha}\partial_{x_{j}}\widetilde{u}_{i}\|\notag
		\\
		\leq& C\varepsilon\|\{|\partial^{\alpha_1}(\rho,\theta)|+\cdot\cdot\cdot+|\nabla_{x}(\rho,\theta)|^{|\alpha_1|}\}\|_{L^{\infty}}
		\|\nabla_x\partial^{\alpha-\alpha_1}\widetilde{u}\|\|\partial^{\alpha}\partial_{x_{j}}\widetilde{u}_{i}\|\notag
		\\
		\leq& C\eta\varepsilon^{r}+C\sqrt{\mathcal{E}_{N,k}(t)}\mathcal{D}_{N,k}(t).
	\end{align}
	For the other case $|\alpha_{1}|=|\alpha|$, we have
	\begin{align}\label{I12}
		&\varepsilon|(\partial^{\alpha_1}[\frac{1}{\rho}\mu(\theta)]\partial^{\alpha-\alpha_1}(\partial_{x_{j}}\widetilde{u}_{i}+\partial_{x_{i}}\widetilde{u}_{j}-\frac{2}{3}\delta_{ij}\nabla_{x}\cdot\widetilde{u}),\partial^{\alpha}\partial_{x_{j}}\widetilde{u}_{i})|\notag
		\\
		\leq& C\varepsilon\|\partial^{\alpha_1}[\frac{1}{\rho}\mu(\theta)]\|_{L^{3}}
		\|\nabla_x\partial^{\alpha-\alpha_1}\widetilde{u}\|_{L^{6}}\|\partial^{\alpha}\partial_{x_{j}}\widetilde{u}_{i}\|_{L^{2}}\notag
		\\
		\leq& C\eta\varepsilon^r+C\sqrt{\mathcal{E}_{N,k}(t)}\mathcal{D}_{N,k}(t).
	\end{align}
	Combining \eqref{defI1}, \eqref{I11} and \eqref{I12}, it holds that
	\begin{align}\label{I1}
		I_1\leq&-\varepsilon\sum^{3}_{j=1}(\frac{1}{\rho}\mu(\theta)(\partial^{\alpha}\partial_{x_{j}}\widetilde{u}_{i}+\partial^{\alpha}\partial_{x_{i}}\widetilde{u}_{j}
		-\frac{2}{3}\delta_{ij}\nabla_{x}\cdot \partial^{\alpha}\widetilde{u}),\partial^{\alpha}\partial_{x_{j}}\widetilde{u}_{i})\notag
		\\
		&+C\eta\varepsilon^r+C\sqrt{\mathcal{E}_{N,k}(t)}\mathcal{D}_{N,k}(t).
	\end{align}
	Since $I_2$ will not produce higher order derivatives than $I_1$, it is easier to estimate. By similar arguments as above, we have
	\begin{align}\label{I2}
		I_2\leq&C\eta\varepsilon^r+C\sqrt{\mathcal{E}_{N,k}(t)}\mathcal{D}_{N,k}(t).
	\end{align}
	Decomposing $I_3$ as in \eqref{defI1}, and using \eqref{background} and the fact that
	\begin{align*}
		&\varepsilon|(\frac{1}{\rho}\mu(\theta)\partial^{\alpha}\partial_{x_{j}}(\partial_{x_{j}}\bar{u}_{i}+\partial_{x_{i}}\bar{u}_{j}
		-\frac{2}{3}\delta_{ij}\nabla_{x}\cdot \partial^{\alpha}\bar{u}),\widetilde{u}_{i})|\leq C\eta\varepsilon^{1+r/2}\leq C\eta\varepsilon^{r},
	\end{align*} one arrives at
	\begin{align}\label{I3}
		I_3\leq&C\eta\varepsilon^r.
	\end{align}
	It follows from \eqref{dtensor}, \eqref{I1}, \eqref{I2} and \eqref{I3} that
	\begin{align}
		\label{tensor}
		&\varepsilon\sum^{3}_{j=1}(\partial^{\alpha}(\frac{1}{\rho}\partial_{x_{j}}[\mu(\theta)D_{ij}]),\partial^{\alpha}\widetilde{u}_i)
		\nonumber\\
		\leq&
		-\varepsilon\sum^{3}_{j=1}(\frac{1}{\rho}\mu(\theta)(\partial^{\alpha}\partial_{x_{j}}\widetilde{u}_{i}
		+\partial^{\alpha}\partial_{x_{i}}\widetilde{u}_{j}
		-\frac{2}{3}\delta_{ij}\nabla_{x}\cdot \partial^{\alpha}\widetilde{u}),\partial^{\alpha}\partial_{x_{j}}\widetilde{u}_{i})\notag\\
		&+C\eta\varepsilon^r+C\sqrt{\mathcal{E}_{N,k}(t)}\mathcal{D}_{N,k}(t).
	\end{align}
	For the second term on the right hand side of \eqref{u1}, by the fact that there exists some $e_j$ with $\al=\al-e_j+e_j$ and $|e_j|=1$, we use integration by parts and \eqref{background} to get
	\begin{align}\label{ug1}
		&|(\partial^{\alpha}(\frac{1}{\rho}\int_{\mathbb{R}^{3}} v_{i}v\cdot\nabla_{x}g_1 dv),\partial^{\alpha}\widetilde{u}_i)|\notag\\=&|(\partial^{\alpha-e_j}(\frac{1}{\rho}\int_{\mathbb{R}^{3}} v_{i}v\cdot\nabla_{x}g_1 dv),\partial^{\alpha+e_j}\widetilde{u}_i)|\notag\\
		=&\sum_{ \alpha_{1}\leq\alpha-e_j}C_{\alpha}^{\alpha_{1}}|(\partial^{\alpha_1}(\frac{1}{\rho})\partial^{\alpha-e_j-\alpha_{1}}\int_{\mathbb{R}^{3}} v_{i}v\cdot\nabla_{x}g_1 dv,\partial^{\alpha+e_j}\widetilde{u}_i)|\notag\\
		\leq& C(\ka+\eta)\mathcal{D}_{N,k}(t)+C_\ka \frac{1}{\varepsilon}\|w_2g_1\|_{H^{N-1}_xL^2_v}+C\sqrt{\mathcal{E}_{N,k}(t)}\mathcal{D}_{N,k}(t).
	\end{align}
	What is left now is the last term on the right hand side of \eqref{u1}, we integrate by parts to get
	\begin{align}\label{reTheta}
		&-(\partial^{\alpha}[\frac{1}{\rho}\int_{\mathbb{R}^{3}}v_{i}v\cdot\nabla_{x}L^{-1}_{M}\Theta dv],\partial^{\alpha}\widetilde{u}_i)
		\nonumber\\
		=&-\sum^{3}_{j=1}(\partial^{\alpha}[\frac{1}{\rho}\partial_{x_j}(\int_{\mathbb{R}^{3}}R\theta B_{ij}(\frac{v-u}{\sqrt{R\theta}})\frac{\Theta}{M}\, dv)],\partial^{\alpha}\widetilde{u}_i)
		\nonumber\\
		=&I_4+I_5,
	\end{align}
	where 
	$$
	I_4=\sum^{3}_{j=1}(\partial^{\alpha}[\partial_{x_j}(\frac{1}{\rho})\int_{\mathbb{R}^{3}}R\theta B_{ij}(\frac{v-u}{\sqrt{R\theta}})\frac{\Theta}{M}\, dv],\partial^{\alpha}\widetilde{u}_i),
	$$
	and 
	$$
	I_5=\sum^{3}_{j=1}(\partial^{\alpha}[\frac{1}{\rho}\int_{\mathbb{R}^{3}}R\theta B_{ij}(\frac{v-u}{\sqrt{R\theta}})\frac{\Theta}{M}\, dv],\partial^{\alpha}\partial_{x_j}\widetilde{u}_i).
	$$
	Substituting \eqref{defTheta} into $I_4$, it holds that
	\begin{align}\label{reI4}
		I_4&=\sum^{3}_{j=1}(\partial^{\alpha}[\partial_{x_j}(\frac{1}{\rho})\int_{\mathbb{R}^{3}}R\theta B_{ij}(\frac{v-u}{\sqrt{R\theta}})\frac{1}{M}\big\{\varepsilon \partial_{t}\overline{G}+\varepsilon P_{1}(\pa_t\sqrt{\mu}g_2)\notag\\
		&\quad+\varepsilon P_{1}(v\cdot\nabla_{x}(\overline{G}+\sqrt{\mu}g_2))-P_{1}(\sqrt{\mu}\CL_Bg_1)-Q(\overline{G}+\sqrt{\mu}g_2,\overline{G}+\sqrt{\mu}g_2)\big\}\, dv],\partial^{\alpha}\widetilde{u}_i)\notag\\
		&=I_{41}+I_{42}+I_{43}+I_{44}+I_{45}.
	\end{align}
	For $I_{41}$, since the time derivative is involved, we use the second equation of \eqref{macro1} to get
	\begin{align}\label{controltime}
		(\partial^{\alpha}\partial_t\widetilde{u},\partial^{\alpha}\partial_t\widetilde{u})&=
		-(\partial^{\alpha}[u\cdot\nabla_{x}\widetilde{u}
		+\frac{2\bar{\theta}}{3\bar{\rho}}\nabla_{x}\widetilde{\rho}
		+\frac{2}{3}\nabla_{x}\widetilde{\theta}
		+\widetilde{u}\cdot\nabla_{x}\bar{u}+\frac{2}{3}(\frac{\theta}{\rho}-\frac{\bar{\theta}}{\bar{\rho}})
		\nabla_{x}\rho],\partial^{\alpha}\partial_t\widetilde{u})
		\nonumber\\
		&\quad-(\partial^{\alpha}(\frac{1}{\rho}\int_{\mathbb{R}^{3}} v\otimes v\cdot\nabla_{x} G dv),\partial^{\alpha}\partial_t\widetilde{u})
		\nonumber\\
		&\leq C\ka\|\partial^{\alpha}\partial_t\widetilde{u}\|^{2}+
		C_\ka\|\partial^{\alpha}(\nabla_{x}\widetilde{\rho},\nabla_{x}\widetilde{u},\nabla_{x}\widetilde{\theta})\|^{2}+C_\ka \|w_2\partial^{\alpha}\nabla_{x} (g_1,g_2)\|^{2}\notag\\
		&\quad+C_\ka\eta^2\varepsilon^{r} +C_\ka\min\{\varepsilon^{-1}\mathcal{E}_{N,k}(t)\mathcal{D}_{N,k}(t),\varepsilon^{-2}\mathcal{E}^2_{N,k}(t)\},
	\end{align}
	for any $0<\ka<1$, which yields, by choosing $\ka$ to be small, that
	\begin{align*}
		\|\partial^{\alpha}\partial_t\widetilde{u}\|^{2}\leq 
		C\mathcal{E}_{N,k}(t)+C\eta^2\varepsilon^{r}+C\min\{\varepsilon^{-1}\mathcal{E}_{N,k}(t)\mathcal{D}_{N,k}(t),\varepsilon^{-2}\mathcal{E}^2_{N,k}(t)\},
	\end{align*}
	for $|\al|\leq N-2$, and
	\begin{align*}
		\|\partial^{\alpha}\partial_t\widetilde{u}\|^{2}\leq 
		C\frac{1}{\varepsilon}\mathcal{D}_{N,k}(t)+C\eta^2\varepsilon^{r}+C\varepsilon^{-1}\mathcal{E}_{N,k}(t)\mathcal{D}_{N,k}(t).
	\end{align*}
	for $|\al|\leq N-1$.
	Similar arguments show that
	\begin{align}\label{patN-2}
		\|\partial^{\alpha}\partial_t(\widetilde{\rho},\widetilde{u},\widetilde{\theta})\|^{2}&\leq 
		C\mathcal{E}_{N,k}(t)+C\eta^2\varepsilon^{r}+C\min\{\varepsilon^{-1}\mathcal{E}_{N,k}(t)\mathcal{D}_{N,k}(t),\varepsilon^{-2}\mathcal{E}^2_{N,k}(t)\}\notag\\
		&\leq C\min\{1,\mathcal{E}_{N,k}(t)+\eta^2\varepsilon^{r}+\varepsilon^{-1}\mathcal{E}_{N,k}(t)\mathcal{D}_{N,k}(t)\},
	\end{align}
	for $|\al|\leq N-2$, and
	\begin{align}\label{patN-1}
		\|\partial^{\alpha}\partial_t(\widetilde{\rho},\widetilde{u},\widetilde{\theta})\|^{2}\leq 
		C\frac{1}{\varepsilon}\mathcal{D}_{N,k}(t)+C\eta^2\varepsilon^{r}+C\varepsilon^{-1}\mathcal{E}_{N,k}(t)\mathcal{D}_{N,k}(t),
	\end{align}
	for all $|\al|\leq N-1$.
	Then for $I_{41}$, it holds from \eqref{background}, \eqref{apriori}, \eqref{barG}, \eqref{patN-2}, the Cauchy-Schwarz and Sobolev inequalities that
	\begin{align}\label{I41}
		I_{41}&=\sum^{3}_{j=1}(\partial^{\alpha}[\partial_{x_j}(\frac{1}{\rho})\int_{\mathbb{R}^{3}}R\theta B_{ij}(\frac{v-u}{\sqrt{R\theta}})\frac{\varepsilon \partial_{t}\overline{G}}{M} dv],\partial^{\alpha}\widetilde{u}_i)\notag\\
		&\leq C\eta
		\varepsilon^r.
	\end{align}
	Recall 
	\begin{align*}
		I_{42}=&\sum^{3}_{j=1}\int_{\mathbb{R}^{3}}\Big\{\partial^{\alpha}[\partial_{x_j}(\frac{1}{\rho})\int_{\mathbb{R}^{3}}R\theta B_{ij}(\frac{v-u}{\sqrt{R\theta}})\frac{\varepsilon P_{1}(\pa_t\sqrt{\mu}g_2)}{M}\, dv]\,\partial^{\alpha}\widetilde{u}_i\Big\} dx.
	\end{align*}
	We integrate by parts to get
	\begin{align*}
		I_{42}=&\sum^{3}_{j=1}\frac{d}{dt}\int_{\mathbb{R}^{3}}\int_{\mathbb{R}^{3}}\partial^{\alpha}
		\Big\{[\partial_{x_j}(\frac{1}{\rho})R\theta B_{ij}(\frac{v-u}{\sqrt{R\theta}})\frac{\varepsilon}{M}]P_{1}(\sqrt{\mu}g_2)\Big\}\partial^{\alpha}\widetilde{u}_i dvdx
		\nonumber\\
		&-\sum^{3}_{j=1}\int_{\mathbb{R}^{3}}\int_{\mathbb{R}^{3}}\partial^{\alpha}
		\Big\{\partial_t[\partial_{x_j}(\frac{1}{\rho})R\theta B_{ij}(\frac{v-u}{\sqrt{R\theta}})\frac{\varepsilon}{M}]P_{1}(\sqrt{\mu}g_2)\Big\}\partial^{\alpha}\widetilde{u}_i dvdx
		\nonumber\\
		&-\sum^{3}_{j=1}\int_{\mathbb{R}^{3}}\int_{\mathbb{R}^{3}}\partial^{\alpha}
		\Big\{[\partial_{x_j}(\frac{1}{\rho})R\theta B_{ij}(\frac{v-u}{\sqrt{R\theta}})\frac{\varepsilon}{M}]P_{1}(\sqrt{\mu}g_2)\Big\}\partial^{\alpha}\partial_t\widetilde{u}_{i} dvdx\notag\\
		&+\sum^{3}_{j=1}\int_{\mathbb{R}^{3}}\int_{\mathbb{R}^{3}}\partial^{\alpha}
		\Big\{[\partial_{x_j}(\frac{1}{\rho})R\theta B_{ij}(\frac{v-u}{\sqrt{R\theta}})\frac{\varepsilon}{M}]\sum_{l=0}^{4}( \sqrt{\mu}g_2,\partial_t\frac{\chi_{l}}{M})_{L^2_v}\chi_{l}\Big\}\partial^{\alpha}\widetilde{u}_{i} dvdx\notag\\
		&+\sum^{3}_{j=1}\int_{\mathbb{R}^{3}}\int_{\mathbb{R}^{3}}\partial^{\alpha}
		\Big\{[\partial_{x_j}(\frac{1}{\rho})R\theta B_{ij}(\frac{v-u}{\sqrt{R\theta}})\frac{\varepsilon}{M}]\sum_{l=0}^{4}( \sqrt{\mu}g_2,\frac{\chi_{l}}{M})_{L^2_v}\partial_t\chi_{l}\Big\}\partial^{\alpha}\widetilde{u}_{i} dvdx.
	\end{align*}
	We keep the first term on the right hand side above there and turn to the second one. Since $|\al|\geq 1$, there is some $e_i$ such that $\partial^{e_i}=\partial_{x_i}$ with $|e_i|=1$. Integrating by parts again, it holds that
	\begin{align*}
		&-\sum^{3}_{j=1}\int_{\mathbb{R}^{3}}\int_{\mathbb{R}^{3}}\partial^{\alpha}
		\Big\{\partial_t[\partial_{x_j}(\frac{1}{\rho})R\theta B_{ij}(\frac{v-u}{\sqrt{R\theta}})\frac{\varepsilon}{M}]P_{1}(\sqrt{\mu}g_2)\Big\}\partial^{\alpha}\widetilde{u}_i dvdx
		\\
		=&\sum^{3}_{j=1}\int_{\mathbb{R}^{3}}\int_{\mathbb{R}^{3}}\partial^{\alpha-e_i}
		\Big\{\partial_t[\partial_{x_j}(\frac{1}{\rho})R\theta B_{ij}(\frac{v-u}{\sqrt{R\theta}})\frac{\varepsilon}{M}]P_{1}(\sqrt{\mu}g_2)\Big\}\partial^{\alpha+e_i}\widetilde{u}_i dvdx
		\\
		\leq&\ka\varepsilon\|\partial^{\alpha+e_i}\widetilde{u}_i\|^{2}+C_\ka\varepsilon
		\sum^{3}_{j=1}\int_{\mathbb{R}^{3}}|\int_{\mathbb{R}^{3}}\partial^{\alpha-e_i}
		\Big\{\partial_t[\partial_{x_j}(\frac{1}{\rho})R\theta B_{ij}(\frac{v-u}{\sqrt{R\theta}})\frac{1}{M}]P_{1}(\sqrt{\mu}g_2)\Big\} dv|^{2} dx,
	\end{align*}
	which, using \eqref{background}, \eqref{apriori}, Lemma \ref{leL-1}, \eqref{patN-2} and \eqref{patN-1}, can be further bounded by
	$$(\ka+C_\ka \varepsilon)\mathcal{D}_{N,k}(t)+C_\ka\sqrt{\mathcal{E}_{N,k}(t)}\mathcal{D}_{N,k}(t)+C_\ka\eta \varepsilon^r.$$
	The other terms can be controlled in the same way. Hence, we arrive at
	\begin{align}\label{I42}
		I_{42}\leq&\sum^{3}_{j=1}\frac{d}{dt}\int_{\mathbb{R}^{3}}\int_{\mathbb{R}^{3}}\partial^{\alpha}
		\Big\{[\partial_{x_j}(\frac{1}{\rho})R\theta B_{ij}(\frac{v-u}{\sqrt{R\theta}})\frac{\varepsilon}{M}]P_{1}(\sqrt{\mu}g_2)\Big\}\partial^{\alpha}\widetilde{u}_i dvdx
		\nonumber\\
		&+(\ka+C_\ka \varepsilon)\mathcal{D}_{N,k}(t)+C_\ka\sqrt{\mathcal{E}_{N,k}(t)}\mathcal{D}_{N,k}(t)+C_\ka\eta \varepsilon^r.
	\end{align}
	The term $I_{43}$ only contains derivatives in space variable. Thus, using very similar arguments to how we estimate $I_{41}$ and $I_{42}$, it holds that
	\begin{align}\label{I43}
		I_{43}\leq (\ka+C_\ka \varepsilon)\mathcal{D}_{N,k}(t)+C_\ka\sqrt{\mathcal{E}_{N,k}(t)}\mathcal{D}_{N,k}(t)+C_\ka\eta \varepsilon^r.
	\end{align}
	We recall
	\begin{align*}
		I_{44}
		=&-\sum^{3}_{j=1}\int_{\mathbb{R}^{3}}\Big\{\partial^{\alpha}[\partial_{x_j}(\frac{1}{\rho})\int_{\mathbb{R}^{3}}R\theta B_{ij}(\frac{v-u}{\sqrt{R\theta}})\frac{(I-P_{0})(\sqrt{\mu}\CL_Bg_1)}{M}\, dv]\,\partial^{\alpha}\widetilde{u}_i\Big\} dx.
	\end{align*}
	We only calculate for the term involving the identity operator $I$ on the right hand side above as an example and the other term containing $P_0$ can be handled in the same way. Noticing $|\pa^{\bar\al}(\widetilde{\rho},\widetilde{u},\widetilde{\theta})|\leq C$ for $|\bar\al|\leq N-2$ by \eqref{apriori}, together with \eqref{boundbur} and \eqref{defLB}, we have
	\begin{align}\label{macrog1}
		&\int_{\mathbb{R}^{3}}\Big\{\partial^{\alpha}[\partial_{x_j}(\frac{1}{\rho})\int_{\mathbb{R}^{3}}R\theta B_{ij}(\frac{v-u}{\sqrt{R\theta}})\frac{\sqrt{\mu}\CL_Bg_1}{M}\, dv]\,\partial^{\alpha}\widetilde{u}_i\Big\} dx\notag\\
		\leq&C\sum_{\al_1+\al_2+\al_3= \al}\int_{\mathbb{R}^{3}}\big|[\partial^{\alpha_1}\partial_{x_j}(\frac{1}{\rho})\int_{\mathbb{R}^{3}}\partial^{\alpha_2}(R\theta B_{ij}(\frac{v-u}{\sqrt{R\theta}}))\frac{\sqrt{\mu}\CL_B\partial^{\alpha_3}g_1}{M}\, dv]\,\partial^{\alpha}\widetilde{u}_i\big|dx\notag\\
		\leq&C\sum_{\al_1+\al_2+\al_3= \al}\int_{\mathbb{R}^{3}}\big|\partial^{\alpha_1}\partial_{x_j}(\frac{1}{\rho})\partial^{\alpha_2}(u,\theta)\partial^{\alpha}\widetilde{u}_i\big(\int_{\mathbb{R}^{3}}|\CL_B\partial^{\alpha_3}g_1|^2 dv\big)^\frac{1}{2}\big|dx\notag\\
		\leq&C\sum_{\al_1+\al_2+\al_3= \al}\int_{\mathbb{R}^{3}}\big|\partial^{\alpha_1}\partial_{x_j}(\frac{1}{\rho})\partial^{\alpha_2}(u,\theta)\sqrt\varepsilon\partial^{\alpha}\widetilde{u}_i\big(\frac{1}{\varepsilon}\int_{\mathbb{R}^{3}}|\partial^{\alpha_3}g_1|^2 dv\big)^\frac{1}{2}\big|dx.
	\end{align}
	We only need to consider the endpoint case $\al_1=\al$ where $\partial^{\alpha_1}\partial_{x_j}$ can possibly reach the highest order derivative, then the above integral is bounded by
	$$
	C\eta\varepsilon^r+C\sqrt{\mathcal{E}_{N,k}(t)}\mathcal{D}_{N,k}(t)+C\big(\frac{1}{\varepsilon}\int_{\mathbb{R}^{3}}\|g_1\|_{L^\infty_x}^2 dv\big)^\frac{1}{2}\|\sqrt\varepsilon\partial^{\alpha}\partial_{x_j}\widetilde\rho\|\|\partial^{\alpha}\widetilde{u}_i\|,
	$$
	which can be further controlled by
	$C\eta\varepsilon^r+C\sqrt{\mathcal{E}_{N,k}(t)}\mathcal{D}_{N,k}(t)$.
	Similarly, all the terms in the rest cases enjoy the same upper bound. Then we arrive at
	\begin{align}\label{I44}
		I_{44}\leq C\sqrt{\mathcal{E}_{N,k}(t)}\mathcal{D}_{N,k}(t)+C\eta \varepsilon^r.
	\end{align}
	Rewrite
	\begin{align*}
		I_{45}&=-\sum^{3}_{j=1}\int_{\mathbb{R}^{3}}\Big\{\partial^{\alpha}[\partial_{x_j}(\frac{1}{\rho})\int_{\mathbb{R}^{3}}R\theta B_{ij}(\frac{v-u}{\sqrt{R\theta}})\frac{\sqrt{\mu}}{M}\Gamma(\frac{\overline{G}+\sqrt{\mu}g_2}{\sqrt{\mu}},\frac{\overline{G}+\sqrt{\mu}g_2}{\sqrt{\mu}})\, dv]\,\partial^{\alpha}\widetilde{u}_i\Big\} dx.
	\end{align*}
	Then using \eqref{background}, \eqref{Trilinear}, \eqref{boundbarG0}, \eqref{boundbarG}, \eqref{apriori} and similar arguments as in \eqref{macrog1}, it holds that
	\begin{align}\label{I45}
		I_{45}\leq C\sqrt{\mathcal{E}_{N,k}(t)}\mathcal{D}_{N,k}(t)+C\eta \varepsilon^r.
	\end{align}
	Collecting \eqref{I41}, \eqref{I42}, \eqref{I43}, \eqref{I44} and \eqref{I45} gives
	\begin{align}\label{I4}
		I_4&\leq \sum^{3}_{j=1}\frac{d}{dt}\int_{\mathbb{R}^{3}}\int_{\mathbb{R}^{3}}\partial^{\alpha}
		\Big\{[\partial_{x_j}(\frac{1}{\rho})R\theta B_{ij}(\frac{v-u}{\sqrt{R\theta}})\frac{\varepsilon}{M}]P_{1}(\sqrt{\mu}g_2)\Big\}\partial^{\alpha}\widetilde{u}_i dvdx
		\nonumber\\
		&+(\ka+C_\ka \varepsilon)\mathcal{D}_{N,k}(t)+C_\ka\sqrt{\mathcal{E}_{N,k}(t)}\mathcal{D}_{N,k}(t)+C_\ka\eta \varepsilon^r.
	\end{align}
	If we resolve $I_5$ as $I_4$ in \eqref{reI4}, we see their structures are very similar and the only difference is that the $\pa_{x_j}$ is moved from $\frac{1}{\rho}$ to $\widetilde{u}$. Hence, we just need to perform the similar calculations as in \eqref{I41} to \eqref{I45}, which gives
	\begin{align}\label{I5}
		I_5\leq& -\sum^{3}_{j=1}\frac{d}{dt}\int_{\mathbb{R}^{3}}\int_{\mathbb{R}^{3}}
		\Big\{\partial^{\alpha}\partial_{x_j}[\frac{1}{\rho}R\theta B_{ij}(\frac{v-u}{\sqrt{R\theta}})\frac{\varepsilon P_{1}(\sqrt{\mu}g_2)}{M}]
		\partial^{\alpha}\widetilde{u}_i\Big\} dvdx
		\nonumber\\
		&+(\ka+C_\ka \varepsilon)\mathcal{D}_{N,k}(t)+C_\ka\sqrt{\mathcal{E}_{N,k}(t)}\mathcal{D}_{N,k}(t)+C_\ka\eta \varepsilon^r.
	\end{align}
	It follows from \eqref{reTheta}, \eqref{I4} and \eqref{I5} that
	\begin{align}\label{uTheta}
		&-(\partial^{\alpha}[\frac{1}{\rho}\int_{\mathbb{R}^{3}}v_{i}v\cdot\nabla_{x}L^{-1}_{M}\Theta dv],\partial^{\alpha}\widetilde{u}_i)
		\nonumber\\
		\leq&\sum^{3}_{j=1}\frac{d}{dt}\int_{\mathbb{R}^{3}}\int_{\mathbb{R}^{3}}\partial^{\alpha}
		\Big\{[\partial_{x_j}(\frac{1}{\rho})R\theta B_{ij}(\frac{v-u}{\sqrt{R\theta}})\frac{\varepsilon}{M}]P_{1}(\sqrt{\mu}g_2)\Big\}\partial^{\alpha}\widetilde{u}_i dvdx
		\nonumber\\
		&-\sum^{3}_{j=1}\frac{d}{dt}\int_{\mathbb{R}^{3}}\int_{\mathbb{R}^{3}}
		\Big\{\partial^{\alpha}\partial_{x_j}[\frac{1}{\rho}R\theta B_{ij}(\frac{v-u}{\sqrt{R\theta}})\frac{\varepsilon P_{1}(\sqrt{\mu}g_2)}{M}]
		\partial^{\alpha}\widetilde{u}_i\Big\} dvdx
		\nonumber\\
		&+(\ka+C_\ka \varepsilon)\mathcal{D}_{N,k}(t)+C_\ka\sqrt{\mathcal{E}_{N,k}(t)}\mathcal{D}_{N,k}(t)+C_\ka\eta \varepsilon^r.
	\end{align}
	We combine \eqref{u1}, \eqref{ufluid}, \eqref{tensor}, \eqref{ug1} and \eqref{uTheta} to obtain
	\begin{align}
		&\frac{1}{2}\frac{d}{dt}\|\partial^{\alpha}\widetilde{u}_i\|^{2}+
		(\frac{2\bar{\theta}}{3\bar{\rho}}\partial^{\alpha}\widetilde{\rho}_{x_{i}},\partial^{\alpha}\widetilde{u}_i)+
		(\frac{2}{3}\partial^{\alpha}\widetilde{\theta}_{x_{i}},\partial^{\alpha}\widetilde{u}_i)
		\nonumber\\
		\leq&-\varepsilon\sum^{3}_{j=1}(\frac{1}{\rho}\mu(\theta)(\partial^{\alpha}\partial_{x_{j}}\widetilde{u}_{i}
		+\partial^{\alpha}\partial_{x_{i}}\widetilde{u}_{j}
		-\frac{2}{3}\delta_{ij}\nabla_{x}\cdot \partial^{\alpha}\widetilde{u}),\partial^{\alpha}\partial_{x_{j}}\widetilde{u}_{i})\notag\\
		&+\sum^{3}_{j=1}\frac{d}{dt}\int_{\mathbb{R}^{3}}\int_{\mathbb{R}^{3}}\partial^{\alpha}
		\Big\{[\partial_{x_j}(\frac{1}{\rho})R\theta B_{ij}(\frac{v-u}{\sqrt{R\theta}})\frac{\varepsilon}{M}]P_{1}(\sqrt{\mu}g_2)\Big\}\partial^{\alpha}\widetilde{u}_i dvdx
		\nonumber\\
		&-\sum^{3}_{j=1}\frac{d}{dt}\int_{\mathbb{R}^{3}}\int_{\mathbb{R}^{3}}
		\Big\{\partial^{\alpha}\partial_{x_j}[\frac{1}{\rho}R\theta B_{ij}(\frac{v-u}{\sqrt{R\theta}})\frac{\varepsilon P_{1}(\sqrt{\mu}g_2)}{M}]
		\partial^{\alpha}\widetilde{u}_i\Big\} dvdx
		\nonumber\\
		&+(\ka+C_\ka\varepsilon)\mathcal{D}_{N,k}(t)+C_\ka\sqrt{\mathcal{E}_{N,k}(t)}\mathcal{D}_{N,k}(t)+C_\ka \frac{1}{\varepsilon}\|w_2g_1\|_{H^{N-1}_xL^2_v}+C_\ka\eta \varepsilon^r\notag\\
		&+C_\ka\min\{\varepsilon^{-1}\sqrt{\mathcal{E}_{N,k}(t)}\mathcal{D}_{N,k}(t),\eta^{3/2}_0\varepsilon^r\},\notag
	\end{align}
	which, by taking summation and using our a prior assumption $\sqrt{\mathcal{E}_{N,k}(t)}\leq \varepsilon^{r/2}$, gives
	\begin{align}\label{sumu}
		&\frac{1}{2}\frac{d}{dt}\|\partial^{\alpha}\widetilde{u}\|^{2}+
		(\frac{2\bar{\theta}}{3\bar{\rho}}\partial^{\alpha}\nabla_{x}\widetilde{\rho},\partial^{\alpha}\widetilde{u})+(\frac{2}{3}\partial^{\alpha}\nabla_{x}\widetilde{\theta},\partial^{\alpha}\widetilde{u})
		\nonumber\\
		&+\sum^{3}_{i,j=1}\frac{d}{dt}\int_{\mathbb{R}^{3}}\int_{\mathbb{R}^{3}}
		\Big\{\partial^{\alpha}[\frac{1}{\rho}\partial_{x_j}(R\theta B_{ij}(\frac{v-u}{\sqrt{R\theta}})\frac{\varepsilon P_{1}(\sqrt{\mu}g_2)}{M})]
		\partial^{\alpha}\widetilde{u}_i\Big\} dvdx
		\nonumber\\
		\leq&-\varepsilon\sum^{3}_{i,j=1}(\frac{1}{\rho}\mu(\theta)(\partial^{\alpha}\partial_{x_{j}}\widetilde{u}_{i}+\partial^{\alpha}\partial_{x_{i}}\widetilde{u}_{j}
		-\frac{2}{3}\delta_{ij}\nabla_{x}\cdot \partial^{\alpha}\widetilde{u}),\partial^{\alpha}\partial_{x_{j}}\widetilde{u}_{i})+(\ka+C_\ka\varepsilon^{r/2})\mathcal{D}_{N,k}(t)\nonumber\\
		&+C_\ka \frac{1}{\varepsilon}\|w_2g_1\|_{H^{N-1}_xL^2_v}+C_\ka\eta \varepsilon^r+C_\ka\min\{\varepsilon^{-1}\sqrt{\mathcal{E}_{N,k}(t)}\mathcal{D}_{N,k}(t),\eta^{3/2}_0\varepsilon^r\}.
	\end{align}
	We have to take one more step to obtain the dissipation of $\widetilde{u}$ from the first term on the right hand side above by
	\begin{align*}
		&\varepsilon\sum^{3}_{i,j=1}(\frac{1}{\rho}\mu(\theta)(\partial^{\alpha}\partial_{x_{j}}\widetilde{u}_{i}+\partial^{\alpha}\partial_{x_{i}}\widetilde{u}_{j}
		-\frac{2}{3}\delta_{ij}\nabla_{x}\cdot \partial^{\alpha}\widetilde{u}),\partial^{\alpha}\partial_{x_{j}}\widetilde{u}_{i})
		\\
		=&\varepsilon\sum^{3}_{i,j=1}(\frac{\mu(\theta)}{\rho}\partial^{\alpha}\partial_{x_{j}}\widetilde{u}_{i},\partial^{\alpha}\partial_{x_{j}}\widetilde{u}_{i})
		+\frac{1}{3}\varepsilon\sum^{3}_{i,j=1}(\frac{\mu(\theta)}{\rho}\partial^{\alpha}\partial_{x_{j}}\widetilde{u}_{j},\partial^{\alpha}\partial_{x_{i}}\widetilde{u}_{i})
		\\
		&-\varepsilon\sum^{3}_{i,j=1}(\partial_{x_{j}}[\frac{\mu(\theta)}{\rho}]\partial^{\alpha}\partial_{x_{i}}\widetilde{u}_{j},\partial^{\alpha}\widetilde{u}_{i})
		+\varepsilon\sum^{3}_{i,j=1}(\partial_{x_{i}}[\frac{\mu(\theta)}{\rho}]\partial^{\alpha}\partial_{x_{j}}\widetilde{u}_{j},\partial^{\alpha}\widetilde{u}_{i})
		\\
		\geq& c\varepsilon\|\partial^{\alpha}\nabla_{x}\widetilde{u}\|^{2}-C\sqrt{\mathcal{E}_{N,k}(t)}\mathcal{D}_{N,k}(t)-C\eta\varepsilon^{r},
	\end{align*}
	which, combining with \eqref{sumu} and \eqref{apriori}, yields \eqref{energyu}.
\end{proof}

Next, for $\widetilde{\theta}$, we have the following lemma. 
\begin{lemma}\label{lem.theta}
	For $1\leq|\alpha|\leq N-1$ and $0<\ka<1$, it holds that
	\begin{align}\label{energytheta}
		&\frac{1}{2}\frac{d}{dt}(\partial^{\alpha}\widetilde{\theta},\frac{1}{\bar{\theta}}\partial^{\alpha}\widetilde{\theta})+(\frac{2}{3}\nabla_{x}\cdot \partial^{\alpha}\widetilde{u},\partial^{\alpha}\widetilde{\theta})
		+c\varepsilon\|\partial^{\alpha}\nabla_{x}\widetilde{\theta}\|^{2}
		\nonumber\\
		&+\sum^{3}_{i=1}\frac{d}{dt}\int_{\mathbb{R}^{3}}\int_{\mathbb{R}^{3}}
		\Big\{\partial^{\alpha}[\frac{1}{\rho}\partial_{x_i}
		((R\theta)^{\frac{3}{2}}A_{i}(\frac{v-u}{\sqrt{R\theta}})\frac{\varepsilon P_{1}(\sqrt{\mu}g_2)}{M})]\frac{1}{\bar{\theta}}\partial^{\alpha}\widetilde{\theta}\Big\} dvdx
		\nonumber\\	
		&+\sum^{3}_{i,j=1}\frac{d}{dt}\int_{\mathbb{R}^{3}}\int_{\mathbb{R}^{3}}
		\Big\{\partial^{\alpha}[\frac{1}{\rho}
		\partial_{x_i}u_j R\theta B_{ij}(\frac{v-u}{\sqrt{R\theta}})\frac{\varepsilon P_{1}(\sqrt{\mu}g_2)}{M}]\frac{1}{\bar{\theta}}\partial^{\alpha}\widetilde{\theta}\Big\} dvdx
		\nonumber\\	
		\leq& (\ka+C_\ka\varepsilon^{r/2})\mathcal{D}_{N,k}(t)+C_\ka \frac{1}{\varepsilon}\|w_3g_1\|_{H^{N-1}_xL^2_v}+C_\ka\eta \varepsilon^r\notag\\
		&+C_\ka\min\{\varepsilon^{-1}\sqrt{\mathcal{E}_{N,k}(t)}\mathcal{D}_{N,k}(t),\eta^{3/2}_0\varepsilon^r\}.
	\end{align}
\end{lemma}
\begin{proof}
	From the third equation of \eqref{macro2} and similar steps to \eqref{rho1} and \eqref{u1}, it holds that
	\begin{align}\label{theta1}
		&(\partial_{t}\partial^{\alpha}\widetilde{\theta},\frac{1}{\bar{\theta}}\partial^{\alpha}\widetilde{\theta})+(\frac{2}{3}\nabla_{x}\cdot \partial^{\alpha}\widetilde{u},\partial^{\alpha}\widetilde{\theta})+\sum_{\alpha_{1}\leq \alpha, |\alpha_{1}|\geq1}C^{\alpha_{1}}_\alpha(\frac{2}{3}\partial^{\alpha_{1}}\bar{\theta}\nabla_{x}\cdot \partial^{\alpha-\alpha_{1}}\widetilde{u},\frac{1}{\bar{\theta}}\partial^{\alpha}\widetilde{\theta})
		\nonumber\\
		=&-(\partial^{\alpha}(u\cdot\nabla_{x}\widetilde{\theta}),\frac{1}{\bar{\theta}}\partial^{\alpha}\widetilde{\theta})-(\partial^{\alpha}(\widetilde{u}\cdot\nabla_{x}\bar{\theta}),\frac{1}{\bar{\theta}}\partial^{\alpha}\widetilde{\theta})
		-(\frac{2}{3}\partial^{\alpha}(\widetilde{\theta}\nabla_{x}\cdot u)
		,\frac{1}{\bar{\theta}}\partial^{\alpha}\widetilde{\theta})\nonumber\\
		&
		+\varepsilon\sum^{3}_{j=1}(\partial^{\alpha}[\frac{1}{\rho}\partial_{x_{j}}(\kappa(\theta)\partial_{x_{j}}\theta)],\frac{1}{\bar{\theta}}\partial^{\alpha}\widetilde{\theta})
		+\varepsilon\sum^{3}_{i,j=1}(\partial^{\alpha}[\frac{1}{\rho}\mu(\theta)\partial_{x_{j}}u_{i}D_{ij}],\frac{1}{\bar{\theta}}\partial^{\alpha}\widetilde{\theta})
		\notag\\
		&-(\partial^{\alpha}[\frac{1}{\rho}\int_{\mathbb{R}^{3}} \frac{1}{2}|v|^{2} v\cdot\nabla_{x}g_1 dv],\frac{1}{\bar{\theta}}\partial^{\alpha}\widetilde{\theta})
		+(\partial^{\alpha}[\frac{1}{\rho}u\cdot\int_{\mathbb{R}^{3}} v\otimes v\cdot\nabla_{x}g_1 dv]
		,\frac{1}{\bar{\theta}}\partial^{\alpha}\widetilde{\theta})\notag\\
		&-(\partial^{\alpha}[\frac{1}{\rho}\int_{\mathbb{R}^{3}} \frac{1}{2}|v|^{2} v\cdot\nabla_{x}L^{-1}_{M}\Theta dv],\frac{1}{\bar{\theta}}\partial^{\alpha}\widetilde{\theta})
		+(\partial^{\alpha}[\frac{1}{\rho}u\cdot\int_{\mathbb{R}^{3}} v\otimes v\cdot\nabla_{x}L^{-1}_{M}\Theta dv]
		,\frac{1}{\bar{\theta}}\partial^{\alpha}\widetilde{\theta}).
	\end{align}
	Fortunately, almost all the terms above can be estimated as in Lemma \ref{lem.rho} and Lemma \ref{lem.u}. We first use integration by parts, the Sobolev inequality \eqref{euler} and  \eqref{background} to get
	\begin{align*}
		(\partial_{t}\partial^{\alpha}\widetilde{\theta},\frac{1}{\bar{\theta}}\partial^{\alpha}\widetilde{\theta})&=\frac{1}{2}\frac{d}{dt}(\partial^{\alpha}\widetilde{\theta},\frac{1}{\bar{\theta}}\partial^{\alpha}\widetilde{\theta})-\frac{1}{2}(\partial^{\alpha}\widetilde{\theta},\partial_{t}(\frac{1}{\bar{\theta}})\partial^{\alpha}\widetilde{\theta})
		\\
		&\geq\frac{1}{2}\frac{d}{dt}(\partial^{\alpha}\widetilde{\theta},\frac{1}{\bar{\theta}}\partial^{\alpha}\widetilde{\theta})
		-C\|\partial_t\bar{\theta}\|_{\infty}\|\partial^{\alpha}\widetilde{\theta}\|^2\notag\\
		&\geq \frac{1}{2}\frac{d}{dt}(\partial^{\alpha}\widetilde{\theta},\frac{1}{\bar{\theta}}\partial^{\alpha}\widetilde{\theta})-C\eta\varepsilon^r.
	\end{align*}
	Then similar arguments as in \eqref{brtutr}, \eqref{urr1}, \eqref{urr2} and \eqref{trutr2} yield
	\begin{align*}
		&\sum_{\alpha_{1}\leq \alpha, |\alpha_{1}|\geq1}C^{\alpha_{1}}_\alpha|(\frac{2}{3}\partial^{\alpha_{1}}\bar{\theta}\nabla_{x}\cdot \partial^{\alpha-\alpha_{1}}\widetilde{u},\frac{1}{\bar{\theta}}\partial^{\alpha}\widetilde{\theta})|
		+|(\partial^{\alpha}(u\cdot\nabla_{x}\widetilde{\theta}),\frac{1}{\bar{\theta}}\partial^{\alpha}\widetilde{\theta})|
		\\
		&\qquad+|(\partial^{\alpha}(\widetilde{u}\cdot\nabla_{x}\bar{\theta}),\frac{1}{\bar{\theta}}\partial^{\alpha}\widetilde{\theta})|
		+|(\frac{2}{3}\partial^{\alpha}(\widetilde{\theta}\nabla_{x}\cdot u)
		,\frac{1}{\bar{\theta}}\partial^{\alpha}\widetilde{\theta})|
		\notag\\
		\leq&C\eta\varepsilon^r+\ka\varepsilon\|\nabla_x\cdot\partial^{\alpha}\widetilde{u}\|^2+C_\ka\min\{\varepsilon^{-1}\sqrt{\mathcal{E}_{N,k}(t)}\mathcal{D}_{N,k}(t),\eta^{3/2}_0\varepsilon^r\}.
	\end{align*}
	It remains now to estimate the last six terms on the right hand side of \eqref{theta1}. As in \eqref{tensor}, we have
	\begin{align*}
		&\varepsilon\sum^{3}_{j=1}(\partial^{\alpha}[\frac{1}{\rho}\partial_{x_{j}}(\kappa(\theta)\partial_{x_{j}}\theta)],\frac{1}{\bar{\theta}}\partial^{\alpha}\widetilde{\theta})\notag\\
		\leq&-\varepsilon\sum^{3}_{j=1}(\frac{1}{\rho}\kappa(\theta)\partial^{\alpha}\partial_{x_{j}}\widetilde{\theta},
		\frac{1}{\bar{\theta}}\partial^{\alpha}\partial_{x_{j}}\widetilde{\theta})
		+C\eta\varepsilon^r+C\sqrt{\mathcal{E}_{N,k}(t)}\mathcal{D}_{N,k}(t).
	\end{align*}
	In the same way,
	\begin{align*}
		\varepsilon\sum^{3}_{i,j=1}(\partial^{\alpha}[\frac{1}{\rho}\mu(\theta)\partial_{x_{j}}u_{i}D_{ij}],\frac{1}{\bar{\theta}}\partial^{\alpha}\widetilde{\theta})\leq C\eta\varepsilon^r+C\sqrt{\mathcal{E}_{N,k}(t)}\mathcal{D}_{N,k}(t).
	\end{align*}
	For sixth and seventh term on the right hand side of \eqref{theta1} containing $g_1$, we have for some $|e_i|=1$ that
	\begin{align*}
		&-(\partial^{\alpha}[\frac{1}{\rho}\int_{\mathbb{R}^{3}} \frac{1}{2}|v|^{2} v\cdot\nabla_{x}g_1 dv],\frac{1}{\bar{\theta}}\partial^{\alpha}\widetilde{\theta})
		+(\partial^{\alpha}[\frac{1}{\rho}u\cdot\int_{\mathbb{R}^{3}} v\otimes v\cdot\nabla_{x}g_1 dv]
		,\frac{1}{\bar{\theta}}\partial^{\alpha}\widetilde{\theta})
		\notag\\
		=&(\partial^{\alpha-e_i}[\frac{1}{\rho}\int_{\mathbb{R}^{3}} \frac{1}{2}|v|^{2} v\cdot\nabla_{x}g_1 dv-\frac{1}{\rho}u\cdot\int_{\mathbb{R}^{3}} v\otimes v\cdot\nabla_{x}g_1 dv]
		,\frac{1}{\bar{\theta}}\partial^{\alpha+e_i}\widetilde{\theta})
		\notag\\
		\leq& C\ka\mathcal{D}_{N,k}(t)+C_\ka \frac{1}{\varepsilon}\|w_3g_1\|_{H^{N-1}_xL^2_v}.
	\end{align*}
	The last two terms can be calculated in the way of \eqref{reTheta}, \eqref{I4} and \eqref{I5}. We omit the detailed steps to get
	\begin{align*}
		&-(\partial^{\alpha}[\frac{1}{\rho}\int_{\mathbb{R}^{3}} \frac{1}{2}|v|^{2} v\cdot\nabla_{x}L^{-1}_{M}\Theta dv],\frac{1}{\bar{\theta}}\partial^{\alpha}\widetilde{\theta})
		+(\partial^{\alpha}[\frac{1}{\rho}u\cdot\int_{\mathbb{R}^{3}} v\otimes v\cdot\nabla_{x}L^{-1}_{M}\Theta dv]
		,\frac{1}{\bar{\theta}}\partial^{\alpha}\widetilde{\theta})\notag\\
		=&-\sum^{3}_{i=1}
		(\partial^{\alpha}[\frac{1}{\rho}\partial_{x_i}
		((R\theta)^{\frac{3}{2}}\int_{\mathbb{R}^{3}}A_{i}(\frac{v-u}{\sqrt{R\theta}})\frac{\Theta}{M} dv)],\frac{1}{\bar{\theta}}\partial^{\alpha}\widetilde{\theta})\\
		&\quad-\sum^{3}_{i,j=1}
		(\partial^{\alpha}[\frac{1}{\rho}
		\partial_{x_i}u_j R\theta\int_{\mathbb{R}^{3}}B_{ij}(\frac{v-u}{\sqrt{R\theta}})\frac{\Theta}{M} dv],\frac{1}{\bar{\theta}}\partial^{\alpha}\widetilde{\theta})\notag\\
		\leq&
		-\sum^{3}_{i=1}\frac{d}{dt}\int_{\mathbb{R}^{3}}\int_{\mathbb{R}^{3}}
		\Big\{\partial^{\alpha}[\frac{1}{\rho}\partial_{x_i}
		((R\theta)^{\frac{3}{2}}A_{i}(\frac{v-u}{\sqrt{R\theta}})\frac{\varepsilon P_{1}(\sqrt{\mu}g_2)}{M})]\frac{1}{\bar{\theta}}\partial^{\alpha}\widetilde{\theta}\Big\} dvdx
		\nonumber\\	
		&-\sum^{3}_{i,j=1}\frac{d}{dt}\int_{\mathbb{R}^{3}}\int_{\mathbb{R}^{3}}
		\Big\{\partial^{\alpha}[\frac{1}{\rho}
		\partial_{x_i}u_j R\theta B_{ij}(\frac{v-u}{\sqrt{R\theta}})\frac{\varepsilon P_{1}(\sqrt{\mu}g_2)}{M}]\frac{1}{\bar{\theta}}\partial^{\alpha}\widetilde{\theta}\Big\} dvdx
		\nonumber\\	
		&+(\ka+C_\ka\varepsilon)\mathcal{D}_{N,k}(t)+C_\ka\sqrt{\mathcal{E}_{N,k}(t)}\mathcal{D}_{N,k}(t)+C_\ka\eta \varepsilon^r.
	\end{align*}
This ends the proof of Lemma \ref{lem.theta}.
\end{proof}

We see from Lemma \ref{lem.rho}, Lemma \ref{lem.u} and Lemma \ref{lem.theta} that one more step is needed to obtain the dissipation of $\widetilde{\rho}$.
\begin{lemma}
	For any $1\leq|\alpha|\leq N-1$, it holds that
	\begin{align}\label{dissirho}	
		&\varepsilon\sum_{1\leq |\alpha|\leq N-1} \frac{d}{dt}(\partial^{\alpha}\widetilde{u},\nabla_x\partial^{\alpha}\widetilde{\rho})
		+c\varepsilon\sum_{2\leq |\alpha|\leq N}\|\partial^{\alpha}\widetilde{\rho}\|^2
		\nonumber\\
		\leq& C\varepsilon\sum_{2\leq |\alpha|\leq N}
		(\|\partial^{\alpha}\widetilde{u}\|^2+\|\partial^{\alpha}\widetilde{\theta}\|^2+\|w_2 \partial^{\alpha}g_1\|^2+\|\partial^{\alpha}\mathbf{P}_1g_2\|_{L^2_{v,D}}^2)\notag\\
		&
		+C\sqrt{\mathcal{E}_{N,k}(t)}\mathcal{D}_{N,k}(t)+C\eta \varepsilon^r.
	\end{align}
\end{lemma}
\begin{proof}
	Applying  $\partial^{\alpha}$ to the second equation of \eqref{macro1} with $1\leq|\alpha|\leq N-1$  
	and then taking the inner product of the resulting equation with  $\nabla_x\partial^{\alpha}\widetilde{\rho}$, one gets
	\begin{align*}	
		&(\partial^{\alpha}\partial_t\widetilde{u},\nabla_x\partial^{\alpha}\widetilde{\rho})
		+(\frac{2\bar{\theta}}{3\bar{\rho}}\nabla_x\partial^{\alpha}\widetilde{\rho},\nabla_x\partial^{\alpha}\widetilde{\rho})
		+\sum_{\alpha_{1}\leq \alpha, |\alpha_{1}|\geq1}C_{\alpha}^{\alpha_{1}}(\partial^{\alpha_1}(\frac{2\bar{\theta}}{3\bar{\rho}})\nabla_x\partial^{\alpha-\alpha_{1}}\widetilde{\rho},\nabla_x\partial^{\alpha}\widetilde{\rho})
		\nonumber\\
		=&
		-\big(\partial^{\alpha}(u\cdot\nabla_x\widetilde{u})+\frac{2}{3}\nabla_x\partial^{\alpha}\widetilde{\theta}+\partial^{\alpha}(\widetilde{u}\cdot\nabla_x\bar{u})+\partial^{\alpha}[
		\frac{2}{3}(\frac{\theta}{\rho}-\frac{\bar{\theta}}{\bar{\rho}})\nabla_x\rho],\nabla_x\partial^{\alpha}\widetilde{\rho}\big)
		\nonumber\\
		&-(\partial^{\alpha}(\frac{1}{\rho}\int_{\mathbb{R}^{3}} v\otimes v\cdot\nabla_x G dv),\nabla_x\partial^{\alpha}\widetilde{\rho}).
	\end{align*}
	Then \eqref{dissirho} follows from $G(t,x,v)=\overline{G}(t,x,v)+g_1(t,x,v)+\sqrt{\mu}g_2(t,x,v)$, \eqref{g2macro} and similar arguments in \eqref{brtutr}, \eqref{urr1}, \eqref{urr2} and \eqref{trutr2}.
\end{proof}
Noticing in \eqref{energyu} and \eqref{energytheta}, there are terms like $$\frac{d}{dt}\int_{\mathbb{R}^{3}}\int_{\mathbb{R}^{3}}
\Big\{\partial^{\alpha}[\frac{1}{\rho}\partial_{x_j}(R\theta B_{ij}(\frac{v-u}{\sqrt{R\theta}})\frac{\varepsilon P_{1}(\sqrt{\mu}g_2)}{M})]
\partial^{\alpha}\widetilde{u}_i\Big\} dvdx.$$ Therefore, before taking combination, we estimate such terms.
\begin{lemma}
	For any $1\leq|\alpha|\leq N-1$ and $0<\ka<1$, it holds that
	\begin{align}\label{cor}
		&\sum_{1\leq |\alpha|\leq N-1}\sum^{3}_{i,j=1}\int_{\mathbb{R}^{3}}\int_{\mathbb{R}^{3}}
		\Big|\partial^{\alpha}[\frac{1}{\rho}\partial_{x_j}(R\theta B_{ij}(\frac{v-u}{\sqrt{R\theta}})\frac{\varepsilon P_{1}(\sqrt{\mu}g_2)}{M})]
		\partial^{\alpha}\widetilde{u}_i\Big| dvdx
		\nonumber\\
		&+\sum_{1\leq |\alpha|\leq N-1}\sum^{3}_{i=1}\int_{\mathbb{R}^{3}}\int_{\mathbb{R}^{3}}
		\Big|\partial^{\alpha}[\frac{1}{\rho}\partial_{x_i}
		((R\theta)^{\frac{3}{2}}A_{i}(\frac{v-u}{\sqrt{R\theta}})\frac{\varepsilon P_{1}(\sqrt{\mu}g_2)}{M})]\frac{1}{\bar{\theta}}\partial^{\alpha}\widetilde{\theta}\Big| dvdx
		\nonumber\\	
		&+\sum_{1\leq |\alpha|\leq N-1}\sum^{3}_{i,j=1}\int_{\mathbb{R}^{3}}\int_{\mathbb{R}^{3}}
		\Big|\partial^{\alpha}[\frac{1}{\rho}
		\partial_{x_i}u_j R\theta B_{ij}(\frac{v-u}{\sqrt{R\theta}})\frac{\varepsilon P_{1}(\sqrt{\mu}g_2)}{M}]\frac{1}{\bar{\theta}}\partial^{\alpha}\widetilde{\theta}\Big| dvdx\notag\\
		\leq&C\ka 
		\sum_{1\leq |\alpha|\leq N-1}\|\partial^{\alpha}(\widetilde{\rho},\widetilde{u},\widetilde{\theta})\|^2+C_{\ka}\varepsilon^2\sum_{|\alpha|=N}\|\partial^{\alpha}g_2\|^2+
		C_{\ka}(\eta+\varepsilon^{r/2})\varepsilon^{r}.
	\end{align}
\end{lemma}
\begin{proof}
	We only prove it for the case
	$$
	\int_{\mathbb{R}^{3}}\int_{\mathbb{R}^{3}}
	\Big|\partial^{\alpha}[\frac{1}{\rho}\partial_{x_j}(R\theta B_{ij}(\frac{v-u}{\sqrt{R\theta}})\frac{\varepsilon \sqrt{\mu}g_2}{M})]
	\partial^{\alpha}\widetilde{u}_i\Big| dvdx,
	$$
since other cases can be calculated with the same approach. It is direct to see that
	\begin{align*}
		&\int_{\mathbb{R}^{3}}\int_{\mathbb{R}^{3}}
		\Big|\partial^{\alpha}[\frac{1}{\rho}\partial_{x_j}(R\theta B_{ij}(\frac{v-u}{\sqrt{R\theta}})\frac{\varepsilon\sqrt{\mu}g_2}{M})]
		\partial^{\alpha}\widetilde{u}_i\Big| dvdx\notag\\
		\leq &\int_{\mathbb{R}^{3}}\int_{\mathbb{R}^{3}}
		\Big|\frac{1}{\rho}R\theta B_{ij}(\frac{v-u}{\sqrt{R\theta}})\frac{\varepsilon \sqrt{\mu}\partial^{\alpha}\partial_{x_j}g_2}{M}
		\partial^{\alpha}\widetilde{u}_i\Big| dvdx\notag\\
		&
		+C\sum_{\alpha_{1}\leq \alpha, |\alpha_{1}|\geq1}\int_{\mathbb{R}^{3}}\int_{\mathbb{R}^{3}}
		\Big|\partial^{\al_1}[\frac{1}{\rho}\partial_{x_j}(R\theta B_{ij}(\frac{v-u}{\sqrt{R\theta}})\frac{1}{M})]\varepsilon\sqrt{\mu}\pa^{\al-\al_1}g_2
		\partial^{\alpha}\widetilde{u}_i\Big| dvdx.
	\end{align*}
	Using \eqref{apriori}, H\"older's inequality and the fact that for $|\al|\leq N-2$,
	$$
	|\pa^\al(\rho-1,u,\theta-\frac{3}{2})|+\|\pa^\al(g_1,g_2)\|_{L^\infty_xL^2_v}\leq C(\eta+\bar{\eta}+\eta_0)\leq C,
	$$
	we have
	$$
	\sum_{\alpha_{1}\leq \alpha, |\alpha_{1}|\geq1}\int_{\mathbb{R}^{3}}\int_{\mathbb{R}^{3}}
	\Big|\partial^{\al_1}[\frac{1}{\rho}\partial_{x_j}(R\theta B_{ij}(\frac{v-u}{\sqrt{R\theta}})\frac{1}{M})]\varepsilon\sqrt{\mu}\pa^{\al-\al_1}g_2
	\partial^{\alpha}\widetilde{u}_i\Big| dvdx\leq C(\eta+\varepsilon^{r/2})\varepsilon^{r},
	$$
	which yields
	\begin{align*}
		&\int_{\mathbb{R}^{3}}\int_{\mathbb{R}^{3}}
		\Big|\partial^{\alpha}[\frac{1}{\rho}\partial_{x_j}(R\theta B_{ij}(\frac{v-u}{\sqrt{R\theta}})\frac{\varepsilon \sqrt{\mu}g_2}{M})]
		\partial^{\alpha}\widetilde{u}_i\Big| dvdx\notag\\
		\leq& \int_{\mathbb{R}^{3}}\int_{\mathbb{R}^{3}}
		\Big|\frac{1}{\rho}R\theta B_{ij}(\frac{v-u}{\sqrt{R\theta}})\frac{\varepsilon \sqrt{\mu}\partial^{\alpha}\partial_{x_j}g_2}{M}
		\partial^{\alpha}\widetilde{u}_i\Big| dvdx+C(\eta+\varepsilon^{r/2})\varepsilon^{r}\notag\\
		\leq& C\int_{\mathbb{R}^{3}}\Big(\int_{\mathbb{R}^{3}}
		\Big| B_{ij}(\frac{v-u}{\sqrt{R\theta}})M^{-\frac{1}{2}}\Big|^2dv\Big)^\frac{1}{2}\varepsilon\| \partial^{\alpha}\partial_{x_j}g_2\|_{L^2_v}
		|\partial^{\alpha}\widetilde{u}_i| dx+C(\eta+\varepsilon^{r/2})\varepsilon^{r}\notag\\
		\leq& C\varepsilon\|\partial^{\alpha}\widetilde{u}\|\|\partial^{\alpha}\partial_{x_j}g_2\|+
		C_{\ka}(\eta+\varepsilon^{r/2})\varepsilon^{r}.
	\end{align*}
	If $|\al|\leq N-2$, then $\varepsilon\|\partial^{\alpha}\widetilde{u}\|\|\partial^{\alpha}\partial_{x_j}g_2\|\leq C\varepsilon^{1+r}$. If $|\al|=N-1$, then 
	$$
	\varepsilon\|\partial^{\alpha}\widetilde{u}\|\|\partial^{\alpha}\partial_{x_j}g_2\|\leq \ka \|\partial^{\alpha}\widetilde{u}\|^2+ C_\ka\varepsilon^2\|\partial^{\alpha}\nabla_{x}g_2\|^2.
	$$ 
	Thus, it holds that
	\begin{align*}
		&\int_{\mathbb{R}^{3}}\int_{\mathbb{R}^{3}}
		\Big|\partial^{\alpha}[\frac{1}{\rho}\partial_{x_j}(R\theta B_{ij}(\frac{v-u}{\sqrt{R\theta}})\frac{\varepsilon \sqrt{\mu}g_2}{M})]
		\partial^{\alpha}\widetilde{u}_i\Big| dvdx\notag\\
		\leq&  C\ka \|\partial^{\alpha}\widetilde{u}\|^2+ C_{\ka}\varepsilon^2\sum_{|\alpha|=N}\|\partial^{\alpha}g_2\|^2+
		C_{\ka}(\eta+\varepsilon^{r/2})\varepsilon^{r},
	\end{align*}
	which gives the desired estimate.
\end{proof}

Combining \eqref{energyrho}, \eqref{energyu}, \eqref{energytheta}, \eqref{dissirho},  \eqref{cor} and our a priori assumption \eqref{apriori}, then choosing $\ka$ to be small, we directly get the following lemma.
\begin{lemma}\label{lem.macroN-1}
	It holds that
	\begin{align}\label{macroN-1}
		&\sum_{1\leq |\alpha|\leq N-1}\|\partial^{\alpha}(\widetilde{\rho},\widetilde{u},\widetilde{\theta})(t)\|^2+c\varepsilon\sum_{2\leq |\alpha|\leq N}\int^t_0\|\partial^{\alpha}(\widetilde{\rho},\widetilde{u},\widetilde{\theta})(s)\|^2ds
		\nonumber\\
		\leq& C\sum_{1\leq |\alpha|\leq N-1}\|\partial^{\alpha}(\widetilde{\rho},\widetilde{u},\widetilde{\theta})(0)\|^2+C\varepsilon^2\sum_{|\alpha|=N}
		\|\partial^{\alpha}\widetilde{\rho}(0)\|^2+C\varepsilon^2\sum_{|\alpha|=N}
		(\|\partial^{\alpha}\widetilde{\rho}(t)\|^2+
		\|\partial^{\alpha}g_2(t)\|^2)
		\nonumber\\
		&+ C\varepsilon\sum_{1\leq |\alpha|\leq N}\int^t_0(\|w_2 \partial^{\alpha}g_1(s)\|^2+\| \partial^{\alpha}\mathbf{P}_1g_2(s)\|_{L^2_{v,D}}^2)ds+(C\ka+C_\ka\varepsilon^{r/2})\int^t_0\mathcal{D}_N(s) ds\notag\\
		&+C_\ka \frac{1}{\varepsilon}\int_0^t\|w_3g_1(s)\|_{H^{N-1}_xL^2_v}ds+C_\ka(\eta+\varepsilon^{r/2})\varepsilon^{r}+C_\ka t\eta\varepsilon^{r}+C_\ka\int^t_0\min\{\eta_0^{1/2}\varepsilon^{r/2-1}\mathcal{D}_{N,k}(s),\eta^{3/2}_0\varepsilon^r\}ds,
	\end{align}
	where $C_\ka>1$ depends only on $\ka$.
\end{lemma}

\section{Estimates on lower order microscopic quantities}
In this section, our main task is to establish Lemma \ref{leloworderg1} and Lemma \ref{leloworderg2}, which give the bounds for microscopic quantities $g_1$ and $g_2$ from \eqref{g1} and \eqref{g2}.

Applying $\partial^{\alpha}$ to \eqref{g1} with $1\leq|\alpha|\leq N-1$ and taking the inner product of the resulting equation with
$w_{2k-8|\al|}\partial^{\alpha}g_1$ over $\mathbb{R}^{3}\times\mathbb{R}^{3}$, one has
\begin{align}\label{ipg1}
	\frac{1}{2}\frac{d}{dt}\|w_{k-4|\al|}\partial^{\alpha}g_1\|^{2}\leq& \frac{1}{\varepsilon}(\CL_D \partial^{\alpha}g_1,w_{2k-8|\al|}\partial^{\alpha}g_1)+\frac{1}{\varepsilon}(\partial^{\alpha}Q(g_1,g_1),w_{2k-8|\al|}\partial^{\alpha}g_1)\notag\\
	&+\frac{1}{\varepsilon}\big((\partial^{\alpha}Q(\sqrt{\mu}g_2,g_1),w_{2k-8|\al|}\partial^{\alpha}g_1)
	+(\partial^{\alpha}Q(g_1,\sqrt{\mu}g_2),w_{2k-8|\al|}\partial^{\alpha}g_1)\big)
	\notag\\
	&+\frac{1}{\varepsilon}\big((\partial^{\alpha}Q(M-\mu,g_1),w_{2k-8|\al|}\partial^{\alpha}g_1)
	+(\partial^{\alpha}Q(g_1,M-\mu),w_{2k-8|\al|}\partial^{\alpha}g_1)\big)
	\notag\\
	&+\frac{1}{\varepsilon}\big((\partial^{\alpha}Q(\overline{G},g_1),w_{2k-8|\al|}\partial^{\alpha}g_1)
	+(\partial^{\alpha}Q(g_1,\overline{G}),w_{2k-8|\al|}\partial^{\alpha}g_1)\big).
\end{align}
In the above inequality, it will be easier to estimate three terms which contain $Q(g_1,\sqrt{\mu}g_2)$, $Q(g_1,M-\mu)$ and $Q(g_1,\overline{G})$ since $\sqrt{\mu}g_2$, $M-\mu$ and $\overline{G}$ all produce exponential decay in $v$ that can absorb any polynomial increase. For the rest terms, we will bound them for two cases, one is that all the derivatives acting on $g_1$ like $Q(\cdot,\pa^\al g_1)$, the other will be in the form of $Q(\cdot,\pa^{\al_1} g_1)$ with $\al_1<\al$. We focus on the terms including $Q(\cdot,\pa^\al g_1)$ first.

\begin{lemma}\label{leestLD}
	There exists a constant $c$ such that for $1\leq|\alpha|\leq N-1$ and $k\geq 25+4|\al|$, it holds that
	\begin{align}\label{estLD}
		&\frac{1}{\varepsilon}(L_D  \partial^\alpha g_1,  w_{2k-8|\alpha|}\partial^\alpha g_1 )   +   \frac{1}{\varepsilon}(Q(g_1+\sqrt\mu g_2+(M-\mu)+\overline{G},  \partial^\alpha g_1),  w_{2k-8|\alpha|}\partial^\alpha g_1 ) 
		\notag\\
		\le&   - c \frac{1}{\varepsilon}\Vert w_{k-4|\alpha |}\partial^\alpha g_1 \Vert_{H^s_{\ga/2}}^2
		+C_k(\eta+\bar\eta)\mathcal{D}_{N,k}(t)+C_k\frac{1}{\sqrt\varepsilon}\sqrt{\mathcal{E}_{N,k}(t)}
		\mathcal{D}_{N,k}(t),
	\end{align}
	where $C_k$ depends only on $k$.
\end{lemma}
\begin{proof}
	We first consider the case $|\al|>1$. Substituting $F=\mu+g_1+\sqrt\mu g_2+(M-\mu)+\overline{G}$ into \eqref{leld}, one gets
	\begin{align}\label{LD2order}
		&\frac{1}{\varepsilon}(L_D  \partial^\alpha g_1,  w_{2k-8|\alpha|}\partial^\alpha g_1 )   +   \frac{1}{\varepsilon}(Q(g_1+\sqrt\mu g_2+(M-\mu)+\overline{G},  \partial^\alpha g_1),  w_{2k-8|\alpha|}\partial^\alpha g_1 )  
		\notag\\
		\le & - \de\frac{1}{\varepsilon} \Vert \partial^\alpha g_1 \Vert_{L^2_x H^s_{k+\gamma/2 -4|\al|}}^2
		+C_k\frac{1}{\varepsilon}\big(\Vert  \partial^\alpha g_1 \Vert_{L^2_x L^2_{14}} \Vert g_1 \Vert_{L^\infty_x   H^s_{ k+\gamma/2 -4|\al| }}\Vert  \partial^\alpha g_1 \Vert_{L^2_x H^s_{ k+\gamma/2 -4|\al|}}\notag\\
		& +\Vert g_1 \Vert_{L^\infty_x L^2_{14} } \Vert  \partial^\alpha g_1 \Vert_{L^2_xH^s_{ k+\gamma/2 -4|\al|}}^2
		+\Vert  \partial^\alpha g_1 \Vert_{L^2_x L^2_{14}} \Vert \sqrt{\mu }g_2 \Vert_{L^\infty_x   H^s_{ k+\gamma/2 -4|\al| }}\Vert  \partial^\alpha g_1 \Vert_{L^2_x H^s_{ k+\gamma/2 -4|\al|}}\notag\\
		& +\Vert \sqrt{\mu }g_2 \Vert_{L^\infty_x L^2_{14} } \Vert  \partial^\alpha g_1 \Vert_{L^2_xH^s_{ k+\gamma/2 -4|\al| }}^2+\Vert  \partial^\alpha g_1 \Vert_{L^2_x L^2_{14}} \Vert M-\mu \Vert_{L^\infty_x   H^s_{ k+\gamma/2 -4|\al| }}\Vert  \partial^\alpha g_1 \Vert_{L^2_x H^s_{ k+\gamma/2 -4|\al|}}\notag\\
		& +\Vert M-\mu \Vert_{L^\infty_x L^2_{14} } \Vert  \partial^\alpha g_1 \Vert_{L^2_xH^s_{ k+\gamma/2 -4|\al| }}^2+\Vert  \partial^\alpha g_1 \Vert_{L^2_x L^2_{14}} \Vert \overline{G} \Vert_{L^\infty_x   H^s_{ k+\gamma/2 -4|\al| }}\Vert  \partial^\alpha g_1 \Vert_{L^2_x H^s_{ k+\gamma/2 -4|\al|}}\notag\\
		& +\Vert \overline{G} \Vert_{L^\infty_x L^2_{14} } \Vert  \partial^\alpha g_1 \Vert_{L^2_xH^s_{ k+\gamma/2 -4|\al| }}^2\big).
	\end{align}
	It then follows by Sobolev embedding that
	\begin{align}\label{3g1}
		&\frac{1}{\varepsilon}\Vert  \partial^\alpha g_1 \Vert_{L^2_x L^2_{14}} \Vert g_1 \Vert_{L^\infty_x   H^s_{ k+\gamma/2 -4|\al| }}\Vert  \partial^\alpha g_1 \Vert_{L^2_x H^s_{ k+\gamma/2 -4|\al|}}\notag\\
		\leq& \frac{1}{\varepsilon}\Vert  \partial^\alpha g_1 \Vert_{L^2_x L^2_{14}} \Vert \nabla_x g_1 \Vert^\frac{1}{2}_{L^2_x   H^s_{ k+\gamma/2 -4|\al| }}\Vert \nabla^2_x g_1 \Vert^\frac{1}{2}_{L^2_x   H^s_{ k+\gamma/2 -4|\al| }}\Vert  \partial^\alpha g_1 \Vert_{L^2_x H^s_{ k+\gamma/2 -4|\al|}}\notag\\
		\leq& C\sqrt{\mathcal{E}_{N,k}(t)}
		\mathcal{D}_{N,k}(t).
	\end{align}
	Similarly,
	\begin{align}\label{3g12}
		&\frac{1}{\varepsilon}\big(\Vert g_1 \Vert_{L^\infty_x L^2_{14} } \Vert  \partial^\alpha g_1 \Vert_{L^2_xH^s_{ k+\gamma/2 -4|\al|}}^2
		+\Vert  \partial^\alpha g_1 \Vert_{L^2_x L^2_{14}} \Vert \sqrt{\mu }g_2 \Vert_{L^\infty_x   H^s_{ k+\gamma/2 -4|\al| }}\Vert  \partial^\alpha g_1 \Vert_{L^2_x H^s_{ k+\gamma/2 -4|\al|}}\notag\\
		& \qquad+\Vert \sqrt{\mu }g_2 \Vert_{L^\infty_x L^2_{14} } \Vert  \partial^\alpha g_1 \Vert_{L^2_xH^s_{ k+\gamma/2 -4|\al| }}^2\big)\notag\\\leq& C\sqrt{\mathcal{E}_{N,k}(t)}
		\mathcal{D}_{N,k}(t).
	\end{align}
	A direct application of Taylor expansion shows that $\|M-\mu\|_{H^s_{ k+\gamma/2}}$ can be bounded by $|(\rho-1,u,\theta-\frac{3}{2})|$. Combining with \eqref{background}, \eqref{apriori} and \eqref{boundbarG}, we get
	\begin{align}\label{3g13}
		&\frac{1}{\varepsilon}\big(\Vert  \partial^\alpha g_1 \Vert_{L^2_x L^2_{14}} \Vert M-\mu \Vert_{L^\infty_x   H^s_{ k+\gamma/2 -4|\al| }} \Vert  \partial^\alpha g_1 \Vert_{L^2_xH^s_{ k+\gamma/2 -4|\al| }} +\Vert M-\mu \Vert_{L^\infty_x L^2_{14} } \Vert  \partial^\alpha g_1 \Vert_{L^2_xH^s_{ k+\gamma/2 -4|\al| }}^2 \notag\\
		& +\Vert  \partial^\alpha g_1 \Vert_{L^2_x L^2_{14}} \Vert \overline{G} \Vert_{L^\infty_x   H^s_{ k+\gamma/2 -4|\al| }}\Vert  \partial^\alpha g_1 \Vert_{L^2_x H^s_{ k+\gamma/2 -4|\al|}}+\Vert \overline{G} \Vert_{L^\infty_x L^2_{14} } \Vert  \partial^\alpha g_1 \Vert_{L^2_xH^s_{ k+\gamma/2 -4|\al| }}^2\big)\notag\\
		\leq& C(\eta+\bar\eta)\mathcal{D}_{N,k}(t)+C\sqrt{\mathcal{E}_{N,k}(t)}
		\mathcal{D}_{N,k}(t).
	\end{align}
	For the case $|\al|>1$, \eqref{estLD} follows from \eqref{LD2order}, \eqref{3g1}, \eqref{3g12} and \eqref{3g13}.
	For $|\al|=1$, since the term $\Vert \nabla^2 g_1 \Vert^\frac{1}{2}_{L^2_x   H^s_{ k+\gamma/2 -4|\al| }}$ in \eqref{3g1} is not directly bounded by the dissipation norm, we use $L^6-L^3-L^2$ H\"older's inequality to get
	\begin{align}\label{LD2order1}
		&\frac{1}{\varepsilon}(L_D  \partial^\alpha g_1,  w_{2k-8|\alpha|}\partial^\alpha g_1 )   +   \frac{1}{\varepsilon}(Q(g_1+\sqrt\mu g_2+(M-\mu)+\overline{G},  \partial^\alpha g_1),  w_{2k-8|\alpha|}\partial^\alpha g_1 )  
		\notag\\
		\le & - \de\frac{1}{\varepsilon} \Vert \partial^\alpha g_1 \Vert_{L^2_x H^s_{k+\gamma/2 -4|\al|}}^2
		+C_k\frac{1}{\varepsilon}\big(\Vert  \partial^\alpha g_1 \Vert_{L^3_x L^2_{14}} \Vert g_1 \Vert_{L^6_x   H^s_{ k+\gamma/2 -4|\al| }}\Vert  \partial^\alpha g_1 \Vert_{L^2_x H^s_{ k+\gamma/2 -4|\al|}}\notag\\
		& +\Vert g_1 \Vert_{L^\infty_x L^2_{14} } \Vert  \partial^\alpha g_1 \Vert_{L^2_xH^s_{ k+\gamma/2 -4|\al|}}^2
		+\Vert  \partial^\alpha g_1 \Vert_{L^3_x L^2_{14}} \Vert \sqrt{\mu }g_2 \Vert_{L^6_x   H^s_{ k+\gamma/2 -4|\al| }}\Vert  \partial^\alpha g_1 \Vert_{L^2_x H^s_{ k+\gamma/2 -4|\al|}}\notag\\
		& +\Vert \sqrt{\mu }g_2 \Vert_{L^\infty_x L^2_{14} } \Vert  \partial^\alpha g_1 \Vert_{L^2_xH^s_{ k+\gamma/2 -4|\al| }}^2+\Vert  \partial^\alpha g_1 \Vert_{L^2_x L^2_{14}} \Vert M-\mu \Vert_{L^\infty_x   H^s_{ k+\gamma/2 -4|\al| }}\Vert  \partial^\alpha g_1 \Vert_{L^2_x H^s_{ k+\gamma/2 -4|\al|}}\notag\\
		& +\Vert M-\mu \Vert_{L^\infty_x L^2_{14} } \Vert  \partial^\alpha g_1 \Vert_{L^2_xH^s_{ k+\gamma/2 -4|\al| }}^2+\Vert  \partial^\alpha g_1 \Vert_{L^2_x L^2_{14}} \Vert \overline{G} \Vert_{L^\infty_x   H^s_{ k+\gamma/2 -4|\al| }}\Vert  \partial^\alpha g_1 \Vert_{L^2_x H^s_{ k+\gamma/2 -4|\al|}}\notag\\
		& +\Vert \overline{G} \Vert_{L^\infty_x L^2_{14} } \Vert  \partial^\alpha g_1 \Vert_{L^2_xH^s_{ k+\gamma/2 -4|\al| }}^2\big).
	\end{align}
	The last four terms, $\Vert g_1 \Vert_{L^\infty_x L^2_{14} } \Vert  \partial^\alpha g_1 \Vert_{L^2_xH^s_{ k+\gamma/2 -4|\al|}}^2$ and $\Vert \sqrt{\mu }g_2 \Vert_{L^\infty_x L^2_{14} } \Vert  \partial^\alpha g_1 \Vert_{L^2_xH^s_{ k+\gamma/2 -4|\al| }}^2$ are estimated as in \eqref{3g12} and \eqref{3g13}. Then \eqref{estLD} holds by \eqref{3g13}, \eqref{LD2order1} and the fact that
	\begin{align*}
		&\frac{1}{\varepsilon}\big(\Vert  \partial^\alpha g_1 \Vert_{L^3_x L^2_{14}} \Vert g_1 \Vert_{L^6_x   H^s_{ k+\gamma/2 -4|\al| }}\Vert  \partial^\alpha g_1 \Vert_{L^2_x H^s_{ k+\gamma/2 -4|\al|}}\notag\\
		&\qquad\qquad+\Vert  \partial^\alpha g_1 \Vert_{L^3_x L^2_{14}} \Vert \sqrt{\mu }g_2 \Vert_{L^6_x   H^s_{ k+\gamma/2 -4|\al| }}\Vert  \partial^\alpha g_1 \Vert_{L^2_x H^s_{ k+\gamma/2 -4|\al|}}\big)\notag\\
		\leq &\frac{1}{\varepsilon}\big(\Vert \partial^\alpha g_1 \Vert^\frac{1}{2}_{L^2_x L^2_{14}}\Vert \nabla_x \partial^\alpha g_1 \Vert^\frac{1}{2}_{L^2_x L^2_{14}} \Vert\nabla_x g_1 \Vert_{L^2_x   H^s_{ k+\gamma/2 -4|\al| }}\Vert  \partial^\alpha g_1 \Vert_{L^2_x H^s_{ k+\gamma/2 -4|\al|}}\notag\\
		&\qquad\qquad+\Vert\partial^\alpha g_1 \Vert^\frac{1}{2}_{L^2_x L^2_{14}}\Vert \nabla_x \partial^\alpha g_1 \Vert^\frac{1}{2}_{L^2_x L^2_{14}} \Vert\nabla_x \sqrt{\mu }g_2 \Vert_{L^2_x   H^s_{ k+\gamma/2 -4|\al| }}\Vert  \partial^\alpha g_1 \Vert_{L^2_x H^s_{ k+\gamma/2 -4|\al|}}\big)\notag\\
		\leq&C\frac{1}{\sqrt\varepsilon}\sqrt{\mathcal{E}_{N,k}(t)}\mathcal{D}_{N,k}(t).
	\end{align*}
This ends the proof of Lemma \ref{leestLD}.
\end{proof}

For other terms in \eqref{ipg1} containing $Q(\cdot,\pa^{\al'} g_1)$ with $\al'<\al$, we have the following lemma.
\begin{lemma}\label{leestcomm}
	For $1\leq|\alpha|\leq N-1$ and $k\geq 25+4|\al|$, it holds that
	\begin{align}\label{estnonlig1}
		&\frac{1}{\varepsilon}\sum_{\alpha_{1}\leq \alpha, |\alpha_{1}|\geq1}C^{\alpha_{1}}_\alpha (Q(\pa^{\al_1}\{g_1+\sqrt\mu g_2+(M-\mu)+\overline{G}\},  \partial^{\alpha-\al_1} g_1),  w_{2k-8|\alpha|}\partial^\alpha g_1 ) 
		\notag\\
		\le& C_k\eta\mathcal{D}_{N,k}(t)+C_k\frac{1}{\sqrt\varepsilon}\sqrt{\mathcal{E}_{N,k}(t)}
		\mathcal{D}_{N,k}(t).
	\end{align}
\end{lemma}
\begin{proof}
	One can see from the proof of Lemma \ref{leestLD} that the terms containing $M-\mu$ and $\overline{G}$ should be easier to handle. We mainly focus on the first two terms on the left hand side. Rewrite
	\begin{align}\label{I6I7}
		&\frac{1}{\varepsilon}(Q(\pa^{\al_1}g_1,  \partial^{\alpha-\al_1} g_1),  w_{2k-8|\alpha|}\partial^\alpha g_1 )\notag\\
		=&\frac{1}{\varepsilon}(Q(\pa^{\al_1}g_1, w_{k-4|\alpha|} \partial^{\alpha-\al_1} g_1),  w_{k-4|\alpha|}\partial^\alpha g_1 )\notag\\
		&+\frac{1}{\varepsilon}\big\{(Q(\pa^{\al_1}g_1,  \partial^{\alpha-\al_1} g_1),  w_{2k-8|\alpha|}\partial^\alpha g_1 )-(Q(\pa^{\al_1}g_1, w_{k-4|\alpha|} \partial^{\alpha-\al_1} g_1),  w_{k-4|\alpha|}\partial^\alpha g_1 )\big\}\notag\\
		=&I_6+I_7.
	\end{align}
	For $I_6$, we first consider $|\al_1|>1$. Apply Lemma \ref{upperQ} and Sobolev embedding to get
	\begin{align*}
		I_6&\le C_k\frac{1}{\varepsilon}\|\partial ^{\alpha _1} g_1\|_{L^3_xL^2_5}\|\partial^{\alpha-\al_1} g_1\|_{L^6_xH^s_{k-4|\alpha |+\gamma/2+2s}}\|\partial^\alpha g_1\|_{H^s_{k-4|\alpha| +\gamma/2}}\notag\\
		&\le C_k\frac{1}{\varepsilon}\|\partial ^{\alpha _1} g_1\|^\frac{1}{2}_{L^2_xL^2_5}\|\nabla_{x}\partial ^{\alpha _1} g_1\|^\frac{1}{2}_{L^2_xL^2_5}\|\nabla_{x}\partial^{\alpha-\al_1} g_1\|_{L^2_xH^s_{k-4(|\alpha|-1/2 )+\gamma/2}}\|\partial^\alpha g_1\|_{H^s_{k-4|\alpha|+\gamma/2}}\notag\\
		&\le C_k\frac{1}{\sqrt\varepsilon}\sqrt{\mathcal{E}_{N,k}(t)}
		\mathcal{D}_{N,k}(t).
	\end{align*}
	The last inequality above holds since
	\begin{align}\label{I6<1}
		&\|\nabla_{x}\partial^{\alpha-\al_1} g_1\|_{L^2_xH^s_{k-4(|\alpha|-1/2 )+\gamma/2}}\leq\|\nabla_{x}\partial^{\alpha-\al_1} g_1\|_{L^2_xH^s_{k-4(|\alpha-\al_1|+1 )+\gamma/2}},
	\end{align}
	for $|\al_1|>1$.
	
	When $|\al_1|=1$, one has
	\begin{align}\label{I6=1}
		I_6&\le C_k\frac{1}{\varepsilon}\|\partial ^{\alpha _1} g_1\|_{L^\infty_xL^2_5}\|\partial^{\alpha-\al_1} g_1\|_{L^2_xH^s_{k-4|\alpha |+2s+\gamma/2}}\|\partial^\alpha g_1\|_{H^s_{k-4|\alpha|+\gamma/2}}\notag\\
		&\le C_k\frac{1}{\varepsilon}\|\nabla_{x}\partial ^{\alpha _1} g_1\|^\frac{1}{2}_{L^2_xL^2_5}\|\nabla^2_{x}\partial ^{\alpha _1} g_1\|^\frac{1}{2}_{L^2_xL^2_5}\|\partial^{\alpha-\al_1} g_1\|_{L^2_xH^s_{k-4(|\alpha-\al_1| )+\gamma/2}}\|\partial^\alpha g_1\|_{H^s_{k-4|\alpha|+\gamma/2}}\notag\\
		&\le C_k\frac{1}{\sqrt\varepsilon}\sqrt{\mathcal{E}_{N,k}(t)}
		\mathcal{D}_{N,k}(t).
	\end{align}
	
	Combining two cases, we have
	\begin{align}\label{I6}
		I_6&\le C_k\frac{1}{\sqrt\varepsilon}\sqrt{\mathcal{E}_{N,k}(t)}
		\mathcal{D}_{N,k}(t).
	\end{align}

	For $I_7$, we use Lemma \ref{commu} to get
	\begin{align}\label{I7}
		I_7&\leq C_k\frac{1}{\varepsilon}\|\partial ^{\alpha _1} g_1\|_{L^3_xL^2_{14}}\|\partial^{\alpha-\al_1} g_1\|_{L^6_xH^s_{k-4|\alpha |+\gamma/2}}\|\partial^\alpha g_1\|_{H^s_{k-4|\alpha|  +\gamma/2}}\notag\\
		&\quad+C_k\frac{1}{\varepsilon}\|\partial ^{\alpha _1} g_1\|_{L^3_xH^s_{\gamma/2}}\|w_{k-4|\alpha |}\partial^{\alpha-\al_1} g_1\|_{L^6_xL^2_{14}}\|\partial^\alpha g_1\|_{H^s_{k-4|\alpha|  +\gamma/2}}\notag\\
		&\leq C_k\frac{1}{\varepsilon}\|\partial ^{\alpha _1} g_1\|^\frac{1}{2}_{L^2_xL^2_{14}}\|\nabla_{x}\partial ^{\alpha _1} g_1\|^\frac{1}{2}_{L^2_xL^2_{14}}\|\nabla_{x}\partial^{\alpha-\al_1} g_1\|_{L^2_xH^s_{k-4|\alpha |+\gamma/2}}\|\partial^\alpha g_1\|_{H^s_{k-4|\alpha|+\gamma/2}}\notag\\
		&\quad+C_k\frac{1}{\varepsilon}\|\partial ^{\alpha _1} g_1\|^\frac{1}{2}_{L^2_xH^s_{\gamma/2}}\|\nabla_{x}\partial ^{\alpha _1} g_1\|^\frac{1}{2}_{L^2_xH^s_{\gamma/2}}\|w_{k-4|\alpha |}\nabla_{x}\partial^{\alpha-\al_1} g_1\|_{L^2_xL^2_{14}}\|\partial^\alpha g_1\|_{H^s_{k-4|\alpha|+\gamma/2}}\notag\\
		&\le C_k\frac{1}{\sqrt\varepsilon}\sqrt{\mathcal{E}_{N,k}(t)}
		\mathcal{D}_{N,k}(t).
	\end{align}
	The combination of \eqref{I6} and \eqref{I7} gives
	\begin{align}\label{g1g1g1}
		&\frac{1}{\varepsilon}(Q(\pa^{\al_1}g_1,  \partial^{\alpha-\al_1} g_1),  w_{2k-8|\alpha|}\partial^\alpha g_1 )\notag\\
		\le& C_k\frac{1}{\sqrt\varepsilon}\sqrt{\mathcal{E}_{N,k}(t)}
		\mathcal{D}_{N,k}(t).
	\end{align}
	The second term on the left hand side of \eqref{estnonlig1} can be estimated in the same way. We also write
	\begin{align}\label{g2g1g11}
		&\frac{1}{\varepsilon}(Q(\pa^{\al_1}\sqrt\mu g_2,  \partial^{\alpha-\al_1} g_1),  w_{2k-8|\alpha|}\partial^\alpha g_1 )\notag\\
		=&\frac{1}{\varepsilon}(Q(\pa^{\al_1}\sqrt\mu g_2, w_{k-4|\alpha|} \partial^{\alpha-\al_1} g_1),  w_{k-4|\alpha|}\partial^\alpha g_1 )\notag\\
		&+\frac{1}{\varepsilon}\big\{(Q(\pa^{\al_1}\sqrt\mu g_2,  \partial^{\alpha-\al_1} g_1),  w_{2k-8|\alpha|}\partial^\alpha g_1 )-(Q(\pa^{\al_1}\sqrt\mu g_2, w_{k-4|\alpha|} \partial^{\alpha-\al_1} g_1),  w_{k-4|\alpha|}\partial^\alpha g_1 )\big\}\notag\\
		=&I_8+I_9.
	\end{align}
	Similar arguments in \eqref{I6} and \eqref{I7} show that
	\begin{align}\label{I8}
		I_8\leq C_k\frac{1}{\sqrt\varepsilon}\sqrt{\mathcal{E}_{N,k}(t)}
		\mathcal{D}_{N,k}(t),
	\end{align}
	and
	\begin{align}\label{I91}
		I_9&\leq C_k\frac{1}{\varepsilon}\|\partial ^{\alpha _1} \sqrt{\mu}g_2\|_{L^3_xL^2_{14}}\|\partial^{\alpha-\al_1} g_1\|_{L^6_xH^s_{k-4|\alpha |+\gamma/2}}\|\partial^\alpha g_1\|_{H^s_{k-4|\alpha|+\gamma/2}}\notag\\
		&\quad+C_k\frac{1}{\varepsilon}\|\partial ^{\alpha _1} \sqrt{\mu}g_2\|_{L^3_xH^s_{\gamma/2}}\|w_{k-4|\alpha |}\partial^{\alpha-\al_1} g_1\|_{L^6_xL^2_{14}}\|\partial^\alpha g_1\|_{H^s_{k-4|\alpha|+\gamma/2}}\notag\\
		&\leq C_k\frac{1}{\varepsilon}\|\partial ^{\alpha _1}g_2\|^\frac{1}{2}\|\nabla_x\partial ^{\alpha _1}g_2\|^\frac{1}{2}\|\nabla_x\partial^{\alpha-\al_1} g_1\|_{L^2_xH^s_{k-4|\alpha |+\gamma/2}}\|\partial^\alpha g_1\|_{H^s_{k-4|\alpha| +\gamma/2}}\notag\\
		&\quad+C_k\frac{1}{\varepsilon}\|\partial ^{\alpha _1} \sqrt{\mu}g_2\|_{L^3_xH^s_{\gamma/2}}\|w_{k-4|\alpha |}\nabla_x\partial^{\alpha-\al_1} g_1\|_{L^2_xL^2_{14}}\|\partial^\alpha g_1\|_{H^s_{k-4|\alpha|+\gamma/2}}.
	\end{align}
	For the second term on the right hand side above, notice that
	\begin{align}\label{I92}
		\|\partial ^{\alpha _1} \sqrt{\mu}g_2\|_{L^3_xH^s_{\gamma/2}}
		\leq& \|\partial ^{\alpha _1} \sqrt{\mu}g_2\|^\frac{1}{2}_{L^2_xH^s_{\gamma/2}}\|\nabla_x\partial ^{\alpha _1} \sqrt{\mu}g_2\|^\frac{1}{2}_{L^2_xH^s_{\gamma/2}}\notag\\
		\leq& (\|\partial ^{\alpha _1} \sqrt{\mu}\mathbf{P}_0g_2\|_{L^2_xH^s_{\gamma/2}}+\|\partial ^{\alpha _1} \sqrt{\mu}\mathbf{P}_1g_2\|_{L^2_xH^s_{\gamma/2}})^\frac{1}{2}\notag\\
		&\times(\|\nabla_x\partial ^{\alpha _1} \sqrt{\mu}\mathbf{P}_0g_2\|_{L^2_xH^s_{\gamma/2}}+\|\nabla_x\partial ^{\alpha _1} \sqrt{\mu}\mathbf{P}_1g_2\|_{L^2_xH^s_{\gamma/2}})^\frac{1}{2},
	\end{align}
	where $\mathbf{P}_0$ and $\mathbf{P}_1$ are defined in \eqref{defP0} and \eqref{defP1}. We further get
	\begin{align*}
		\|\partial ^{\alpha _1} \sqrt{\mu}\mathbf{P}_0g_2\|_{L^2_xH^s_{\gamma/2}}&\leq C\|\partial ^{\alpha _1} g_2\|\leq C\sqrt{\mathcal{E}_{N,k}(t)},\\
		\|\partial ^{\alpha _1} \sqrt{\mu}\mathbf{P}_1g_2\|_{L^2_xH^s_{\gamma/2}}&\leq C\|\pa^{\al_1}\mathbf{P}_1g_2\|_{L^2_xL^2_{v,D}}\leq C\sqrt\varepsilon\sqrt{\mathcal{D}_{N,k}(t)},\\
		\|\nabla_x\partial ^{\alpha _1} \sqrt{\mu}\mathbf{P}_0g_2\|_{L^2_xH^s_{\gamma/2}}&\leq C\|\nabla_x\partial ^{\alpha _1} g_2\|\leq C\frac{1}{\varepsilon}\sqrt{\mathcal{E}_{N,k}(t)},
	\end{align*}
	and
	\begin{align*}
		\|\nabla_x\partial ^{\alpha _1} \sqrt{\mu}\mathbf{P}_1g_2\|_{L^2_xH^s_{\gamma/2}}
		\leq& C\|\nabla_x\partial ^{\alpha _1}\mathbf{P}_1g_2\|_{L^2_xL^2_{v,D}}\leq C\frac{1}{\sqrt\varepsilon}\sqrt{\mathcal{D}_{N,k}(t)},
	\end{align*}
	which, combined with \eqref{I91}, \eqref{I92}, Sobolev embedding and \eqref{apriori}, yield
	\begin{align}\label{I9}
		I_9&\leq C_k\frac{1}{\sqrt\varepsilon}\sqrt{\mathcal{E}_{N,k}(t)}
		\mathcal{D}_{N,k}(t).
	\end{align}
	It follows from \eqref{g2g1g11}, \eqref{I8} and \eqref{I9} that
	\begin{align}\label{g2g1g1}
		&\frac{1}{\varepsilon}(Q(\pa^{\al_1}\sqrt\mu g_2,  \partial^{\alpha-\al_1} g_1),  w_{2k-8|\alpha|}\partial^\alpha g_1 )
		\leq C_k\frac{1}{\sqrt\varepsilon}\sqrt{\mathcal{E}_{N,k}(t)}
		\mathcal{D}_{N,k}(t).
	\end{align}
	For $(Q(\pa^{\al_1}\{(M-\mu)+\overline{G}\},  \partial^{\alpha-\al_1} g_1),  w_{2k-8|\alpha|}\partial^\alpha g_1 )$, a direct calculation shows that
	\begin{align}\label{1M}
		\partial_{x_i}M=M\big(\frac{\partial_{x_i}\rho}{\rho}+\frac{(v-u)\cdot\partial_{x_i}u}{R\theta}
		+(\frac{|v-u|^{2}}{2R\theta}-\frac{3}{2})\frac{\partial_{x_i}\theta}{\theta} \big).
	\end{align}
	Moreover, for $|\alpha|\geq 2$ with $\partial^{\alpha}=\partial^{\alpha'}\partial_{x_i}$, it holds that
	\begin{align}\label{2M}
		\partial^{\alpha}M=&M\big(\frac{\partial^{\alpha}\rho}{\rho}+\frac{(v-u)\cdot\partial^{\alpha}u}{R\theta}
		+(\frac{|v-u|^{2}}{2R\theta}-\frac{3}{2})\frac{\partial^{\alpha}\theta}{\theta} \big)
		\nonumber\\
		&+\sum_{\alpha_{1}\leq \alpha',|\al_1|\geq1}C^{\alpha_1}_{\alpha'}\big(\partial^{\alpha_{1}}(M\frac{1}{\rho})\partial^{\alpha'-\alpha_{1}}\partial_{x_i}\rho+\partial^{\alpha_{1}}(M\frac{v-u}{R\theta})\cdot\partial^{\alpha'-\alpha_{1}}\partial_{x_i}u
		\nonumber\\
		&\hspace{2cm}+\partial^{\alpha_{1}}(M\frac{|v-u|^{2}}{2R\theta^{2}}-M\frac{3}{2\theta})\partial^{\alpha'-\alpha_{1}}\partial_{x_i}\theta\big).
	\end{align}	
	Then using \eqref{background}, \eqref{apriori}, \eqref{1M}, \eqref{2M} and similar arguments in \eqref{3g13}, we obtain
	\begin{align}\label{restg1g1}
		&\frac{1}{\varepsilon}(Q(\pa^{\al_1}\{(M-\mu)+\overline{G}\},  \partial^{\alpha-\al_1} g_1),  w_{2k-8|\alpha|}\partial^\alpha g_1 )\notag\\
		\leq&C\frac{1}{\varepsilon}\Vert  \partial^{\alpha-\al_1} g_1 \Vert_{L^6_x L^2_{14}} \Vert \pa^{\al_1}\{(M-\mu)+\overline{G}\} \Vert_{L^3_x   H^s_{ k+\gamma/2 -4|\al| }}\Vert\partial^\alpha g_1 \Vert_{L^2_x H^s_{ k+\gamma/2 -4|\al|}} \notag\\
		& +C\frac{1}{\varepsilon}\Vert \pa^{\al_1}\{(M-\mu)+\overline{G}\} \Vert_{L^3_x L^2_{14} } \Vert  \partial^{\alpha-\al_1} g_1 \Vert_{L^6_xH^s_{ k+\gamma/2 -4|\al| }}^2 \Vert\partial^\alpha g_1 \Vert_{L^2_x H^s_{ k+\gamma/2 -4|\al|}}\notag\\
		\leq& C_k\eta\mathcal{D}_{N,k}(t)+C_k\frac{1}{\sqrt\varepsilon}\sqrt{\mathcal{E}_{N,k}(t)}
		\mathcal{D}_{N,k}(t).
	\end{align}
	Hence, \eqref{estnonlig1} holds by \eqref{g1g1g1}, \eqref{g2g1g1} and \eqref{restg1g1}.
\end{proof}
One sees from \eqref{ipg1}, \eqref{estLD} and \eqref{estnonlig1} that we have already bounded 
\begin{align*}
	&\frac{1}{\varepsilon}(\CL_D \partial^{\alpha}g_1,w_{2k-8|\al|}\partial^{\alpha}g_1)+\frac{1}{\varepsilon}(\partial^{\alpha}Q(g_1,g_1),w_{2k-8|\al|}\partial^{\alpha}g_1)+\frac{1}{\varepsilon}(\partial^{\alpha}Q(\sqrt{\mu}g_2,g_1),w_{2k-8|\al|}\partial^{\alpha}g_1)\notag\\&
	+\frac{1}{\varepsilon}(\partial^{\alpha}Q(M-\mu,g_1),w_{2k-8|\al|}\partial^{\alpha}g_1)
	+\frac{1}{\varepsilon}(\partial^{\alpha}Q(\overline{G},g_1),w_{2k-8|\al|}\partial^{\alpha}g_1),
\end{align*}
on the right hand side of \eqref{ipg1}. Now we turn to the rest part.
\begin{lemma}\label{leestrestg1}
	For $1\leq|\alpha|\leq N-1$ and $k\geq 25+4|\al|$, it holds that
	\begin{align}\label{estrestg1}
		&\frac{1}{\varepsilon}\big((\partial^{\alpha}Q(g_1,\sqrt{\mu}g_2+M-\mu+\overline{G}),w_{2k-8|\al|}\partial^{\alpha}g_1)\big)
		\le  C_k\eta\mathcal{D}_{N,k}(t)+C_k\frac{1}{\sqrt\varepsilon}\sqrt{\mathcal{E}_{N,k}(t)}
		\mathcal{D}_{N,k}(t).
	\end{align}
\end{lemma}
\begin{proof}
	Rewrite
	\begin{align*}
		&\frac{1}{\varepsilon}\big((\partial^{\alpha}Q(g_1,\sqrt{\mu}g_2+M-\mu+\overline{G}),w_{2k-8|\al|}\partial^{\alpha}g_1)\big)\notag\\
		=&\frac{1}{\varepsilon}\sum_{\alpha_{1}\leq \alpha}C^{\alpha_{1}}_\alpha (Q(\partial^{\al_1} g_1,\pa^{\alpha-\al_1}\{\sqrt{\mu}g_2+(M-\mu)+\overline{G}\}),  w_{2k-8|\alpha|}\partial^\alpha g_1 ) .
	\end{align*}
	Then it follows by \eqref{fgh} that
	\begin{align}\label{g1g2g1}
		&(Q(\partial^{\al_1} g_1,\pa^{\alpha-\al_1}\{\sqrt{\mu}g_2+(M-\mu)+\overline{G}\}),  w_{2k-8|\alpha|}\partial^\alpha g_1 )\notag\\
		\leq& C_k \Vert \partial^{\al_1}g_1 \Vert_{L^3_xL^2_{14}} \Vert \pa^{\alpha-\al_1}\{\sqrt{\mu}g_2+(M-\mu)+\overline{G}\} \Vert_{L^6_xH^s_{k-4|\al|+\gamma/2+2s }}  \Vert\partial^\alpha g_1\Vert_{L^2_xH^s_{k-4|\al|+\gamma/2}}\notag\\
		&+  \Vert \pa^{\alpha-\al_1}\{\sqrt{\mu}g_2+(M-\mu)+\overline{G}\} \Vert_{L^6_xL^2_{14}}  \Vert \partial^{\al_1}g_1 \Vert_{L^3_xH^s_{k-4|\al|+\gamma/2}}  \Vert\partial^\alpha g_1\Vert_{L^2_xH^s_{k-4|\al|+\gamma/2}}.
	\end{align}
	Using
	\begin{align}\label{g2split}
		\Vert \pa^{\alpha-\al_1}\{\sqrt{\mu}g_2\} \Vert_{L^6_xH^s_{k-4|\al|+\gamma/2+2s }}\leq C\|\nabla_x\partial ^{\alpha-\al_1} g_2\|_{L^2_xL^2_{v,D}},
	\end{align}
	\eqref{rhoutheta} and similar argument as in \eqref{restg1g1}, we obtain \eqref{estrestg1}.
\end{proof}
Combining \eqref{ipg1}, \eqref{estLD}, \eqref{estnonlig1}, \eqref{estrestg1} and \eqref{apriori}, we directly get the following lemma.
\begin{lemma}\label{leloworderg1}
	For $1\leq|\alpha|\leq N-1$ and $k\geq 25+4|\al|$, it holds that
	\begin{align}\label{loworderg1}
		&\frac{1}{2}\frac{d}{dt}\|w_{k-4|\al|}\partial^{\alpha}g_1\|^{2}+c \frac{1}{\varepsilon}\Vert \partial^\alpha g_1 \Vert_{H^s_{k-4|\alpha |+\ga/2}}^2\leq C_k(\eta+\bar\eta+\eta^\frac{1}{2}_0)\mathcal{D}_{N,k}(t).
	\end{align}
\end{lemma}
We want to derive similar bound for $g_2$. 
\begin{lemma}\label{leloworderg2}
	For $1\leq|\alpha|\leq N-1$ and $k\geq 25+4|\al|$, it holds that
	\begin{align}\label{loworderg2}
		&\frac{1}{2}\frac{d}{dt}\|\partial^{\alpha}g_2\|^{2}+c \frac{1}{\varepsilon}\Vert \partial^\alpha \mathbf{P}_1g_2 \Vert_{L^2_xL^2_{v,D}}^2-C_k\frac{1}{\varepsilon}\Vert \partial^\alpha g_1 \Vert_{H^s_{k-4|\alpha |+\ga/2}}^2\notag\\
		\leq& C_k\varepsilon\{\|\partial^{\alpha}(\nabla_{x}\widetilde{u},\nabla_{x}\widetilde{\theta})\|^{2}+\|\partial^{\alpha}\nabla_{x}g_1(t)\|_{L^2_xH^s_{k-4N+\ga/2}}^{2}+\|\pa^\al\nabla_{x}\mathbf{P}_1g_2\|^2_{L^2_xL^2_{v,D}}\}\notag\\ &+C(\eta+\bar\eta+\eta_0^\frac{1}{2})\mathcal{D}_{N,k}(t)+C\eta\varepsilon^r.
	\end{align}
\end{lemma}
\begin{proof}
	From 
	\eqref{g2}, we take derivatives and inner product, and use \eqref{coercive} to get
	\begin{align}\label{ipg2}
		&\frac{1}{2}\frac{d}{dt}\|\partial^{\alpha}g_2\|^{2}+
		c \frac{1}{\varepsilon}\|\partial^{\alpha}\mathbf{P}_1g_2\|_{L^2_xL^2_{v,D}}^{2}\notag\\
		\leq&\frac{1}{\varepsilon}(\partial^{\alpha}\CL_B g_1+\partial^{\alpha}\Gamma(\frac{M-\mu}{\sqrt{\mu}},g_2)
		+\partial^{\alpha}\Gamma(g_2,\frac{M-\mu}{\sqrt{\mu}})+\partial^{\alpha}\Gamma(\frac{\overline{G}+\sqrt{\mu}g_2}{\sqrt{\mu}},\frac{\overline{G}+\sqrt{\mu}g_2}{\sqrt{\mu}}),\partial^{\alpha}g_2)
		\nonumber\\
		&+(\frac{\partial^{\alpha}P_{0}(v\cdot\nabla_{x}(g_1+\sqrt{\mu}g_2))}{\sqrt{\mu}},\partial^{\alpha}g_2)
		-(\frac{\partial^{\alpha}P_{1}(v\cdot\nabla_{x}\overline{G})}{\sqrt{\mu}},\partial^{\alpha}g_2)
		-(\frac{\partial^{\alpha}\partial_{t}\overline{G}}{\sqrt{\mu}},\partial^{\alpha}g_2)
		\nonumber\\
		&-(\frac{1}{\sqrt{\mu}}\partial^{\alpha}P_{1}\{v\cdot(\frac{|v-u|^{2}
			\nabla_{x}\widetilde{\theta}}{2R\theta^{2}}
		+\frac{(v-u)\cdot\nabla_{x}\widetilde{u}}{R\theta})M\},\partial^{\alpha}g_2),
	\end{align}
	with $1\leq|\alpha|\leq N-1$. In order to prove the lemma, it is required to bound each term in \eqref{ipg2}. Fortunately, almost all terms can be estimated in a similar way to how we prove Lemma 
	\ref{leloworderg1}. 
	
	Using the definition of $\CL_B$ in \eqref{defLB}, we have
	\begin{align}\label{1Lb}
		\frac{1}{\varepsilon}(\partial^{\alpha}\CL_B g_1,\partial^{\alpha}g_2)&\leq C_k\frac{1}{\varepsilon}\int_{\R^3}\int_{\R^3}|\partial^{\alpha}g_1\partial^{\alpha}g_2|dxdv\notag\\
		&\leq C_k\frac{1}{\varepsilon}\|w_{k-4|\al|}\partial^{\alpha}g_1\|(\|\partial^{\alpha}\mathbf{P}_0g_2\|+\|\partial^{\alpha}\mathbf{P}_1g_2\|_{L^2_xL^2_{v,D}}).
	\end{align}
	From \eqref{1Lb} and \eqref{g2macro}, one gets
	\begin{align}\label{Lb}
		\frac{1}{\varepsilon}(\partial^{\alpha}\CL_B g_1,\partial^{\alpha}g_2)\leq C_{k,\ka}\frac{1}{\varepsilon}\|w_{k-4|\al|}\partial^{\alpha}g_1\|^2+\ka\frac{1}{\varepsilon}\|\partial^{\alpha}\mathbf{P}_1g_2\|^2_{L^2_xL^2_{v,D}},
	\end{align}
	for any $0<\ka<1$. Then we use the approach in \eqref{3g13} to estimate $(\partial^{\alpha}\Gamma(\frac{M-\mu}{\sqrt{\mu}},g_2)
	,\partial^{\alpha}g_2)$. By \eqref{Trilinear}, \eqref{1M}, \eqref{2M}, \eqref{background} and \eqref{apriori}, it holds that
	\begin{align}\label{Mg2g2}
		\frac{1}{\varepsilon}(\partial^{\alpha}\Gamma(\frac{M-\mu}{\sqrt{\mu}},g_2)
		,\partial^{\alpha}g_2)\leq& C\frac{1}{\varepsilon}\sum_{\alpha_{1}\leq \alpha}\int_{\R^3}\Vert \pa^{\al_1}\big(\frac{M-\mu}{\sqrt{\mu}}\big)\Vert_{L^2_v} \Vert \pa^{\al-\al_1}g_2 \Vert_{L^2_{v,D}} \Vert \partial^{\alpha}g_2 \Vert_{L^2_{v,D}}dx\notag\\
		\leq& C\frac{1}{\varepsilon}\sum_{\alpha_{1}\leq \alpha,|\al_1|<1}\Vert \pa^{\al_1}\big(\frac{M-\mu}{\sqrt{\mu}}\big)\Vert_{L^\infty_xL^2_v} \Vert \pa^{\al-\al_1}g_2 \Vert_{L^2_xL^2_{v,D}} \Vert \partial^{\alpha}g_2 \Vert_{L^2_xL^2_{v,D}}\notag\\
		&+C\frac{1}{\varepsilon}\sum_{\alpha_{1}\leq \alpha,|\al_1|\geq1}\Vert \pa^{\al_1}\big(\frac{M-\mu}{\sqrt{\mu}}\big)\Vert_{L^3_xL^2_v} \Vert \pa^{\al-\al_1}g_2 \Vert_{L^6_xL^2_{v,D}} \Vert \partial^{\alpha}g_2 \Vert_{L^2_xL^2_{v,D}}\notag\\
		\leq&  C(\eta+\bar\eta)\mathcal{D}_{N,k}(t)+C\frac{1}{\sqrt\varepsilon}\sqrt{\mathcal{E}_{N,k}(t)}
		\mathcal{D}_{N,k}(t).
	\end{align}
	Likewise,
	\begin{align}\label{g2Mg2}
		&\frac{1}{\varepsilon}(\partial^{\alpha}\Gamma(g_2,\frac{M-\mu}{\sqrt{\mu}})
		,\partial^{\alpha}g_2)
		\leq  C(\eta+\bar\eta)\mathcal{D}_{N,k}(t)+C\frac{1}{\sqrt\varepsilon}\sqrt{\mathcal{E}_{N,k}(t)}
		\mathcal{D}_{N,k}(t),
	\end{align}
	and
	\begin{align}\label{Gg2g2}
		&\frac{1}{\varepsilon}(\partial^{\alpha}\Gamma(\frac{\overline{G}}{\sqrt{\mu}},g_2)
		,\partial^{\alpha}g_2)+\frac{1}{\varepsilon}(\partial^{\alpha}\Gamma(g_2,\frac{\overline{G}}{\sqrt{\mu}})
		,\partial^{\alpha}g_2)+\frac{1}{\varepsilon}(\partial^{\alpha}\Gamma(g_2,g_2)
		,\partial^{\alpha}g_2)\notag\\
		\leq & C\eta\mathcal{D}_{N,k}(t)+C\frac{1}{\sqrt\varepsilon}\sqrt{\mathcal{E}_{N,k}(t)}
		\mathcal{D}_{N,k}(t).
	\end{align} 
	Moreover, it follows from \eqref{boundbarG}, \eqref{background}, \eqref{apriori} and Cauchy-Schwarz inequality that
	\begin{align}\label{GGg2}
		&\frac{1}{\varepsilon}(\partial^{\alpha}\Gamma(\frac{\overline{G}}{\sqrt{\mu}},\frac{\overline{G}}{\sqrt{\mu}})
		,\partial^{\alpha}g_2)\notag\\\leq& C\frac{1}{\varepsilon}\sum_{\alpha_{1}\leq \alpha}\Vert \pa^{\al_1}\big(\frac{\overline{G}}{\sqrt{\mu}}\big)\Vert_{L^6_xL^2_v} \Vert \pa^{\al-\al_1}\big(\frac{\overline{G}}{\sqrt{\mu}}\big) \Vert_{L^3_xL^2_{v,D}} \Vert \partial^{\alpha}g_2 \Vert_{L^2_xL^2_{v,D}}\notag\\
		\leq&  C\eta\mathcal{D}_{N,k}(t)+C\eta\varepsilon^r+C\sqrt{\mathcal{E}_{N,k}(t)}
		\mathcal{D}_{N,k}(t).
	\end{align}
	Up to now, we finish the estimates of the first four terms on the right hand side of \eqref{ipg2}, which are very similar to how we prove \eqref{estnonlig1}. We turn to the rest part, which has slightly different structure compared to \eqref{estnonlig1}. Using the definition of $P_0$ in \eqref{defP}, we have
	\begin{align*}
		(\frac{\partial^{\alpha}P_{0}(v\cdot\nabla_{x}(g_1+\sqrt{\mu}g_2))}{\sqrt{\mu}},\partial^{\alpha}g_2)
		=&\int_{\R^3}\int_{\R^3}\mu^{-\frac{1}{2}}\partial^{\alpha}\big\{\sum_{i=0}^{4}\langle v\cdot\nabla_{x}(g_1+\sqrt{\mu}g_2),\frac{\chi_{i}}{M}\rangle\chi_{i}\big\}\partial^{\alpha}g_2dxdv\notag\\
		\leq& C\int_{\R^3}\int_{\R^3}\mu^{\frac{1}{4}}\|w_{3}\partial^{\alpha}\nabla_{x}(g_1+\sqrt{\mu}g_2)\|_{L^2_v}|w_{-2}\partial^{\alpha}g_2|dxdv\notag\\
		&+C\eta\varepsilon^r+C\frac{1}{\sqrt\varepsilon}\sqrt{\mathcal{E}_{N,k}(t)}
		\mathcal{D}_{N,k}(t),
	\end{align*}
	where $C\eta\varepsilon^r+C\frac{1}{\sqrt\varepsilon}\sqrt{\mathcal{E}_{N,k}(t)}
	\mathcal{D}_{N,k}(t)$ controls all the low order derivative terms such as $|(\pa^{\al_1}\rho\nabla_{x}\|\pa^{\al-\al_1}g_1\|_{L^2_v},\|\partial^{\alpha}g_2\|_{L^2_v})|$ with $|\al_1|>1$ by $L^\infty-L^2-L^2$ or $L^6-L^3-L^2$ H\"older's inequality, \eqref{background} and \eqref{apriori}. Recalling from 
	\eqref{g2macro} that $\|w_l\pa^\al\mathbf{P}_0g_2\|$ can be controlled by $\|w_j\pa^\al g_1\|$ for any $j,k\in \R$, by the fact that
	\begin{align*}
		\int_{\R^3}\int_{\R^3}\mu^{\frac{1}{4}}\|w_{3}\partial^{\alpha}\nabla_{x}(g_1+\sqrt{\mu}g_2)\|_{L^2_v}||w_{-2}\partial^{\alpha}g_2|dxdv\leq C\varepsilon\mathcal{D}_{N,k}(t)
	\end{align*}
	for $|\al|<N-1$ and
	\begin{align*}
		&\int_{\R^3}\int_{\R^3}|w_{3}\partial^{\alpha}\nabla_{x}(g_1+\sqrt{\mu}g_2)||w_{-2}\partial^{\alpha}g_2|dxdv\notag\\
		\leq& \ka\frac{1}{\varepsilon}(\|\partial^{\alpha}g_1(t)\|_{L^2_xH^s_{\ga/2}}^{2}+\|\pa^\al\mathbf{P}_1g_2\|^2_{L^2_xL^2_{v,D}})+C_\ka \varepsilon\|w_{3}\pa^\al\nabla_{x}(g_1,\sqrt{\mu}g_2)\|^2,
	\end{align*}
	for $|\al|=N-1$ and $0<\ka<1$, it holds that
	\begin{align}\label{g1g2}
		&(\frac{\partial^{\alpha}P_{0}(v\cdot\nabla_{x}(g_1+\sqrt{\mu}g_2))}{\sqrt{\mu}},\partial^{\alpha}g_2)\notag\\
		\leq& \ka\frac{1}{\varepsilon}\{\|\partial^{\alpha}g_1(t)\|_{L^2_xH^s_{\ga/2}}^{2}+\|\pa^\al\mathbf{P}_1g_2\|^2_{L^2_xL^2_{v,D}}\}+C_\ka \varepsilon\{\|\partial^{\alpha}\nabla_{x}g_1(t)\|_{L^2_xH^s_{3}}^{2}+\|\pa^\al\nabla_{x}\mathbf{P}_1g_2\|^2_{L^2_xL^2_{v,D}}\}\notag\\
		&\quad+C\eta\varepsilon^r+C\frac{1}{\sqrt\varepsilon}\sqrt{\mathcal{E}_{N,k}(t)}
		\mathcal{D}_{N,k}(t).
	\end{align}
	By \eqref{boundbarG}, \eqref{background}, \eqref{apriori} and similar argument in \eqref{I41}, one has 
	\begin{align}\label{Gg2}
		(\frac{\partial^{\alpha}P_{1}(v\cdot\nabla_{x}\overline{G})}{\sqrt{\mu}},\partial^{\alpha}g_2)
		+(\frac{\partial^{\alpha}\partial_{t}\overline{G}}{\sqrt{\mu}},\partial^{\alpha}g_2)\leq C\eta\varepsilon^r.
	\end{align}
	Likewise, 
	\begin{align}\label{thetag2}
		&(\frac{1}{\sqrt{\mu}}\partial^{\alpha}P_{1}\{v\cdot(\frac{|v-u|^{2}
			\nabla_{x}\widetilde{\theta}}{2R\theta^{2}}
		+\frac{(v-u)\cdot\nabla_{x}\widetilde{u}}{R\theta})M\},\partial^{\alpha}g_2)\notag\\
		\leq& \ka\frac{1}{\varepsilon}(\|\partial^{\alpha}g_1(t)\|_{L^2_xH^s_{\ga/2}}^{2}+\|\pa^\al\mathbf{P}_1g_2\|^2_{L^2_xL^2_{v,D}})+ C_{\ka}\varepsilon\|\partial^{\alpha}(\nabla_{x}\widetilde{u},\nabla_{x}\widetilde{\theta})\|^{2}\notag\\
		&+C\eta\varepsilon^r+C\frac{1}{\sqrt\varepsilon}\sqrt{\mathcal{E}_{N,k}(t)}
		\mathcal{D}_{N,k}(t).
	\end{align}
	Thus, we obtain \eqref{loworderg2} by \eqref{ipg2}, \eqref{Lb}, \eqref{Mg2g2}, \eqref{g2Mg2}, \eqref{Gg2g2}, \eqref{GGg2}, \eqref{g1g2}, \eqref{Gg2}, \eqref{thetag2}, \eqref{apriori} and choosing $\ka$ to be small enough.
\end{proof}

\section{Highest order derivatives estimates}\label{sec.8}
This section is devoted to obtaining the $N-$order space derivatives estimates. In order to overcome the increase of spacial derivative from the term ${P_{0}[v\cdot\nabla_{x}(g_1+\sqrt{\mu}g_2)]}/{\sqrt{\mu}}$ in the equation of $g_2$, we separate our proof into two parts, one is the pure $g_1$ estimate, for the other, we need to derive and make use of the equation for $M+\overline{G}+\sqrt{\mu}g_2$. 
\begin{lemma}
	For $k\geq 25+4|\al|$, it holds that
	\begin{align}\label{Norderg1}
		\frac{1}{2}\frac{d}{dt}\varepsilon^2\sum_{|\alpha|=N}\|w_{k-4N}\partial^{\alpha}g_1\|^{2}+c {\varepsilon}\sum_{|\alpha|=N}\Vert \partial^\alpha g_1 \Vert_{L^2_xH^s_{k-4N+\ga/2}}^2
		\leq C_k(\eta+\bar\eta+\eta_0^{\frac{1}{2}})\mathcal{D}_{N,k}(t).
	\end{align}
\end{lemma}
\begin{proof}
	The proof is similar to how we prove Lemma \ref{leestLD} to Lemma \ref{leestrestg1}. However, we need to be extra careful here to avoid the increase of highest order derivatives. Notice that \eqref{LD2order} still holds for $|\al|=N$. Similar arguments as in \eqref{3g1}, \eqref{3g12} and \eqref{3g13} give
	\begin{align}\label{estNLD}
		&\varepsilon\sum_{|\alpha|=N}(L_D  \partial^\alpha g_1+Q(g_1+\sqrt\mu g_2+(M-\mu)+\overline{G},  \partial^\alpha g_1),  w_{2k-8N}\partial^\alpha g_1 ) 
		\notag\\
		\le&   - c \varepsilon\sum_{|\alpha|=N}\Vert \partial^\alpha g_1 \Vert_{L^2_xH^s_{k-4N+\ga/2}}^2
		+C_k(\eta+\bar\eta)\mathcal{D}_{N,k}(t)+C_k\frac{1}{\sqrt\varepsilon}\sqrt{\mathcal{E}_{N,k}(t)}
		\mathcal{D}_{N,k}(t).
	\end{align}
	Then we need to derive similar estimate as in Lemma \ref{leestcomm}, which requires us to bound
	$$
	\frac{1}{\varepsilon}\sum_{\alpha_{1}\leq \alpha, |\alpha_{1}|\geq1}C^{\alpha_{1}}_\alpha (Q(\pa^{\al_1}\{g_1+\sqrt\mu g_2+(M-\mu)+\overline{G}\},  \partial^{\alpha-\al_1} g_1),  w_{2k-8|\alpha|}\partial^\alpha g_1 ).
	$$ We directly start from \eqref{I6I7} for $\al_1\leq\al$, $|\al_1|\geq1$, and estimate $I_6$ and $I_7$ for $|\al|=N$. If $1\leq|\al_1|\leq |\al|-1=N-1$, since the highest order of derivatives for both $\nabla_{x}\partial ^{\alpha _1} g_1$ and $\nabla_{x}\partial ^{\al-\alpha _1} g_1$ will not exceed $N$, we use the same arguments as in \eqref{I6<1}, \eqref{I6=1} and \eqref{I6} to get
	\begin{align}\label{NI61}
		I_6\le C_k\frac{1}{\varepsilon^2}\frac{1}{\sqrt\varepsilon}\sqrt{\mathcal{E}_{N,k}(t)}
		\mathcal{D}_{N,k}(t).
	\end{align}
	If $\al_1=\al$, which implies $|\al_1|=N$, it holds that
	\begin{align}\label{NI6N}
		I_6&\le C_k\frac{1}{\varepsilon}\|\partial ^{\alpha _1} g_1\|_{L^2_xL^2_5}\| g_1\|_{L^\infty_xH^s_{k-4N+\gamma/2+2s}}\|\partial^\alpha g_1\|_{L^2_xH^s_{k-4N+\gamma/2}}\notag\\
		&\le C_k\frac{1}{\varepsilon}\|\partial ^{\alpha _1} g_1\|_{L^2_xL^2_5}\|\nabla_{x} g_1\|^\frac{1}{2}_{L^2_xH^s_{k-4N+\gamma/2}}\|\nabla^2_{x} g_1\|^\frac{1}{2}_{L^2_xH^s_{k-4N+\gamma/2}}\|\partial^\alpha g_1\|_{L^2_xH^s_{k-4|\alpha|+\gamma/2}}\notag\\
		&\le C_k\frac{1}{\varepsilon^2}\frac{1}{\sqrt\varepsilon}\sqrt{\mathcal{E}_{N,k}(t)}
		\mathcal{D}_{N,k}(t).
	\end{align}
	We also discuss $I_7$ for two cases. If $1\leq |\al_1|\leq N-1$, the argument in \eqref{I7} directly gives
	\begin{align}\label{NI71}
		I_7\le C_k\frac{1}{\varepsilon^2}\frac{1}{\sqrt\varepsilon}\sqrt{\mathcal{E}_{N,k}(t)}
		\mathcal{D}_{N,k}(t).
	\end{align}
	If $\al_1=\al$, then
	\begin{align}\label{NI7N}
		I_7&\leq C_k\frac{1}{\varepsilon}\|\partial ^{\alpha _1} g_1\|_{L^2_xL^2_{14}}\| g_1\|_{L^\infty_xH^s_{k-4N+\gamma/2}}\|\partial^\alpha g_1\|_{L^2_xH^s_{k-4N+\gamma/2}}\notag\\
		&\quad+C_k\frac{1}{\varepsilon}\|\partial ^{\alpha _1} g_1\|_{L^2_xH^s_{\gamma/2}}\|w_{k-4N} g_1\|_{L^\infty_xL^2_{14}}\|\partial^\alpha g_1\|_{L^2_xH^s_{k-4N+\gamma/2}}\notag\\
		&\leq C_k\frac{1}{\varepsilon}\|\partial ^{\alpha _1} g_1\|_{L^2_xL^2_{14}}\|\nabla_{x} g_1\|_{L^2_xH^s_{k-4N+\gamma/2}}^\frac{1}{2}\|\nabla^2_{x} g_1\|^\frac{1}{2}_{L^2_xH^s_{k-4N+\gamma/2}}\|\partial^\alpha g_1\|_{L^2_xH^s_{k-4N+\gamma/2}}\notag\\
		&\quad+C_k\frac{1}{\varepsilon}\|\partial ^{\alpha _1} g_1\|_{L^2_xH^s_{\gamma/2}}\|w_{k-4N}\nabla_{x} g_1\|^\frac{1}{2}_{L^2_xL^2_{14}}\|w_{k-4N}\nabla_{x}^2 g_1\|^\frac{1}{2}_{L^2_xL^2_{14}}\|\partial^\alpha g_1\|_{L^2_xH^s_{k-4N+\gamma/2}}\notag\\
		&\le C_k\frac{1}{\varepsilon^2}\frac{1}{\sqrt\varepsilon}\sqrt{\mathcal{E}_{N,k}(t)}
		\mathcal{D}_{N,k}(t).
	\end{align}
	It follows from \eqref{NI61}, \eqref{NI6N}, \eqref{NI71} and \eqref{NI7N} that
	\begin{align}\label{I6I7N}
		&\frac{1}{\varepsilon}(Q(\pa^{\al_1}g_1,  \partial^{\alpha-\al_1} g_1),  w_{2k-8|\alpha|}\partial^\alpha g_1 )\le C_k\frac{1}{\varepsilon^2}\frac{1}{\sqrt\varepsilon}\sqrt{\mathcal{E}_{N,k}(t)}
		\mathcal{D}_{N,k}(t).
	\end{align}
	Then from similar calculations as above, together with the arguments as in \eqref{g2g1g11}, \eqref{I8}, \eqref{I91}, \eqref{I92} and \eqref{I9}, we obtain
	\begin{align}\label{I8I9N}
		&\frac{1}{\varepsilon}(Q(\pa^{\al_1}\sqrt\mu g_2,  \partial^{\alpha-\al_1} g_1),  w_{2k-8|\alpha|}\partial^\alpha g_1 )\le C_k\frac{1}{\varepsilon^2}\frac{1}{\sqrt\varepsilon}\sqrt{\mathcal{E}_{N,k}(t)}
		\mathcal{D}_{N,k}(t).
	\end{align}
	Likewise,
	\begin{align}\label{Nrestg1g1}
		\frac{1}{\varepsilon}(Q(\pa^{\al_1}\{(M-\mu)+\overline{G}\},  \partial^{\alpha-\al_1} g_1),  w_{2k-8|\alpha|}\partial^\alpha g_1 )
		\leq C_k\frac{\eta}{\varepsilon^2}\mathcal{D}_{N,k}(t)+C_k\frac{1}{\varepsilon^2}\frac{1}{\sqrt\varepsilon}\sqrt{\mathcal{E}_{N,k}(t)}
		\mathcal{D}_{N,k}(t).
	\end{align}
	We collect \eqref{I6I7N}, \eqref{I8I9N} and \eqref{Nrestg1g1} to get
	\begin{align}\label{estNnonlig1}
		&\varepsilon\sum_{|\al|=N,\alpha_{1}\leq \alpha, |\alpha_{1}|\geq1}C^{\alpha_{1}}_\alpha (Q(\pa^{\al_1}\{g_1+\sqrt\mu g_2+(M-\mu)+\overline{G}\},  \partial^{\alpha-\al_1} g_1),  w_{2k-8N}\partial^\alpha g_1 ) 
		\notag\\
		\le& C_k\eta\mathcal{D}_{N,k}(t)+C_k\frac{1}{\sqrt\varepsilon}\sqrt{\mathcal{E}_{N,k}(t)}
		\mathcal{D}_{N,k}(t).
	\end{align}
	In order to derive inequality like \eqref{estrestg1}, for $|\al|=N$ and $1\leq|\al_1|\leq N-1$, we use \eqref{g1g2g1} and \eqref{g2split} to obtain
	\begin{align}\label{Ng1g2g11}
		&\frac{1}{\varepsilon}(Q(\partial^{\al_1} g_1,\pa^{\alpha-\al_1}\{\sqrt{\mu}g_2+(M-\mu)+\overline{G}\}),  w_{2k-8|\alpha|}\partial^\alpha g_1 )\notag\\
		\leq& C_k\frac{\eta}{\varepsilon^2}\mathcal{D}_{N,k}(t)+C_k\frac{1}{\varepsilon^2}\frac{1}{\sqrt\varepsilon}\sqrt{\mathcal{E}_{N,k}(t)}
		\mathcal{D}_{N,k}(t).
	\end{align}
	For $|\al|=N$ and $\al_1=\al$, it holds that
	\begin{align}\label{Ng1g2g12}
		&\frac{1}{\varepsilon}(Q(\partial^{\al_1} g_1,\pa^{\alpha-\al_1}\{\sqrt{\mu}g_2+(M-\mu)+\overline{G}\}),  w_{2k-8|\alpha|}\partial^\alpha g_1 )\notag\\
		\leq& C_k\frac{1}{\varepsilon}\big( \Vert \partial^{\al_1}g_1 \Vert_{L^2_xL^2_{14}} \Vert \{\sqrt{\mu}g_2+(M-\mu)+\overline{G}\} \Vert_{L^\infty_xH^s_{k-4|\al|+\gamma/2+2s }}  \Vert\partial^\alpha g_1\Vert_{L^2_xH^s_{k-4|\al|+\gamma/2}}\notag\\
		&\qquad+  \Vert \{\sqrt{\mu}g_2+(M-\mu)+\overline{G}\} \Vert_{L^\infty_xL^2_{14}}  \Vert \partial^{\al_1}g_1 \Vert_{L^2_xH^s_{k-4|\al|+\gamma/2}}  \Vert\partial^\alpha g_1\Vert_{L^2_xH^s_{k-4|\al|+\gamma/2}}\big)\notag\\
		\leq& C_k\frac{\eta+\bar\eta}{\varepsilon^2}\mathcal{D}_{N,k}(t)+C_k\frac{1}{\varepsilon^2}\frac{1}{\sqrt\varepsilon}\sqrt{\mathcal{E}_{N,k}(t)}
		\mathcal{D}_{N,k}(t).
	\end{align}
	For $|\al|=N$ and $\al_1=0$, one has
	\begin{align}\label{Ng1g2g13}
		&\frac{1}{\varepsilon}(Q(\partial^{\al_1} g_1,\pa^{\alpha-\al_1}\{\sqrt{\mu}g_2+(M-\mu)+\overline{G}\}),  w_{2k-8|\alpha|}\partial^\alpha g_1 )\notag\\
		\leq& C_k \frac{1}{\varepsilon}\big(\Vert g_1 \Vert_{L^\infty_xL^2_{14}} \Vert \partial^{\alpha}\{\sqrt{\mu}g_2+(M-\mu)+\overline{G}\} \Vert_{L^2_xH^s_{k-4|\al|+\gamma/2+2s }}  \Vert\partial^\alpha g_1\Vert_{L^2_xH^s_{k-4|\al|+\gamma/2}}\notag\\
		&\qquad+  \Vert\partial^{\alpha} \{\sqrt{\mu}g_2+(M-\mu)+\overline{G}\} \Vert_{L^2_xL^2_{14}}  \Vert g_1 \Vert_{L^\infty_xH^s_{k-4|\al|+\gamma/2}}  \Vert\partial^\alpha g_1\Vert_{L^2_xH^s_{k-4|\al|+\gamma/2}}\big)\notag\\
		\leq& C_k\frac{\eta}{\varepsilon^2}\mathcal{D}_{N,k}(t)+C_k\frac{1}{\varepsilon^2}\frac{1}{\sqrt\varepsilon}\sqrt{\mathcal{E}_{N,k}(t)}
		\mathcal{D}_{N,k}(t).
	\end{align}
	The combination of \eqref{Ng1g2g11}, \eqref{Ng1g2g12} and \eqref{Ng1g2g13} gives
	\begin{align}\label{estNrestg1}
		&\varepsilon\sum_{|\al|=N}\big((\partial^{\alpha}Q(g_1,\sqrt{\mu}g_2+M-\mu+\overline{G}),w_{2k-8|\al|}\partial^{\alpha}g_1)\big)
		\notag\\
		\le&  C_k(\eta+\bar\eta)\mathcal{D}_{N,k}(t)+C_k\frac{1}{\sqrt\varepsilon}\sqrt{\mathcal{E}_{N,k}(t)}
		\mathcal{D}_{N,k}(t).
	\end{align}
	Hence, \eqref{Norderg1} follows from \eqref{LD2order},  \eqref{estNLD}, \eqref{estNnonlig1}, \eqref{estNrestg1} and \eqref{apriori}.
\end{proof}
For the highest order of fluid quantities and $g_2$, we derive the equation of $M+\overline{G}+\sqrt{\mu}g_2$, which avoids both the extra $\nabla_{x}$ and the exponential increase in $v$ of $g_1$. The difference of \eqref{1M+G} and \eqref{g1} gives
\begin{align}\label{Nequation}
	&\partial_{t}(\frac{M+\overline{G}+\sqrt{\mu}g_2}{\sqrt{\mu}})+v\cdot\nabla_{x}(\frac{M+\overline{G}+\sqrt{\mu}g_2}{\sqrt{\mu}})\notag\\
	=&\frac{1}{\varepsilon}Lg_2+\frac{1}{\varepsilon}\CL_Bg_1+\frac{1}{\varepsilon}\Gamma(\frac{M-\mu}{\sqrt{\mu}},g_2)+\frac{1}{\varepsilon}\Gamma(g_2,\frac{M-\mu}{\sqrt{\mu}})+\frac{1}{\varepsilon}\Gamma(\frac{\overline{G}+\sqrt{\mu}g_2}{\sqrt{\mu}},\frac{\overline{G}+\sqrt{\mu}g_2}{\sqrt{\mu}})
	\nonumber\\
	&+\frac{1}{\sqrt{\mu}}P_{1}\big\{v\cdot(\frac{|v-u|^{2}\nabla_{x}\bar{\theta}}{2R\theta^{2}}+\frac{(v-u)\cdot\nabla_{x}\bar{u}}{R\theta}) M\big\}.
\end{align}
Applying $\partial^{\alpha}$ to \eqref{Nequation} with $|\alpha|=N$ and taking the inner product with $\frac{\partial^{\alpha}(M+\overline{G}+\sqrt{\mu}g_2)}{\sqrt{\mu}}$ over $\mathbb{R}^{3}\times\mathbb{R}^{3}$, we obtain
\begin{align}\label{Ninnerproduct}
	&\frac{1}{2}\frac{d}{dt}\|\frac{\partial^{\alpha}(M+\overline{G}+\sqrt{\mu}g_2)}{\sqrt{\mu}}\|^{2}-\frac{1}{\varepsilon}(L\partial^{\alpha}g_2,\frac{\partial^{\alpha}(M+\overline{G}+\sqrt{\mu}g_2)}{\sqrt{\mu}})
	\notag\\=&\frac{1}{\varepsilon}(\CL_B\partial^{\alpha}g_1,\frac{\partial^{\alpha}(M+\overline{G}+\sqrt{\mu}g_2)}{\sqrt{\mu}})\notag\\
	&+\frac{1}{\varepsilon}(\partial^{\alpha}\Gamma(\frac{M-\mu}{\sqrt{\mu}},g_2)+\partial^{\alpha}\Gamma(g_2,\frac{M-\mu}{\sqrt{\mu}})+\partial^{\alpha}\Gamma(\frac{\overline{G}+\sqrt{\mu}g_2}{\sqrt{\mu}},\frac{\overline{G}+\sqrt{\mu}g_2}{\sqrt{\mu}}),\frac{\partial^{\alpha}(M+\overline{G}+\sqrt{\mu}g_2)}{\sqrt{\mu}})
	\nonumber\\
	&+(\frac{1}{\sqrt{\mu}}\partial^{\alpha}P_{1}\big\{v\cdot(\frac{|v-u|^{2}
		\nabla_x\overline{\theta}}{2R\theta^{2}}+\frac{(v-u)\cdot\nabla_x\bar{u}}{R\theta}) M\big\},\frac{\partial^{\alpha}(M+\overline{G}+\sqrt{\mu}g_2)}{\sqrt{\mu}}).
\end{align}
A natural concern is whether $\|\frac{\partial^{\alpha}(M+\overline{G}+\sqrt{\mu}g_2)}{\sqrt{\mu}}\|^{2}$ is comparable with the energy functional.
\begin{lemma}\label{leboundF}
	It holds that
	\begin{align}\label{boundF}
		&\|\frac{\partial^{\alpha}(M+\overline{G}+\sqrt{\mu}g_2)}{\sqrt{\mu}}\|^{2}\notag\\
		\geq& (c-C(\eta+\bar{\eta}+\eta_0^{1/2}))(\|\partial^{\alpha}(\widetilde{\rho},\widetilde{u},\widetilde{\theta})\|^{2}+\|\partial^{\alpha}g_2\|^{2})-C\|\partial^{\alpha}g_1\|^{2}-C(\eta+\bar{\eta}+\eta^{3/2}_0),
	\end{align}
	for any $\alpha$ with $|\alpha|=N$.
\end{lemma}
\begin{proof}
	For $|\alpha|=N$, 
	\begin{align*}
		&\|\frac{\partial^{\alpha}(M+\overline{G}+\sqrt{\mu}g_2)}{\sqrt{\mu}}\|^{2}=\int_{\mathbb{R}^{3}}\int_{\mathbb{R}^{3}}\frac{(\partial^{\alpha}M)^{2}+(\partial^{\alpha}(\overline{G}+\sqrt{\mu}g_2))^{2}+2\partial^{\alpha}(\overline{G}+\sqrt{\mu}g_2)\partial^{\alpha}M}{\mu} dvdx.
	\end{align*}
	Rewrite
	\begin{align}\label{reM2}
		\int_{\mathbb{R}^{3}}\int_{\mathbb{R}^{3}}\frac{(\partial^{\alpha}M)^{2}}{\mu} dvdx
		=\int_{\mathbb{R}^{3}}\int_{\mathbb{R}^{3}}\frac{(\partial^{\alpha}M)^{2}}{M}+(\frac{1}{\mu}-\frac{1}{M})(\partial^{\alpha}M)^{2} dvdx.
	\end{align}
	A direct calculation shows that for $\partial^{\alpha}=\partial^{\alpha'}\partial_{x_i}$,
	\begin{align}\label{defJ1J2}
		\partial^{\alpha}M=&M\big(\frac{\partial^{\alpha'}\partial_{x_i}\rho}{\rho}+\frac{(v-u)\cdot\partial^{\alpha'}\partial_{x_i}u}{R\theta}+(\frac{|v-u|^{2}}{2R\theta}-\frac{3}{2})\frac{\partial^{\alpha'}\partial_{x_i}\theta}{\theta} \big)
		\nonumber\\
		&+\sum_{\alpha_{1}\leq \alpha',|\al_1|\geq1}C^{\alpha_1}_{\alpha'}\Big(\partial^{\alpha_{1}}(M\frac{1}{\rho})\partial^{\alpha'-\alpha_{1}}\partial_{x_i}\rho+\partial^{\alpha_{1}}(M\frac{v-u}{R\theta})\cdot\partial^{\alpha'-\alpha_{1}}\partial_{x_i}u
		\nonumber\\
		&\hspace{2cm}+\partial^{\alpha_{1}}(M\frac{|v-u|^{2}}{2R\theta^{2}}-M\frac{3}{2\theta})\partial^{\alpha'-\alpha_{1}}\partial_{x_i}\theta\Big)
		\nonumber\\
		:=&J_1+J_2.
	\end{align}
	It follows from \eqref{basis}, \eqref{background} and \eqref{apriori} that
	\begin{align}\label{J12}
		\int_{\mathbb{R}^{3}}\int_{\mathbb{R}^{3}}\frac{(J_1)^{2}}{M}dvdx&=\int_{\mathbb{R}^{3}}\int_{\mathbb{R}^{3}}M\Big\{\frac{\partial^{\alpha}\rho}{\rho}+\frac{(v-u)\cdot\partial^{\alpha}u}{R\theta}+(\frac{|v-u|^{2}}{2R\theta}-\frac{3}{2})\frac{\partial^{\alpha}\theta}{\theta} \Big\}^{2}dvdx\notag
		\\
		&=\int_{\mathbb{R}^{3}}\int_{\mathbb{R}^{3}}M
		\Big\{(\frac{\partial^{\alpha}\rho}{\rho})^{2}+(\frac{(v-u)\cdot\partial^{\alpha}u}{R\theta})^{2}+((\frac{|v-u|^{2}}{2R\theta}-\frac{3}{2})\frac{\partial^{\alpha}\theta}{\theta})^{2} \Big\}dvdx\notag\\
		&\geq c\|\partial^{\alpha}(\widetilde{\rho},\widetilde{u},\widetilde{\theta})\|^{2}-C\eta.
	\end{align}
	Similarly,
	\begin{align}\label{J22}
		\int_{\mathbb{R}^{3}}\int_{\mathbb{R}^{3}}\frac{2J_1J_2+(J_2)^{2}}{M}dvdx\leq C(\eta+\eta^{3/2}_0).
	\end{align}
	Combining with \eqref{reM2}, \eqref{J12}, \eqref{J22} and the fact that
	\begin{align*}
		\int_{\mathbb{R}^{3}}\int_{\mathbb{R}^{3}}(\frac{1}{\mu}-\frac{1}{M})(\partial^{\alpha}M)^{2}dvdx\leq
		C(\eta+\bar{\eta})+C(\eta+\bar{\eta}+\eta_0^{1/2})\|\partial^{\alpha}(\widetilde{\rho},\widetilde{u},\widetilde{\theta})\|^{2},
	\end{align*}
	one has 
	\begin{align}\label{M2}
		\int_{\mathbb{R}^{3}}\int_{\mathbb{R}^{3}}\frac{(\partial^{\alpha}M)^{2}}{\mu}dvdx
		\geq (c-C(\eta+\bar{\eta}+\eta_0^{1/2}))\|\partial^{\alpha}(\widetilde{\rho},\widetilde{u},\widetilde{\theta})\|^{2}-C(\eta+\bar{\eta}+\eta^{3/2}_0).
	\end{align}
	Likewise, by \eqref{boundbarG}, \eqref{background} and \eqref{apriori},
	\begin{align}\label{G2}
		&\int_{\mathbb{R}^{3}}\int_{\mathbb{R}^{3}}\frac{(\partial^{\alpha}\overline{G}+\sqrt{\mu}\partial^{\alpha}g_2)^{2}}{\mu}dvdx\notag
		\\
		=&\int_{\mathbb{R}^{3}}\int_{\mathbb{R}^{3}}(\partial^{\alpha}g_2)^{2}dvdx
		+\int_{\mathbb{R}^{3}}\int_{\mathbb{R}^{3}}\frac{(\partial^{\alpha}\overline{G})^{2}
			+2\sqrt{\mu}\partial^{\alpha}g_2\partial^{\alpha}\overline{G}}{\mu}dvdx
		\geq \frac{1}{2}\|\partial^{\alpha}g_2\|^{2}-C\eta.
	\end{align}
	Again rewrite
	\begin{align}\label{reGM}
		&\int_{\mathbb{R}^{3}}\int_{\mathbb{R}^{3}}\frac{2\partial^{\alpha}(\overline{G}+\sqrt{\mu}g_2)\partial^{\alpha}M}{\mu}dvdx\notag\\
		=&	\int_{\mathbb{R}^{3}}\int_{\mathbb{R}^{3}}\frac{2\partial^{\alpha}(\overline{G}+\sqrt{\mu}g_2)\partial^{\alpha}M}{M}
		+(\frac{1}{\mu}-\frac{1}{M})2\partial^{\alpha}(\overline{G}+\sqrt{\mu}g_2)\partial^{\alpha}Mdvdx.
	\end{align}
	Noticing that from \eqref{g2macro} and the fact that $\overline{G}+\sqrt{\mu}\mathbf{P}_1g_2$ is microscopic, for any $0<\ka<1$, it holds that
	\begin{align}\label{GM1}
		&\int_{\mathbb{R}^{3}}\int_{\mathbb{R}^{3}}\frac{2\partial^{\alpha}(\overline{G}+\sqrt{\mu}g_2)J_1}{M}dvdx\notag\\
		=&\int_{\mathbb{R}^{3}}\int_{\mathbb{R}^{3}}2\partial^{\alpha}(\overline{G}+\sqrt{\mu}g_2)\Big(\frac{\partial^{\alpha}\rho}{\rho}+
		\frac{(v-u)\cdot\partial^{\alpha}u}{R\theta}+(\frac{|v-u|^{2}}{2R\theta}-\frac{3}{2})\frac{\partial^{\alpha}\theta}{\theta} \Big)dvdx\notag\\
		=&\int_{\mathbb{R}^{3}}\int_{\mathbb{R}^{3}}2\partial^{\alpha}\sqrt{\mu}\mathbf{P}_0g_2\Big(\frac{\partial^{\alpha}\rho}{\rho}+
		\frac{(v-u)\cdot\partial^{\alpha}u}{R\theta}+(\frac{|v-u|^{2}}{2R\theta}-\frac{3}{2})\frac{\partial^{\alpha}\theta}{\theta} \Big)dvdx\notag\\
		\leq&\ka\|\partial^{\alpha}(\widetilde{\rho},\widetilde{u},\widetilde{\theta})\|^{2}+C_\ka \eta+C_\ka \|\partial^{\alpha}g_1\|^2.
	\end{align}
	For the rest terms, similar calculation gives
	\begin{align}\label{GM2}
		&\int_{\mathbb{R}^{3}}\int_{\mathbb{R}^{3}}\frac{2\partial^{\alpha}(\overline{G}+\sqrt{\mu}g_2)J_2}{M}+(\frac{1}{\mu}-\frac{1}{M})2\partial^{\alpha}(\overline{G}+\sqrt{\mu}g_2)\partial^{\alpha}Mdvdx\notag\\
		\leq& C(\eta+\bar{\eta}+\eta^{3/2}_0)+C(\eta+\bar{\eta}+\eta_0^{1/2})(\|\partial^{\alpha}g_2\|^2+\|\partial^{\alpha}(\widetilde{\rho},\widetilde{u},\widetilde{\theta})\|^{2}).
	\end{align}
	Therefore, \eqref{boundF} holds by \eqref{M2}, \eqref{G2}, \eqref{reGM}, \eqref{GM1}, \eqref{GM2} and choosing $\ka$ to be small.
\end{proof}
Now we can consider other terms in \eqref{Ninnerproduct} to derive the main lemma for the energy estimate of $N-$order fluid quantities and $g_2$.

\begin{lemma}
	For any $0<\ka<1$, it holds that
	\begin{align}\label{Nfluid}
		&\varepsilon^{2}(c-C(\eta+\bar{\eta}+\eta_0^{1/2}))\sum_{|\alpha|=N}(\|\partial^{\alpha}(\widetilde{\rho},\widetilde{u},\widetilde{\theta})(t)\|^{2}+\|\partial^{\alpha}g_2(t)\|^{2})+c\varepsilon\int^t_0\sum_{|\alpha|=N}\|\partial^{\alpha}\mathbf{P}_1g_2(s)\|_{L^2_xL^2_{v,D}}^{2}ds
		\nonumber\\
		\leq&C\varepsilon^{2}\sum_{|\alpha|=N}(\|\partial^{\alpha}(\widetilde{\rho},\widetilde{u},\widetilde{\theta})(0)\|^{2}+\|\partial^{\alpha}g(0)\|^{2})+C\varepsilon^2\sum_{|\alpha|=N}\|\partial^{\alpha}g_1\|^{2}+C_{k,\ka}\varepsilon\int^t_0\|\pa^\al g_1(s)\|^2ds\notag\\
		&+C(\eta+\bar\eta+\eta_0^{1/2}+\ka)\int^t_0\mathcal{D}_{N,k}(s)ds+C(\eta+\bar{\eta}+\eta^{3/2}_0)\varepsilon^r+Ct\eta\varepsilon^r.
	\end{align}
\end{lemma}
\begin{proof}
	We start from the second term in \eqref{Ninnerproduct} since the first one is estimated in Lemma \ref{leboundF}. It is straightforward to get from \eqref{coercive} that
	\begin{align}\label{cog2g2}
		-\frac{1}{\varepsilon}(L\partial^{\alpha}g_2,\partial^{\alpha}g_2)\geq c\frac{1}{\varepsilon}\|\partial^{\alpha}\mathbf{P}_1g_2\|_{L^2_{v,D}}^{2}. 
	\end{align}
	For $(L\partial^{\alpha}g_2,\frac{\partial^{\alpha}M}{\sqrt{\mu}})$, we further decompose \eqref{defJ1J2} into
	\begin{align}\label{J1}
		J_1=&\mu\big(\frac{\partial^{\alpha}\rho}{\rho}+\frac{(v-u)\cdot\partial^{\alpha}u}{R\theta}+(\frac{|v-u|^{2}}{2R\theta}-\frac{3}{2})\frac{\partial^{\alpha}\theta}{\theta}\big)\notag
		\\
		&+(M-\mu)\big(\frac{\partial^{\alpha}\widetilde{\rho}}{\rho}
		+\frac{(v-u)\cdot\partial^{\alpha}\widetilde{u}}{R\theta}+(\frac{|v-u|^{2}}{2R\theta}
		-\frac{3}{2})\frac{\partial^{\alpha}\widetilde{\theta}}{\theta}\big)\notag
		\\
		&+(M-\mu)\big(\frac{\partial^{\alpha}\bar{\rho}}{\rho}
		+\frac{(v-u)\cdot\partial^{\alpha}\bar{u}}{R\theta}+(\frac{|v-u|^{2}}{2R\theta}
		-\frac{3}{2})\frac{\partial^{\alpha}\bar{\theta}}{\theta}\big)\notag\\
		=&J_{11}+J_{12}+J_{13},
	\end{align}
	and
	\begin{align}\label{J2}
		J_2=&\sum_{\alpha_{1}\leq \alpha',|\al_1|\geq1}C^{\alpha_1}_{\alpha'}\Big(\partial^{\alpha_{1}}(M\frac{1}{\rho})\partial^{\alpha'-\alpha_{1}}\partial_{x_i}\widetilde{\rho}+\partial^{\alpha_{1}}(M\frac{v-u}{R\theta})\cdot\partial^{\alpha'-\alpha_{1}}\partial_{x_i}\widetilde{u}\notag\\
		&\qquad\qquad\qquad\qquad\qquad\qquad\qquad\qquad+\partial^{\alpha_{1}}(M\frac{|v-u|^{2}}{2R\theta^{2}}-M\frac{3}{2\theta})\partial^{\alpha'-\alpha_{1}}\partial_{x_i}\widetilde{\theta}\Big)\notag\\
		&+\sum_{\alpha_{1}\leq \alpha',|\al_1|\geq1}C^{\alpha_1}_{\alpha'}\Big(\partial^{\alpha_{1}}(M\frac{1}{\rho})\partial^{\alpha'-\alpha_{1}}\partial_{x_i}\bar{\rho}+\partial^{\alpha_{1}}(M\frac{v-u}{R\theta})\cdot\partial^{\alpha'-\alpha_{1}}\partial_{x_i}\bar{u}\notag\\
		&\qquad\qquad\qquad\qquad\qquad\qquad\qquad\qquad+\partial^{\alpha_{1}}(M\frac{|v-u|^{2}}{2R\theta^{2}}-M\frac{3}{2\theta})\partial^{\alpha'-\alpha_{1}}\partial_{x_i}\bar{\theta}\Big)\notag\\
		=&J_{21}+J_{22}.
	\end{align}
	Then $\partial^{\alpha}M=J_{11}+J_{12}+J_{13}+J_{21}+J_{22}.$ It is straight forward to see 
	\begin{align}\label{LJ11}(Lg_2,\frac{J_{11}}{\sqrt{\mu}})=0\end{align} by the fact that $\frac{J_{11}}{\sqrt{\mu}}\in\ker{L}$. Using \eqref{Trilinear} and \eqref{Trilinears}, we obtain
	\begin{align}\label{LJ12}
		\frac{1}{\varepsilon}|(L\partial^{\alpha}g_2,\frac{J_{12}}{\sqrt{\mu}})|
		&\leq C(\eta+\bar\eta+\|(\widetilde{\rho},\widetilde{u},\widetilde{\theta})\|_{L^\infty_x})\frac{1}{\varepsilon}\|\partial^{\alpha}g_2\|_{L^2_xL^2_{v,D}}\|\partial^{\alpha}(\widetilde{\rho},\widetilde{u},\widetilde{\theta})\|\notag
		\\
		&\leq C\frac{1}{\varepsilon^2}(\eta+\bar\eta)\mathcal{D}_{N,k}(t)+C\frac{1}{\varepsilon^2}\sqrt{\mathcal{E}_{N,k}(t)}\mathcal{D}_{N,k}(t).
	\end{align}
	Similarly, let $\partial^{\alpha}=\partial^{\alpha'}\partial_{x_i}$ for some $\al'$ and $x_i$, then
	\begin{align}\label{LJ13}
		\frac{1}{\varepsilon}|(L\partial^{\alpha}g_2,\frac{J_{13}}{\sqrt{\mu}})|
		=&\frac{1}{\varepsilon}|(\partial^{\alpha'}Lg_2,\partial_{x_i}\big[\frac{M-\mu}{\sqrt{\mu}}\{\frac{\partial^{\alpha}\bar{\rho}}{\rho}+\frac{(v-u)\cdot\partial^{\alpha}\bar{u}}{R\theta}+(\frac{|v-u|^{2}}{2R\theta}-\frac{3}{2})\frac{\partial^{\alpha}\bar{\theta}}{\theta}\}\big])|
		\nonumber\\
		\leq& C\frac{1}{\varepsilon}\int_{\mathbb{R}^{3}}\|\partial^{\alpha'}g_2\|_{L^2_xL^2_{v,D}}
		\big\{(|\partial_{x_i}\partial^{\alpha}\bar{\rho}|+|\partial_{x_i}\partial^{\alpha}\bar{u}|+|\partial_{x_i}\partial^{\alpha}\bar{\theta}|)
		\nonumber\\
		&\hspace{1cm}+(|\partial^{\alpha}\bar{\rho}|+|\partial^{\alpha}\bar{u}|+|\partial^{\alpha}\bar{\theta}|)(|\partial_{x_i}\rho|+|\partial_{x_i}u|+|\partial_{x_i}\theta|)\big\}dx
		\nonumber\\
		\leq& C\eta\frac{1}{\varepsilon}\|\partial^{\alpha'}g_2\|_{L^2_xL^2_{v,D}}+C\frac{1}{\varepsilon^2}\eta\mathcal{D}_{N,k}(t)\nonumber\\
		\leq& C\eta\frac{1}{\varepsilon^2}\|\partial^{\alpha'}g_2\|^2_{L^2_xL^2_{v,D}}+C\eta+C\frac{1}{\varepsilon^2}\eta\mathcal{D}_{N,k}(t).
	\end{align}
	Also,
	\begin{align*}
		\frac{1}{\varepsilon}|(L\partial^{\alpha}g_2,\frac{J_{21}}{\sqrt{\mu}})|
		&\leq 
		C\sum_{\alpha_{1}\leq \alpha',|\al_1|\geq1}\frac{1}{\varepsilon}\int_{\mathbb{R}^{3}}
		\|\partial^{\alpha}g_2\|_{L^2_{v,D}}|\partial^{\alpha'-\alpha_{1}}\partial_{x_i}(\widetilde{\rho}, \widetilde{u},\widetilde{\theta})|\notag\\
		&\qquad\qquad\qquad\qquad\times
		(|\partial^{\alpha_{1}}(\rho, u,\theta)|+\cdots+|\nabla_{x}(\rho, u,\theta)|^{|\alpha_{1}|})
		dx.
	\end{align*}
	For $1\leq|\alpha_{1}|<|\alpha'|=N-1$, we have
	\begin{align*}
		\frac{1}{\varepsilon}|(L\partial^{\alpha}g_2,\frac{J_{21}}{\sqrt{\mu}})|&\leq \frac{C}{\varepsilon}\|\partial^{\alpha}g_2\|_{L^2_xL^2_{v,D}}\|\partial^{\alpha'-\alpha_{1}}\partial_{x_i}(\widetilde{\rho},\widetilde{u},\widetilde{\theta})\|\notag\\
		&\qquad\times\|(|\partial^{\alpha_{1}}(\rho, u,\theta)|+\cdots+|\nabla_{x}(\rho, u,\theta)|^{|\alpha_{1}|})\|_{L^{\infty}}  
		\\
		&
		\leq C\eta\frac{1}{\varepsilon}(\|\partial^{\alpha}g_2\|_{L^2_xL^2_{v,D}}^2+\varepsilon^{2})+C\frac{1}{\varepsilon^2}\sqrt{\mathcal{E}_{N,k}(t)}\mathcal{D}_{N,k}(t).
	\end{align*}
	For $|\alpha_{1}|=|\alpha'|=N-1$, it holds that
	\begin{align*}
		\frac{1}{\varepsilon}|(L\partial^{\alpha}g_2,\frac{J_{21}}{\sqrt{\mu}})|&\leq C\frac{1}{\varepsilon}\|\partial^{\alpha}g_2\|_{L^2_xL^2_{v,D}}\|\partial_{x_i}(\widetilde{\rho},\widetilde{u},\widetilde{\theta})\|_{L^{6}}\notag\\
		&\qquad\times\|(|\partial^{\alpha_{1}}(\rho, u,\theta)|+\cdot\cdot\cdot+|\nabla_{x}(\rho, u,\theta)|^{|\alpha_{1}|})\|_{L^{3}}  
		\\
		&\leq  C\eta\frac{1}{\varepsilon}(\|\partial^{\alpha}g_2\|_{L^2_xL^2_{v,D}}^2+\varepsilon^{2})+C\frac{1}{\varepsilon^2}\sqrt{\mathcal{E}_{N,k}(t)}\mathcal{D}_{N,k}(t).
	\end{align*}
	We have from the above two inequalities that
	\begin{align}\label{LJ21}
		\frac{1}{\varepsilon}|(L\partial^{\alpha}g_2,\frac{J_{21}}{\sqrt{\mu}})|&\leq C\eta\frac{1}{\varepsilon}(\|\partial^{\alpha}g_2\|_{L^2_xL^2_{v,D}}^2+\varepsilon^{2})+C\frac{1}{\varepsilon^2}\sqrt{\mathcal{E}_{N,k}(t)}\mathcal{D}_{N,k}(t).
	\end{align}
	A similar calculation as in \eqref{LJ13} gives
	\begin{align}\label{LJ22}
		\frac{1}{\varepsilon}|(L\partial^{\alpha}g_2,\frac{J_{22}}{\sqrt{\mu}})|&\leq C\eta\frac{1}{\varepsilon^2}\|\partial^{\alpha'}g_2\|^2_{L^2_xL^2_{v,D}}+C\eta+C\frac{1}{\varepsilon^2}\eta\mathcal{D}_{N,k}(t).
	\end{align}
	We combine \eqref{LJ11}, \eqref{LJ12}, \eqref{LJ13}, \eqref{LJ21} and \eqref{LJ22} to get
	\begin{align}\label{paalM}
		\frac{1}{\varepsilon}|(L\partial^{\alpha}g_2,\frac{\partial^{\alpha}M}{\sqrt{\mu}})|
		\leq C\eta+C\frac{1}{\varepsilon^2}(\eta+\bar\eta)\mathcal{D}_{N,k}(t)+C\frac{1}{\varepsilon^2}\sqrt{\mathcal{E}_{N,k}(t)}\mathcal{D}_{N,k}(t).
	\end{align}
	Using \eqref{boundbarG}, \eqref{background} and \eqref{apriori}, one further has
	\begin{align}\label{paalG}
		\frac{1}{\varepsilon}|(\mathcal{L}\partial^{\alpha}g_2,\frac{\partial^{\alpha}\overline{G}}{\sqrt{\mu}})|
		&\leq C\frac{1}{\varepsilon}\|\partial^{\alpha}g_2\|_{L^2_xL^2_{v,D}}
		\|\frac{\partial^{\alpha}\overline{G}}{\sqrt{\mu}}\|_{L^2_xL^2_{v,D}}\leq \eta\frac{1}{\varepsilon}\|\partial^{\alpha}g_2\|_{L^2_xL^2_{v,D}}^2
		+C\eta\notag
		\\
		&\leq \eta\frac{1}{\varepsilon^2}\mathcal{D}_{N,k}(t)
		+C\eta.
	\end{align}
	Therefore, by choosing $\eta$ to be small, it follows from \eqref{cog2g2}, \eqref{paalM} and \eqref{paalG} that
	\begin{align}\label{NLf}
		&-\frac{1}{\varepsilon}(L\partial^{\alpha}g_2,\frac{\partial^{\alpha}(M+\overline{G}+\sqrt{\mu}g_2)}{\sqrt{\mu}})\notag\\\geq& c\frac{1}{\varepsilon}\|\partial^{\alpha}\mathbf{P}_1g_2\|_{L^2_{v,D}}^{2}-C\eta-C\frac{1}{\varepsilon^2}(\eta+\bar\eta)\mathcal{D}_{N,k}(t)-C\frac{1}{\varepsilon^2}\sqrt{\mathcal{E}_{N,k}(t)}\mathcal{D}_{N,k}(t).
	\end{align}
	For the first term on the right hand side in \eqref{Ninnerproduct}, for any $0<\ka<1$, we apply Cauchy-Schwarz inequality, \eqref{boundbarG}, \eqref{background}, \eqref{DNk1} to get 
	\begin{align}\label{HighLb}
		&\frac{1}{\varepsilon}(\CL_B\partial^{\alpha}g_1,\frac{\partial^{\alpha}(M+\overline{G}+\sqrt{\mu}g_2)}{\sqrt{\mu}})\notag\\
		\leq& \frac{1}{\varepsilon}(\CL_B\partial^{\alpha}g_1,\frac{\partial^{\alpha}M}{\sqrt{\mu}}) +\ka\frac{1}{\varepsilon}\|w_{-2}\frac{\partial^{\alpha}(\overline{G}+\sqrt{\mu}g_2)}{\sqrt{\mu}}\|^2+C_{k,\ka}\frac{1}{\varepsilon}\|\pa^\al g_1\|^2\notag\\
		\leq& \frac{1}{\varepsilon}(\CL_B\partial^{\alpha}g_1,\frac{\partial^{\alpha}M}{\sqrt{\mu}})+C\ka\frac{1}{\varepsilon^2}\mathcal{D}_{N,k}(t)+C\eta+C_{k,\ka}\frac{1}{\varepsilon}\|\pa^\al g_1\|^2.
	\end{align}
	The term $\frac{1}{\varepsilon}(\CL_B\partial^{\alpha}g_1,\frac{\partial^{\alpha}M}{\sqrt{\mu}})$ can be estimated as in \eqref{paalM}, but the biggest difference is that \eqref{LJ11} fails to hold now. In terms of \eqref{J1} and \eqref{J2}, we still write
	\begin{align}\label{reHLb}
		\frac{1}{\varepsilon}(\CL_B\partial^{\alpha}g_1,\frac{\partial^{\alpha}M}{\sqrt{\mu}})=\frac{1}{\varepsilon}(\CL_B\partial^{\alpha}g_1,\frac{J_{11}+J_{12}+J_{13}+J_{21}+J_{22}}{\sqrt{\mu}}).
	\end{align}
	Similar arguments in \eqref{LJ12}, \eqref{LJ13}, \eqref{LJ21} and \eqref{LJ22} give
	\begin{align}\label{HLb1}
		&\frac{1}{\varepsilon}(\CL_B\partial^{\alpha}g_1,\frac{J_{12}+J_{13}+J_{21}+J_{22}}{\sqrt{\mu}})\notag\\
		\leq& C_k\frac{1}{\varepsilon}\|\partial^{\alpha}g_1\|^2+C\eta+C\frac{1}{\varepsilon^2}(\eta+\bar\eta)\mathcal{D}_{N,k}(t)+C\frac{1}{\varepsilon^2}\sqrt{\mathcal{E}_{N,k}(t)}\mathcal{D}_{N,k}(t).
	\end{align}
	Then for any $0<\ka<1$, denoting $\partial^{\alpha}=\partial^{\alpha'}\partial_{x_i}$ for some $\al'$ and $x_i$, we use Cauchy-Schwarz inequality and integration by parts to obtain
	\begin{align}\label{HLb2}
		\frac{1}{\varepsilon}(\CL_B\partial^{\alpha}g_1,\frac{J_{11}}{\sqrt{\mu}})&=\frac{1}{\varepsilon}(\CL_B\partial^{\alpha}g_1,\mu\big(\frac{\partial^{\alpha}\rho}{\rho}+\frac{(v-u)\cdot\partial^{\alpha}u}{R\theta}+(\frac{|v-u|^{2}}{2R\theta}-\frac{3}{2})\frac{\partial^{\alpha}\theta}{\theta}\big))\notag\\
		&\leq \ka\frac{1}{\varepsilon}\|\pa^\al(\widetilde{\rho},\widetilde{u},\widetilde{\theta})\|^2+C_{k,\ka}\frac{1}{\varepsilon}\|\partial^{\alpha}g_1\|^2\notag\\
		&\quad+ \frac{1}{\varepsilon}(\CL_B\partial^{\alpha'}g_1,\mu\pa_{x_i}\big(\frac{\partial^{\alpha}\bar\rho}{\rho}+\frac{(v-u)\cdot\partial^{\alpha}\bar u}{R\theta}+(\frac{|v-u|^{2}}{2R\theta}-\frac{3}{2})\frac{\partial^{\alpha}\bar\theta}{\theta}\big))\notag\\
		&\leq C_{k,\ka}\frac{1}{\varepsilon}\|\partial^{\alpha}g_1\|^2+C\eta+C\frac{\eta+\ka}{\varepsilon^2}\mathcal{D}_{N,k}(t).
	\end{align}
	We collect \eqref{HighLb}, \eqref{reHLb}, \eqref{HLb1} and \eqref{HLb2} to get
	\begin{align}\label{paalLb}
		&\frac{1}{\varepsilon}(\CL_B\partial^{\alpha}g_1,\frac{\partial^{\alpha}(M+\overline{G}+\sqrt{\mu}g_2)}{\sqrt{\mu}})\notag\\
		\leq&C(\ka+\eta+\bar\eta)\frac{1}{\varepsilon^2}\mathcal{D}_{N,k}(t)+C\eta+C_{k,\ka}\frac{1}{\varepsilon}\|\pa^\al g_1\|^2+C\frac{1}{\varepsilon^2}\sqrt{\mathcal{E}_{N,k}(t)}\mathcal{D}_{N,k}(t).
	\end{align}
	We turn to the second term on the right hand side in \eqref{Ninnerproduct}. Decompose
	\begin{align}\label{decomGa}
		\partial^{\alpha}\Gamma(\frac{M-\mu}{\sqrt{\mu}},g_2)=&
		\Gamma(\frac{M-\mu}{\sqrt{\mu}},\partial^{\alpha}g_2)
		+\sum_{\alpha_{1}\leq \alpha,|\al_1|\geq1}C^{\alpha_{1}}_{\alpha}\Gamma(\frac{\partial^{\alpha_{1}}(M-\mu)}{\sqrt{\mu}},
		\partial^{\alpha-\alpha_{1}}g_2).
	\end{align}
	Similar arguments in \eqref{LJ12}, \eqref{LJ13} \eqref{LJ21}, \eqref{LJ22} and \eqref{paalG} show
	\begin{align*}
		\frac{1}{\varepsilon}\big(\Gamma(\frac{M-\mu}{\sqrt{\mu}},\partial^{\alpha}g_2),\frac{\partial^{\alpha}M+\partial^{\alpha}\overline{G}}{\sqrt{\mu}}\big)
		=&\frac{1}{\varepsilon}\Big(\Gamma(\frac{M-\mu}{\sqrt{\mu}},\partial^{\alpha}g_2),\frac{J_1+J_2}{\sqrt{\mu}}
		\Big)+\frac{1}{\varepsilon}\big(\Gamma(\frac{M-\mu}{\sqrt{\mu}},\partial^{\alpha}g_2),\frac{\partial^{\alpha}\overline{G}}{\sqrt{\mu}}\big)\notag\\
		\leq& C\eta+C\frac{1}{\varepsilon^2}(\eta+\bar\eta)\mathcal{D}_{N,k}(t).
	\end{align*}
	Using \eqref{Trilinear}, \eqref{background} and \eqref{apriori}, we further have
	\begin{align*}
		\frac{1}{\varepsilon}(\Gamma(\frac{M-\mu}{\sqrt{\mu}},\partial^{\alpha}g_2),\partial^{\alpha}g_2)
		&\leq C\frac{1}{\varepsilon}\|\frac{M-\mu}{\sqrt{\mu}}\|_{L^\infty_xL^2_v}\|\partial^{\alpha}g_2\|_{L^2_xL^2_{v,D}}^2\notag\\
		&\leq C\frac{1}{\varepsilon^2}(\eta+\bar\eta)\mathcal{D}_{N,k}(t)+C\frac{1}{\varepsilon^2}\sqrt{\mathcal{E}_{N,k}(t)}
		\mathcal{D}_{N,k}(t).
	\end{align*}
	Combine the above two inequalities to get
	\begin{align}\label{highg2F}
		&\frac{1}{\varepsilon}(\Gamma(\frac{M-\mu}{\sqrt{\mu}},\partial^{\alpha}g_2),\frac{\partial^{\alpha}(M+\overline{G}+\sqrt{\mu}g_2)}{\sqrt{\mu}})
		\leq C\eta+C\frac{1}{\varepsilon^2}(\eta+\bar\eta)\mathcal{D}_{N,k}(t)+C\frac{1}{\varepsilon^2}\sqrt{\mathcal{E}_{N,k}(t)}
		\mathcal{D}_{N,k}(t).
	\end{align}
	For the second term on the right hand side of \eqref{decomGa}, again use \eqref{Trilinear} to obtain
	\begin{align*}
		&\frac{1}{\varepsilon}(\Gamma(\partial^{\alpha_{1}}(\frac{M-\mu}{\sqrt{\mu}}),
		\partial^{\alpha-\alpha_{1}}g_2),\frac{\partial^{\alpha}(M+\overline{G}+\sqrt{\mu}g_2)}{\sqrt{\mu}})\notag	
		\\
		\leq	
		&C\frac{1}{\varepsilon}\int_{\mathbb{R}^{3}}\|\partial^{\alpha_{1}}(\frac{M-\mu}{\sqrt{\mu}})\|_{L^2_v}
		\|\partial^{\alpha-\alpha_{1}}g_2\|_{L^2_{v,D}}\|\frac{\partial^{\alpha}(M+\overline{G}+\sqrt{\mu}g_2)}{\sqrt{\mu}}\|_{L^2_{v,D}}dx.
	\end{align*}
	From \eqref{defJ1J2}, \eqref{boundbarG}, \eqref{background} and \eqref{apriori}, one can see for any $l\in\R$,
	\begin{align}\label{NFdiss}
		\|w_l\frac{\partial^{\alpha}(M+\overline{G}+\sqrt{\mu}g_2)}{\sqrt{\mu}}\|_{L^2_xL^2_{v,D}}\leq C\eta+C\frac{1}{\sqrt{\varepsilon}}\sqrt{\mathcal{D}_{N,k}(t)}.
	\end{align}
	Then if $|\alpha_{1}|=1$, we use the Cauchy-Schwarz and Sobolev inequalities, \eqref{defJ1J2}, \eqref{background} and \eqref{apriori}  to get
	\begin{align*}	
		&\frac{1}{\varepsilon}\int_{\mathbb{R}^{3}}\|\partial^{\alpha_{1}}(\frac{M-\mu}{\sqrt{\mu}})\|_{L^2_v}
		\|\partial^{\alpha-\alpha_{1}}g_2\|_{L^2_{v,D}}\|\frac{\partial^{\alpha}(M+\overline{G}+\sqrt{\mu}g_2)}{\sqrt{\mu}}\|_{L^2_{v,D}}dx\notag\\\leq& C\frac{1}{\varepsilon}\|\partial^{\alpha_{1}}(\frac{M-\mu}{\sqrt{\mu}})\|_{L^\infty_xL^2_v}
		\|\partial^{\alpha-\alpha_{1}}g_2\|_{L^2_xL^2_{v,D}}\|\frac{\partial^{\alpha}(M+\overline{G}+\sqrt{\mu}g_2)}{\sqrt{\mu}}\|_{L^2_xL^2_{v,D}}\notag
		\\
		\leq& C\frac{1}{\varepsilon}\||(|\partial^{\alpha_1}(\rho,u,\theta)|+\cdot\cdot\cdot+|\nabla_x(\rho,u,\theta)|^{|\alpha_{1}|})|\|_{L^{\infty}_x}\sqrt{\varepsilon}\sqrt{\mathcal{D}_{N,k}(t)}\Big(\eta+\frac{1}{\sqrt{\varepsilon}}\sqrt{\mathcal{D}_{N,k}(t)}\Big)\notag
		\\
		\leq& C\frac{1}{\varepsilon}\Big(\eta+\min\{\frac{1}{\sqrt\varepsilon}\sqrt{\mathcal{E}_{N,k}(t)},\frac{1}{\sqrt{\varepsilon}}\sqrt{\mathcal{D}_{N,k}(t)}\}\Big)\sqrt{\varepsilon}\sqrt{\mathcal{D}_{N,k}(t)}\Big(\eta+\frac{1}{\sqrt{\varepsilon}}\sqrt{\mathcal{D}_{N,k}(t)}\Big)\notag
		\\
		\leq& 
		C\eta+C\frac{1}{\varepsilon^2}\eta\mathcal{D}_{N,k}(t)+C\frac{1}{\varepsilon^2}\frac{1}{\sqrt\varepsilon}\sqrt{\mathcal{E}_{N,k}(t)}\mathcal{D}_{N,k}(t).
	\end{align*}
	If $1<|\alpha_{1}|\leq N-1$, a similar calculation yields
	\begin{align*}	
		&\frac{1}{\varepsilon}\int_{\mathbb{R}^{3}}\|\partial^{\alpha_{1}}(\frac{M-\mu}{\sqrt{\mu}})\|_{L^2_v}
		\|\partial^{\alpha-\alpha_{1}}g_2\|_{L^2_{v,D}}\|\frac{\partial^{\alpha}(M+\overline{G}+\sqrt{\mu}g_2)}{\sqrt{\mu}}\|_{L^2_{v,D}}dx\notag\\
		\leq& C\frac{1}{\varepsilon}\|\partial^{\alpha_{1}}(\frac{M-\mu}{\sqrt{\mu}})\|_{L^3_xL^2_v}
		\|\partial^{\alpha-\alpha_{1}}g_2\|_{L^6_xL^2_{v,D}}\|\frac{\partial^{\alpha}(M+\overline{G}+\sqrt{\mu}g_2)}{\sqrt{\mu}}\|_{L^2_xL^2_{v,D}}\notag
		\\
		\leq& 
		C\eta+C\frac{1}{\varepsilon^2}\eta\mathcal{D}_{N,k}(t)+C\frac{1}{\varepsilon^2}\frac{1}{\sqrt\varepsilon}\sqrt{\mathcal{E}_{N,k}(t)}\mathcal{D}_{N,k}(t).
	\end{align*}
	If $|\al_1|=N$, then
	\begin{align*}	
		&\frac{1}{\varepsilon}\int_{\mathbb{R}^{3}}\|\partial^{\alpha_{1}}(\frac{M-\mu}{\sqrt{\mu}})\|_{L^2_v}
		\|\partial^{\alpha-\alpha_{1}}g_2\|_{L^2_{v,D}}\|\frac{\partial^{\alpha}(M+\overline{G}+\sqrt{\mu}g_2)}{\sqrt{\mu}}\|_{L^2_{v,D}}dx\notag\\\leq& C\frac{1}{\varepsilon}\|\partial^{\alpha_{1}}(\frac{M-\mu}{\sqrt{\mu}})\|_{L^2_xL^2_v}
		\|g_2\|_{L^\infty_xL^2_{v,D}}\|\frac{\partial^{\alpha}(M+\overline{G}+\sqrt{\mu}g_2)}{\sqrt{\mu}}\|_{L^2_xL^2_{v,D}}\notag
		\\
		\leq& 
		C\eta+C\frac{1}{\varepsilon^2}\eta\mathcal{D}_{N,k}(t)+C\frac{1}{\varepsilon^2}\frac{1}{\sqrt\varepsilon}\sqrt{\mathcal{E}_{N,k}(t)}\mathcal{D}_{N,k}(t).
	\end{align*}
	Hence, the above three cases imply
	\begin{align}\label{lowg2F}
		&\frac{1}{\varepsilon}(\Gamma(\partial^{\alpha_{1}}(\frac{M-\mu}{\sqrt{\mu}}),
		\partial^{\alpha-\alpha_{1}}g_2),\frac{\partial^{\alpha}(M+\overline{G}+\sqrt{\mu}g_2)}{\sqrt{\mu}})\notag\\\leq&
		C\eta+C\frac{1}{\varepsilon^2}\eta\mathcal{D}_{N,k}(t)+C\frac{1}{\varepsilon^2}\frac{1}{\sqrt\varepsilon}\sqrt{\mathcal{E}_{N,k}(t)}\mathcal{D}_{N,k}(t).
	\end{align}
	It then follows from \eqref{decomGa}, \eqref{highg2F} and \eqref{lowg2F} that
	\begin{align}\label{Mg2F}
		&\frac{1}{\varepsilon}(\partial^{\alpha}\Gamma(\frac{M-\mu}{\sqrt{\mu}},g_2),\frac{\partial^{\alpha}(M+\overline{G}+\sqrt{\mu}g_2)}{\sqrt{\mu}})\leq  C\eta+C\frac{1}{\varepsilon^2}(\eta+\bar\eta)\mathcal{D}_{N,k}(t)+C\frac{1}{\varepsilon^2}\frac{1}{\sqrt\varepsilon}\sqrt{\mathcal{E}_{N,k}(t)}
		\mathcal{D}_{N,k}(t).
	\end{align}
	We also deduce from similar arguments, \eqref{Trilinears} and \eqref{boundbarG} that
	\begin{align}
		&\frac{1}{\varepsilon}(\partial^{\alpha}\Gamma(g_2,\frac{M-\mu}{\sqrt{\mu}}),\frac{\partial^{\alpha}(M+\overline{G}+\sqrt{\mu}g_2)}{\sqrt{\mu}})\leq C\eta+C\frac{(\eta+\bar\eta)}{\varepsilon^2}\mathcal{D}_{N,k}(t)+C\frac{1}{\varepsilon^2}\frac{1}{\sqrt\varepsilon}\sqrt{\mathcal{E}_{N,k}(t)}\mathcal{D}_{N,k}(t),\label{g2MF}\\
		&	\frac{1}{\varepsilon}(\partial^{\alpha}\Gamma(\frac{\overline{G}}{\sqrt{\mu}},g_2),\frac{\partial^{\alpha}(M+\overline{G}+\sqrt{\mu}g_2)}{\sqrt{\mu}})\leq C\eta+C\frac{(\eta+\bar\eta)}{\varepsilon^2}\mathcal{D}_{N,k}(t)+C\frac{1}{\varepsilon^2}\frac{1}{\sqrt\varepsilon}\sqrt{\mathcal{E}_{N,k}(t)}\mathcal{D}_{N,k}(t),\label{Gg2F}\\
		&\frac{1}{\varepsilon}(\partial^{\alpha}\Gamma(g_2,\frac{\overline{G}}{\sqrt{\mu}}),\frac{\partial^{\alpha}(M+\overline{G}+\sqrt{\mu}g_2)}{\sqrt{\mu}})\leq C\eta+C\frac{(\eta+\bar\eta)}{\varepsilon^2}\mathcal{D}_{N,k}(t)+C\frac{1}{\varepsilon^2}\frac{1}{\sqrt\varepsilon}\sqrt{\mathcal{E}_{N,k}(t)}\mathcal{D}_{N,k}(t),\label{g2GF}\\
		&\frac{1}{\varepsilon}(\partial^{\alpha}\Gamma(\frac{\overline{G}}{\sqrt{\mu}},\frac{\overline{G}}{\sqrt{\mu}}),\frac{\partial^{\alpha}(M+\overline{G}+\sqrt{\mu}g_2)}{\sqrt{\mu}})
		\leq C\eta,\label{GGF}
	\end{align}
	and
	\begin{align}
		\frac{1}{\varepsilon}(\partial^{\alpha}\Gamma(g_2,g_2),\partial^{\alpha}g_2)
		&\leq C\frac{1}{\varepsilon^2}\frac{1}{\sqrt\varepsilon}\sqrt{\mathcal{E}_{N,k}(t)}\mathcal{D}_{N,k}(t).\label{g2g2Gg2}
	\end{align}
	Using \eqref{Trilinear}, one gets
	\begin{align}\label{g2g2M1}
		\frac{1}{\varepsilon}(\partial^{\alpha}\Gamma(g_2,g_2),\frac{\partial^{\alpha}(M+\overline{G})}{\sqrt{\mu}})
		\leq& C\sum_{\alpha_{1}\leq \alpha}\frac{1}{\varepsilon}\int_{\R^3}\|\pa^{\al_1}g_2\|\|\pa^{\al-\al_1}g_2\|_{L^2_{v,D}}\|\frac{\partial^{\alpha}(M+\overline{G})}{\sqrt{\mu}}\|_{L^2_{v,D}}dx.
	\end{align}
	If $\al_1=\al$, \eqref{NFdiss} shows 
	\begin{align}\label{g2g2M2}
		&\frac{1}{\varepsilon}\int_{\R^3}\|\pa^{\al_1}g_2\|\|\pa^{\al-\al_1}g_2\|_{L^2_{v,D}}\|\frac{\partial^{\alpha}(M+\overline{G})}{\sqrt{\mu}}\|_{L^2_{v,D}}dx\notag\\
		\leq& C\frac{1}{\varepsilon}\|\pa^{\al}g_2\|\|g_2\|_{L^\infty_xL^2_{v,D}}\|\frac{\partial^{\alpha}(M+\overline{G})}{\sqrt{\mu}}\|_{L^2_xL^2_{v,D}}\notag\\
		\leq& C\frac{1}{\varepsilon}\|\pa^{\al}g_2\|\|\nabla_{x}g_2\|^\frac{1}{2}_{L^2_xL^2_{v,D}}\|\nabla^2_{x}g_2\|^\frac{1}{2}_{L^2_xL^2_{v,D}}\big(\eta+\frac{1}{\sqrt{\varepsilon}}\sqrt{\mathcal{D}_{N,k}(t)}\big)\notag\\
		\leq& C\frac{1}{\varepsilon}\eta(\varepsilon^2\|\pa^{\al}g_2\|^2+\frac{1}{\varepsilon^2}\|\nabla_{x}g_2\|_{L^2_xL^2_{v,D}}\|\nabla^2_{x}g_2\|_{L^2_xL^2_{v,D}})+ C\frac{1}{\varepsilon^2}\sqrt{\mathcal{E}_{N,k}(t)}\mathcal{D}_{N,k}(t)\notag\\
		\leq& C\eta+C\frac{1}{\varepsilon^2}\eta\mathcal{D}_{N,k}(t)+C\frac{1}{\varepsilon^2}\sqrt{\mathcal{E}_{N,k}(t)}\mathcal{D}_{N,k}(t).
	\end{align}
	If $\al_1=0$, it holds that
	\begin{align}\label{g2g2M3}
		&\frac{1}{\varepsilon}\int_{\R^3}\|\pa^{\al_1}g_2\|\|\pa^{\al-\al_1}g_2\|_{L^2_{v,D}}\|\frac{\partial^{\alpha}(M+\overline{G})}{\sqrt{\mu}}\|_{L^2_{v,D}}dx\notag\\
		\leq& C\frac{1}{\varepsilon}\|g_2\|_{L^\infty_xL^2_v}\|\pa^{\al}g_2\|_{L^2_xL^2_{v,D}}\|\frac{\partial^{\alpha}(M+\overline{G})}{\sqrt{\mu}}\|_{L^2_xL^2_{v,D}}\notag\\
		\leq& C\eta+C\frac{1}{\varepsilon^2}\eta\mathcal{D}_{N,k}(t)+C\frac{1}{\varepsilon^2}\sqrt{\mathcal{E}_{N,k}(t)}\mathcal{D}_{N,k}(t).
	\end{align}
	Similarly, for $1\leq|\al_1|<N$,
	\begin{align}\label{g2g2M4}
		&\frac{1}{\varepsilon}\int_{\R^3}\|\pa^{\al_1}g_2\|\|\pa^{\al-\al_1}g_2\|_{L^2_{v,D}}\|\frac{\partial^{\alpha}(M+\overline{G})}{\sqrt{\mu}}\|_{L^2_{v,D}}dx\notag\\
		\leq& C\frac{1}{\varepsilon}\|\pa^{\al_1}g_2\|_{L^3_xL^2_v}\|\pa^{\al-\al_1}g_2\|_{L^6_xL^2_{v,D}}\|\frac{\partial^{\alpha}(M+\overline{G})}{\sqrt{\mu}}\|_{L^2_xL^2_{v,D}}\notag\\
		\leq& C\eta+C\frac{1}{\varepsilon^2}\eta\mathcal{D}_{N,k}(t)+C\frac{1}{\varepsilon^2}\sqrt{\mathcal{E}_{N,k}(t)}\mathcal{D}_{N,k}(t).
	\end{align}
	We combine \eqref{g2g2M1}, \eqref{g2g2M2}, \eqref{g2g2M3} and \eqref{g2g2M4} to obtain
	\begin{align}\label{g2g2M}
		\frac{1}{\varepsilon}(\partial^{\alpha}\Gamma(g_2,g_2),\frac{\partial^{\alpha}(M+\overline{G})}{\sqrt{\mu}})
		&\leq C\eta+C\frac{1}{\varepsilon^2}\eta\mathcal{D}_{N,k}(t)+C\frac{1}{\varepsilon^2}\sqrt{\mathcal{E}_{N,k}(t)}\mathcal{D}_{N,k}(t).
	\end{align}
	Then it follows from \eqref{Mg2F}, \eqref{g2MF}, \eqref{Gg2F}, \eqref{g2GF}, \eqref{GGF}, \eqref{g2g2Gg2} and \eqref{g2g2M} that
	\begin{align}\label{NGa}
		&\frac{1}{\varepsilon}(\partial^{\alpha}\Gamma(\frac{M-\mu}{\sqrt{\mu}},g_2)+\partial^{\alpha}\Gamma(g_2,\frac{M-\mu}{\sqrt{\mu}})+\partial^{\alpha}\Gamma(\frac{\overline{G}+\sqrt{\mu}g_2}{\sqrt{\mu}},\frac{\overline{G}+\sqrt{\mu}g_2}{\sqrt{\mu}}),\frac{\partial^{\alpha}(M+\overline{G}+\sqrt{\mu}g_2)}{\sqrt{\mu}})\notag\\
		\leq& C\eta+C\frac{1}{\varepsilon^2}(\eta+\bar\eta)\mathcal{D}_{N,k}(t)+C\frac{1}{\varepsilon^2}\frac{1}{\sqrt\varepsilon}\sqrt{\mathcal{E}_{N,k}(t)}\mathcal{D}_{N,k}(t).
	\end{align}
	Using \eqref{NFdiss}, the last term of \eqref{Ninnerproduct} is bounded by
	\begin{align}\label{NP1}
		&(\frac{1}{\sqrt{\mu}}\partial^{\alpha}P_{1}\big\{v\cdot(\frac{|v-u|^{2}
			\nabla_x\overline{\theta}}{2R\theta^{2}}+\frac{(v-u)\cdot\nabla_x\bar{u}}{R\theta}) M\big\},\frac{\partial^{\alpha}(M+\overline{G}+\sqrt{\mu}g_2)}{\sqrt{\mu}})\notag\\
		\leq&C\|\frac{1}{\sqrt{\mu}}\partial^{\alpha}P_{1}\big\{v\cdot(\frac{|v-u|^{2}\nabla_x\overline{\theta}}{2R\theta^{2}}+\frac{(v-u)\cdot\nabla_x\bar{u}}{R\theta}) M\big\}\|^{2}+C\|\frac{\partial^{\alpha}(M+\overline{G}+\sqrt\mu g_2)}{\sqrt{\mu}}\|^2\notag\\
		\leq& C\eta+C\frac{1}{\varepsilon}\mathcal{D}_{N,k}(t).
	\end{align}
	Hence, noticing that $g_1(0)=g(0)$, $g_2(0)=0$, and
		\begin{align*}
		\|\frac{\partial^{\alpha}(M+\overline{G}+\sqrt{\mu}g_2)}{\sqrt{\mu}}(0)\|^{2}
		\leq& C(\|\partial^{\alpha}(\widetilde{\rho},\widetilde{u},\widetilde{\theta})(0)\|^{2}+\|\partial^{\alpha}g(0)\|^{2})+C(\eta+\bar{\eta}+\eta^{3/2}_0),
	\end{align*}
	 which can be deduced from the proof of Lemma \ref{leboundF}, then \eqref{Nfluid} follows from \eqref{Ninnerproduct}, \eqref{boundF}, \eqref{NLf}, \eqref{paalLb}, \eqref{NGa}, \eqref{NP1} and our apriori assumption \eqref{apriori}.
\end{proof}

\section{Proof of main theorems}\label{sec.9}
We prove Theorem \ref{thm1.1} and Theorem \ref{thm1.2} in this section. Using the results obtained in previous sections, we first focus on Theorem \ref{thm1.1}.
\begin{proof}[Proof of Theorem \ref{thm1.1}]
	The linear combination of \eqref{zerofluid} and \eqref{0orderg1} gives
	\begin{align*}
		&\|(\widetilde{\rho},\widetilde{u},\widetilde{\theta},w_kg_1)(t)\|^{2}+\varepsilon\int^{t}_0
		\|\nabla_x(\widetilde{\rho},\widetilde{u},\widetilde{\theta})(s)\|^{2}ds+\frac{1}{\varepsilon}\int^{t}_0\Vert g_1(s) \Vert_{L^2_xH^s_{k+\ga/2}}^2ds
		\nonumber\\
		\leq&C\|(\widetilde{\rho},\widetilde{u},\widetilde{\theta},w_kg)(0)\|^{2}+ C_k(\eta+\bar\eta+\varepsilon^{r/2})\int^t_0\mathcal{D}_{N,k}(s)ds
		+C_k\varepsilon^{1+r}+C t\eta\varepsilon^{r},
	\end{align*}
	which, with the help of \eqref{zerog2}, further yields the zero order energy estimate:
	\begin{align}\label{zeroorder}
		&\|(\widetilde{\rho},\widetilde{u},\widetilde{\theta},w_kg_1,g_2)(t)\|^{2}+\varepsilon\int^{t}_0
		\|\nabla_x(\widetilde{\rho},\widetilde{u},\widetilde{\theta})(s)\|^{2}ds+\frac{1}{\varepsilon}\int^{t}_0\{\Vert  g_1(s) \Vert_{L^2_xH^s_{k+\ga/2}}^2+\Vert\mathbf{P}_1 g_2(s) \Vert_{L^2_xL^2_{v,D}}^2\}ds
		\nonumber\\
		\leq&C\|(\widetilde{\rho},\widetilde{u},\widetilde{\theta},w_kg)(0)\|^{2}+ C_k(\eta+\bar\eta+\varepsilon^{r/2})\int^t_0\mathcal{D}_{N,k}(s)ds
		+C_k\varepsilon^{1+r}+C t\eta\varepsilon^{r}.
	\end{align}
	Likewise, it follows from \eqref{macroN-1} and \eqref{loworderg1} that
	\begin{align*}
		&\sum_{1\leq |\alpha|\leq N-1}\|\partial^{\alpha}(\widetilde{\rho},\widetilde{u},\widetilde{\theta},w_{k-4|\alpha |}\partial^\alpha g_1)(t)\|^2+\varepsilon\sum_{2\leq |\alpha|\leq N}\int^t_0\|\partial^{\alpha}(\widetilde{\rho},\widetilde{u},\widetilde{\theta})(s)\|^2ds\notag\\
		&\qquad\qquad+\frac{1}{\varepsilon}\sum_{2\leq |\alpha|\leq N}\int^t_0\Vert w_{k-4|\alpha |}\partial^\alpha g_1(s) \Vert_{L^2_xH^s_{\ga/2}}^2ds
		\nonumber\\
		\leq& C\sum_{1\leq |\alpha|\leq N-1}\|\partial^{\alpha}(\widetilde{\rho},\widetilde{u},\widetilde{\theta},w_{k-4|\alpha |}\partial^\alpha g)(0)\|^2+C\varepsilon^2\sum_{|\alpha|=N}
		\|\partial^{\alpha}\widetilde{\rho}(0)\|^2\notag\\
		&+C\varepsilon^2\sum_{|\alpha|=N}
		(\|\partial^{\alpha}\widetilde{\rho}(t)\|^2+
		\|\partial^{\alpha}g_2(t)\|^2)+(C\ka+C_{\ka,k}(\eta+\bar\eta+\eta^{1/2}_0+\varepsilon^{r/2}))\int^t_0\mathcal{D}_N(s) ds\notag\\
		&+C_{\ka,k}\int_0^t\min\{\eta_0^{1/2}\varepsilon^{r/2-1}\mathcal{D}_{N,k}(s),\eta^{3/2}_0\varepsilon^r\}ds+C_{\ka,k}((\eta+\varepsilon^{r/2})\varepsilon^{2}+t\eta\varepsilon^{r}),
	\end{align*}
	which, combined with \eqref{loworderg2}, gives
	\begin{align}\label{derivativeest}
		&\sum_{1\leq |\alpha|\leq N-1}\|\partial^{\alpha}(\widetilde{\rho},\widetilde{u},\widetilde{\theta},w_{k-4|\alpha |}\partial^\alpha g_1,g_2)(t)\|^2+\varepsilon\sum_{2\leq |\alpha|\leq N}\int^t_0\|\partial^{\alpha}(\widetilde{\rho},\widetilde{u},\widetilde{\theta})(s)\|^2ds\notag\\
		&\qquad\qquad+\frac{1}{\varepsilon}\sum_{2\leq |\alpha|\leq N}\int^t_0\{\Vert w_{k-4|\alpha |}\partial^\alpha g_1(s) \Vert_{L^2_xH^s_{\ga/2}}^2+\Vert\partial^\alpha \mathbf{P}_1g_2(s) \Vert_{L^2_xL^2_{v,D}}^2\}ds
		\nonumber\\
		\leq&C_k\sum_{1\leq |\alpha|\leq N-1}\|\partial^{\alpha}(\widetilde{\rho},\widetilde{u},\widetilde{\theta},w_{k-4|\alpha |}\partial^\alpha g)(0)\|^2+C_k\varepsilon^2\sum_{|\alpha|=N}
		\|\partial^{\alpha}\widetilde{\rho}(0)\|^2\notag\\
		&+ C_{\ka,k}\varepsilon^2\sum_{|\alpha|=N}
		(\|\partial^{\alpha}\widetilde{\rho}(t)\|^2+\|w_{k-4N}\partial^{\alpha}\nabla_{x}g_1(t)\|_{L^2_xH^s_{\ga/2}}^{2}+\|\pa^\al\nabla_{x}\mathbf{P}_1g_2\|^2_{L^2_xL^2_{v,D}})\notag\\
		&+(C\ka+C_{\ka,k}(\eta+\bar\eta+\eta^{1/2}_0+\varepsilon^{r/2}))\int^t_0\mathcal{D}_N(s) ds\notag\\
		&+C_{\ka,k}\int_0^t\min\{\eta_0^{1/2}\varepsilon^{r/2-1}\mathcal{D}_{N,k}(s),\eta^{3/2}_0\varepsilon^r\}ds+C_{\ka,k}((\eta+\varepsilon^{r/2})\varepsilon^{r}+t\eta\varepsilon^{r}).
	\end{align}
	For the highest order estimate, it holds from \eqref{Norderg1}, \eqref{Nfluid} and our a priori assumption \eqref{apriori} that
	\begin{align}\label{Nest}
		&\varepsilon^{2}\sum_{|\alpha|=N}((c-C(\eta+\bar{\eta}+\eta_0^{1/2}))(\|\partial^{\alpha}(\widetilde{\rho},\widetilde{u},\widetilde{\theta})(t)\|^{2}+\|\partial^{\alpha}g_2(t)\|^{2})+\Vert w_{k-4N}\partial^\alpha g_1 \Vert^2)\notag\\
		&\qquad
		+\varepsilon\sum_{|\alpha|=N}\int^t_0\{\Vert w_{k-4N}\partial^\alpha g_1 \Vert_{L^2_xH^s_{\ga/2}}^2+\|\partial^{\alpha}\mathbf{P}_1g_2(s)\|_{L^2_xL^2_{v,D}}^{2}\} ds
		\nonumber\\
		\leq&C_{k,\ka}\varepsilon^{2}\sum_{|\alpha|=N}(\|\partial^{\alpha}(\widetilde{\rho},\widetilde{u},\widetilde{\theta})(0)\|^{2}+\{C_{k,\ka}(\eta+\bar\eta+\eta^{1/2}_0)+C\ka\}\int^t_0\mathcal{D}_{N,k}(s)ds+C(\eta+\bar{\eta}+\eta^{3/2}_0)\varepsilon^r+Ct\eta\varepsilon^r.
	\end{align}
	Letting $\bar{\eta},\eta,\eta_0$ to be so small that $c-C(\eta+\bar{\eta}+\eta_0^{1/2})>c/2$, taking suitable linear combination of \eqref{zeroorder}, \eqref{derivativeest} and \eqref{Nest}, we obtain
	\begin{align*}
		&\mathcal{E}_{N,k}(t)+\int^{t}_{0}\mathcal{D}_{N,k}(s) ds\notag\\
		\leq&C_{k,\ka}\mathcal{E}_{N,k}(0)+(C\ka+C_{\ka,k}(\eta+\bar\eta+\eta^{1/2}_0))\int^t_0\mathcal{D}_N(s) ds+C_{\ka,k}\int_0^t\min\{\eta_0^{1/2}\varepsilon^{r/2-1}\mathcal{D}_{N,k}(s),\eta^{3/2}_0\varepsilon^r\}ds\notag\\&+C_{\ka,k}((\eta+\bar{\eta}+\eta^{3/2}_0+\varepsilon^{r/2})\varepsilon^{r}+t\eta\varepsilon^{r}).
	\end{align*}
	We first fix a small $\ka$, then choose $\bar\eta$, $\eta$, $\eta_0$ and $\varepsilon$ to be small and use $t\leq T$ to get
	\begin{align}\label{energy}
		&\mathcal{E}_{N,k}(t)+\frac{1}{2}\int^{t}_{0}\mathcal{D}_{N,k}(s) ds\notag\\
		\leq&C_{k}\mathcal{E}_{N,k}(0)+ C_k\int_0^t\min\{\eta_0^{1/2}\varepsilon^{r/2-1}\mathcal{D}_{N,k}(s),\eta^{3/2}_0\varepsilon^r\}ds+C_k((\eta+\bar{\eta}+\eta^{3/2}_0+\varepsilon^{r/2})\varepsilon^{r}+T\eta\varepsilon^{r}).
	\end{align}
	It holds from the above inequality that
	\begin{align*}
		\mathcal{E}_{N,k}(t)+\frac{1}{2}\int^{t}_{0}\mathcal{D}_{N,k}(s) ds
		\leq&C_{k}\mathcal{E}_{N,k}(0)+ C_k((\eta+\bar{\eta}+\eta^{3/2}_0+\varepsilon^{r/2})\varepsilon^{r}+T(\eta+\eta^{3/2}_0)\varepsilon^{r}).
	\end{align*}
	Hence, from our a priori assumption \eqref{apriori}, it is direct to see that \eqref{energyestimate} holds by letting
	$$
	C_1=\frac{1}{8C_k}
	$$
	for the constant $C_1$ in \eqref{initial}, and choosing $\bar\eta$, $\eta$, $\eta_0$ and $\varepsilon$ to be small such that
	$$C_k\left((\eta+\bar{\eta}+\eta^{3/2}_0+\varepsilon^{r/2})+T(\eta+\eta^{3/2}_0)\right)\leq\frac{1}{8}.$$
Therefore, we have proven \eqref{energyestimate}.

In the end, we show that \eqref{energyestimate} implies \eqref{estsolution}. Notice $F-M_{[\bar{\rho},\bar{u},\bar{\theta}]}=M-M_{[\bar{\rho},\bar{u},\bar{\theta}]}+\overline{G}+g_1+\sqrt{\mu}g_2$, which yields for $|\al|\leq N-1$,
	\begin{align*}
		&\|w_{k-4|\al|}\pa^\al\{F(t,x,v)-M_{[\bar{\rho},\bar{u},\bar{\theta}](t,x)}(v)\}\|^2_{L^2_xL_{v}^{2}}\notag\\
		\leq& \|w_{k-4|\al|}\pa^\al\{M_{[\rho,u,\theta](t,x)}(v)-M_{[\bar{\rho},\bar{u},\bar{\theta}](t,x)}(v)\}\|^2_{L^2_xL_{v}^{2}}+\|w_{k-4|\al|}\pa^\al\{\overline{G}+g_1+\sqrt{\mu}g_2\}\|^2_{L^2_xL_{v}^{2}}.
	\end{align*}
The second term on the right hand side of the above inequality can be bounded by $C_T\varepsilon^r$ by Lemma \ref{lebarG}, \eqref{ENk} and \eqref{apriori}. If $|\al|=0$, it holds by Taylor expansion that
$$
\|w_{k-4|\al|}\pa^\al\{M_{[\rho,u,\theta](t,x)}(v)-M_{[\bar{\rho},\bar{u},\bar{\theta}](t,x)}(v)\}\|^2_{L^2_xL_{v}^{2}}\leq C_T \|(\widetilde{\rho},\widetilde{u},\widetilde{\theta})(t)\|^2\leq C_T\varepsilon^r.
$$
If $1\leq|\al|\leq N-1$, we apply $\pa^\al$ to $M-\overline{M}$ to get 
\begin{align*}
	&\pa^\al\{M_{[\rho,u,\theta](t,x)}(v)-M_{[\bar{\rho},\bar{u},\bar{\theta}](t,x)}(v)\}\notag\\ &=M\big(\frac{\partial^{\alpha'}\partial_{x_i}\rho}{\rho}+\frac{(v-u)\cdot\partial^{\alpha'}\partial_{x_i}u}{R\theta}+(\frac{|v-u|^{2}}{2R\theta}-\frac{3}{2})\frac{\partial^{\alpha'}\partial_{x_i}\theta}{\theta} \big)
	\nonumber\\
	&\quad -\overline{M}\big(\frac{\partial^{\alpha'}\partial_{x_i}\bar{\rho}}{\bar{\rho}}+\frac{(v-\bar{u})\cdot\partial^{\alpha'}\partial_{x_i}\bar{u}}{R\bar{\theta}}+(\frac{|v-\bar{u}|^{2}}{2R\bar{\theta}}-\frac{3}{2})\frac{\partial^{\alpha'}\partial_{x_i}\bar{\theta}}{\bar{\theta}} \big)
	\nonumber\\
	&+\sum_{\alpha_{1}\leq \alpha',|\al_1|\geq1}C^{\alpha_1}_{\alpha'}\{\partial^{\alpha_{1}}(M\frac{1}{\rho})\partial^{\alpha'-\alpha_{1}}\partial_{x_i}\rho+\partial^{\alpha_{1}}(M\frac{v-u}{R\theta})\cdot\partial^{\alpha'-\alpha_{1}}\partial_{x_i}u+\partial^{\alpha_{1}}(M(\frac{|v-u|^{2}}{2R\theta^{2}}-\frac{3}{2\theta}))\partial^{\alpha'-\alpha_{1}}\partial_{x_i}\theta\notag\\
	&\qquad\qquad\qquad\quad-\partial^{\alpha_{1}}(\overline{M}\frac{1}{\bar{\rho}})\partial^{\alpha'-\alpha_{1}}\partial_{x_i}\bar{\rho}-\partial^{\alpha_{1}}(\overline{M}\frac{v-\bar{u}}{R\bar{\theta}})\cdot\partial^{\alpha'-\alpha_{1}}\partial_{x_i}\bar{u}-\partial^{\alpha_{1}}(\overline{M}(\frac{|v-\bar{u}|^{2}}{2R\bar{\theta}^{2}}-\frac{3}{2\bar{\theta}}))\partial^{\alpha'-\alpha_{1}}\partial_{x_i}\bar{\theta}\}.
\end{align*}
One can see that using the triangular inequality, \eqref{apriori} and our zero order estimate, the first difference on the right hand side above can be bounded by $C_T\varepsilon^r$. Then the second difference involving all lower order derivatives also enjoys the same upper bound by triangular inequality and induction.
The case $|\al|=N$ can be handled in the similar way. Hence, \eqref{estsolution} holds once we have \eqref{energyestimate}.
\end{proof}

As a direct application of the estimate \eqref{energy}, we prove Theorem \ref{thm1.2}.

\begin{proof}[Proof of Theorem \ref{thm1.2}]
	Let  $(\bar{\rho},\bar{u},\bar{\theta})(t,x)\equiv (1,0,\frac{3}{2}+\bar\eta)$ be a constant state as the trivial solution to the compressible Euler system \eqref{euler}. In view of \eqref{defeta}, we let $\eta=0$, to obtain from \eqref{energy} that
	\begin{align*}
		\mathcal{E}_{N,k}(t)+\frac{1}{2}\int^{t}_{0}\mathcal{D}_{N,k}(s) ds
		\leq C_{k}\mathcal{E}_{N,k}(0)+ C_k\eta_0^{1/2}\varepsilon^{r/2-1}\int_0^t\mathcal{D}_{N,k}(s)ds+C_k(\bar{\eta}+\eta^{3/2}_0+\varepsilon^{r/2})\varepsilon^{r}.
	\end{align*}
	Therefore, we still let 
		$$
	C_1=\frac{1}{4C_k},
	$$
	then \eqref{globalenergy} holds by setting $r=2$, and choosing $\bar\eta,\eta_0$ and $\varepsilon$ to be small such that
$$
C_k\eta_0^{1/2}\varepsilon^{r/2-1}=C_k\eta_0^{1/2}\leq \frac{1}{8},\quad C_k(\bar{\eta}+\eta^{3/2}_0+\varepsilon^{r/2})\leq \frac{1}{8},
$$
which proves \eqref{globalenergy}. Similarly for showing \eqref{estsolution}, one can get \eqref{estglobal} from \eqref{globalenergy}, which then finishes the proof of Theorem \ref{thm1.2}.
\end{proof}

\medskip
\noindent {\bf Acknowledgment:}\,
The research of Renjun Duan was partially supported by the General Research Fund (Project No.~14301822) from RGC of Hong Kong and the Direct Grant (Project No.~4053652) from CUHK.

\medskip

\noindent{\bf Conflict of Interest:} The authors declare that they have no conflict of interest.


\end{document}